\documentclass[reqno,11pt]{amsart}
\usepackage{epsfig,amscd,amssymb,amsmath,amsfonts}

\usepackage{amsmath}
\usepackage{graphicx}
\usepackage{amsthm,color}
\usepackage{tikz}
\usepackage{float}
\usepackage{afterpage}
\usetikzlibrary{graphs}
\usetikzlibrary{graphs,quotes}
\usetikzlibrary{decorations.pathmorphing}
\usepackage{longtable}

\tikzset{snake it/.style={decorate, decoration=snake}}
\tikzset{snake it/.style={decorate, decoration=snake}}

\usetikzlibrary{decorations.pathreplacing,decorations.markings,snakes}

\newtheorem{theorem}{Theorem}[section]
\newtheorem{lemma}[theorem]{Lemma}
\newtheorem{proposition}{Proposition}[section]
\theoremstyle{definition}
\newtheorem{definition}[theorem]{Definition}
\newtheorem{corollary}[theorem]{Corollary}
\newtheorem{conjecture}[theorem]{Conjecture}

\newtheorem{example}[theorem]{Example}

\theoremstyle{remark}
\newtheorem{remark}[theorem]{Remark}

\numberwithin{equation}{section}

\usepackage[margin=1.05in]{geometry}
\usepackage[colorlinks]{hyperref}
\usepackage{siunitx}

\usetikzlibrary{graphs}
\usetikzlibrary{graphs,quotes}
\usetikzlibrary{decorations.pathmorphing}

\tikzset{snake it/.style={decorate, decoration=snake}}
\tikzset{snake it/.style={decorate, decoration=snake}}

\makeatletter
\let\@wraptoccontribs\wraptoccontribs
\makeatother

\begin{document}

\title[Twin-star hypothesis and cycle-free $d$-partitions of $K_{2d}$ ]{Twin-star hypothesis and cycle-free $d$-partitions of $K_{2d}$  }


\author{Matthew J. Fyfe }
\address{Department of Mathematics and Statistics, Bowling Green State University, Bowling Green, OH 43403 }
\email{mjfyfe@bgsu.edu}


\author{Steven R. Lippold}
\address{Department of Chemistry, Mathematics, and Physics, Geneva College, Beaver Falls, PA 15010}
\email{srlippol@geneva.edu}

%    Information for third author

\author{Patrick Nyadjo  Fonga }
\address{Department of Mathematics and Statistics, Bowling Green State University, Bowling Green, OH 43403 }
\email{nyadjop@bgsu.edu}

%    Information for fourth author

\author{Mihai D. Staic}
\address{Department of Mathematics and Statistics, Bowling Green State University, Bowling Green, OH 43403 }
\address{Institute of Mathematics of the Romanian Academy, PO.BOX 1-764, RO-70700 Bu\-cha\-rest, Romania.}

%\curraddr{}
\email{mstaic@bgsu.edu}

%	Information for fifth author

\author{Alin Stancu}
\address{Department of Mathematics, Columbus State  University, Columbus, GA 31907}
\email{stancu\_alin1@columbusstate.edu}

%    General info
%\date{January 1, 1994 and, in revised form, June 22, 1994.}


\begin{abstract} In this paper we study an equivalence relation defined  on the set of cycle-free $d$-partitions of the complete graph $K_{2d}$. We discuss a conjecture which states that this equivalence relation has only one equivalence class, and show that the conjecture is equivalent with the so called twin-star hypothesis. We check the conjecture in the case $d=4$ and discuss how this relates to the determinant-like map $det^{S^2}$.
\end{abstract}

\maketitle

%\section*{This is an unnumbered first-level section head}

%
%%%%%%%%%%%%%%%%%%%%%%%%%%%%%%%%%%%%%%%%%%%%%%%%%%%

%%%%%%%%%%%%%%%%%%%%%%%%%%%%%%%%%%%%%%%%%%%%%%%%%%%%
\section{Introduction}

Edge partitions of graphs have been intensively studied in recent years.  Most of these studies deal with partitions of the complete graph, and concentrate on some specific property of the sub graphs that form the partition (for examples see \cite{bhr,brus,cc,fro,jpk,km}). In this paper we are interested in partitions of the complete graph $K_{2d}$ such that each graph of the partition  is a spanning-tree (i.e. cycle-free $d$-partition of $K_{2d}$). While specific examples of such partitions are well known (see \cite{kp,ku}), it seems that the properties of the set $\mathcal{P}^{cf}_d(K_{2d})$ (of all cycle-free $d$-partitions of the complete graph $K_{2d}$) have not been studied.

It is well known that if $\Gamma$ is a subgraph of the complete graph $K_n$ with $|E(\Gamma)|\geq n$ then $\Gamma$ has a cycle. Using this fact, one can show that if $\Gamma_1,\dots ,\Gamma_d$ are cycle-free subgraphs of $K_{n}$ such that the set of edges $E(\Gamma_1),\dots,E(\Gamma_d)$ form a partition for $E(K_n)$,  then $d\geq \lfloor \frac{n+1}{2} \rfloor$. In particular, once we fix $d$ the largest $n$ for which such a partition  exists is $n=2d$. Moreover, in this limit case, if $\Gamma_1,\dots, \Gamma_d$ is a cycle-free $d$-partition for $K_{2d}$ then $|E(\Gamma_i)|=2d-1$ for each $1\leq i\leq d$, and so each $\Gamma_i$ is a spanning-tree for $K_{2d}$. It turns out that the set of cycle-free $d$-partitions of the complete graph $K_{2d}$ is naturally endowed with a collection of involutions, that play an important role in the study of a certain determinant-like map $det^{S^2}:V_d^{\otimes d(2d-1)}\to k$  from \cite{edge, dets2} (see  Subsection \ref{subsection21} for some details).

%More precisely, for a vector space $V_d$ of dimension $d$, the classical determinant can be viewed as the unique (up to a scalar) nontrivial linear map $det: V_d^{\otimes d}\rightarrow k$, such that $det(\otimes_{ 1\leq i\leq d} (v_i)=0)$, if there exist $1\leq x < y \leq d$ such that $v_x = v_y$. It was shown in \cite{dets2} that there exist a nontrivial linear map $det^{S^2}:V_d^{\otimes d(2d-1)}\to k$ with the property that $det^{S^2}(\otimes_{1\leq i<j\leq 2d}(v_{i,j}))=0$, if there exist $1\leq x<y<z\leq 2d$ such that $v_{x,y}=v_{x,z}=v_{y,z}$. This result gives a partial answer to a conjecture from  \cite{sta2} where it was proposed that such a nontrivial map exists and is unique (up to a scalar). The uniqueness in the case $d=2$ was checked in \cite{sta2} by brute force.

%In an attempt to establish the uniqueness of $det^{S^2}$, for any value of $d$, cycle-free $d$-partitions of the complete graph $K_{2d}$ were introduced in \cite{edge}.

More precisely, using some linear algebra considerations related to the map $det^{S^2}$, in \cite{edge} it was shown there that for any cycle-free $d$-partition $(\Gamma_1,\dots,\Gamma_d)$ of $K_{2d}$ and any three vertices $x$, $y$ and $z$, there exists a unique cycle-free $d$-partition $(\Gamma_1,\dots,\Gamma_d)^{(x,y,z)}$ which coincides with the initial partition everywhere except on at least two of the edges  $(x,y)$, $(x,z)$, and $(y,z)$. Moreover, this  defines an involution $(x,y,z)$ on the set of cycle-free $d$-partitions $\mathcal{P}^{cf}_d(K_{2d})$. % and allows us to restate the uniqueness of the map $det^{S^2}$ in combinatorial terms.
It follows from  \cite{edge} that the map $det^{S^2}$ is unique (up to a scalar) if and only if the action induced by the involutions $(x,y,z)$ on the set of cycle-free $d$-partitions $\mathcal{P}^{cf}_d(K_{2d})$ is transitive. This was shown to be the case for $d=3$ in \cite{edge}, while the general case is still an open question.

Even thought they have a very natural combinatorial definition, as far as we know, the involutions $(x,y,z)$ on $\mathcal{P}^{cf}_d(K_{2d})$ have not been studied before their introduction in \cite{edge}.

In this paper we translate to combinatorial language some of the linear algebra results from \cite{edge}. This makes the uniqueness problem easier to state and allows us to reduce it to a simpler question, the so called  twin-star hypothesis.  Using this reduction we check the conjecture in the case $d=4$.
In Section  \ref{section1} we recall some definitions and results from \cite{edge}. We define the cycle-free $d$-partitions of the complete graph $K_{2d}$ and introduce the twin-star graph $TS_d$ and the graph $I_d$, which will play an important role in the paper. We also recall Lemma $\ref{keylemma}$, which is the main ingredient in defining the involution mentioned above. We explain how these partitions arise naturally in relation to the determinant-like map $det^{S^2}$.

In Section \ref{section2} we define the notions of involution equivalence and weak equivalence on the set $\mathcal{P}^{cf}_d(K_{2d})$ of cycle-free $d$-partitions of the complete graph $K_{2d}$, and discuss some of their properties. From \cite{edge} and \cite{dets2} it follows that the uniqueness (up to a scalar) of the determinant map $det^{S^2}$ is equivalent to Conjecture \ref{mainconj}, which states that any two cycle-free $d$-partitions of $K_{2d}$ are involution equivalent.

In Section \ref{section3} the twin-star hypothesis $TS(d)$ is introduced. A cycle-free $d$-partition satisfies the hypothesis if it is weakly equivalent to one which has at least one graph isomorphic to $TS_d$. We note here that this hypothesis is true for $d=1, 2,$ and $3$, since Conjecture \ref{mainconj} is true (cf. \cite{sta2} and \cite{edge}). The lemmas in this section are in support of establishing  one of main results of the paper, Theorem \ref{mainth}, which asserts that Conjecture \ref{mainconj} is true if the twin-star hypothesis $TS(i)$ is satisfied for $1\leq i\leq d$. We conclude this section with Theorem \ref{theI2d}, which shows that any cycle-free partition of $K_{2d}$ is involution equivalent to one that contains a graph of the form $I_{2d}$.

Section \ref{section4} is dedicated entirely to proving Conjecture \ref{mainconj} for $d=4$.  We analyze the isomorphism types of trees with $8$ vertices, and  show that every cycle-free $4$-partition $(\Gamma_1,\Gamma_2,\Gamma_3,\Gamma_4)$ of the complete graph $K_8$ is weakly  equivalent with a cycle-free $4$-partition  $(\Delta_1,\Delta_2,\Delta_3,\Delta_4)$ such that $\Delta_4$ is the  twin star graph $TS_4$. This shows that the twin-star hypothesis $\mathcal{TS}(4)$ holds true. Since the twin-star hypothesis is also true for $d=1, 2,$ and  $3$, it follows from Theorem \ref{mainth} that Conjecture \ref{mainconj} is true for $d=4$.

Finally, Appendix \ref{appendix1} contains a proof of Lemma \ref{lemma3}, which we postponed for later, to streamline the presentation. Appendix \ref{appendix2} consists of the list of the $23$ isomorphism types of trees with $8$ vertices, which is taken from \cite{h}, and is included for the convenience of the reader.

\section{Preliminaries}
\label{section1}

All graphs considered in this paper are simple (i.e. no loops and no multiple edges). For a graph $\Gamma$ we denote by $V(\Gamma)$ and $E(\Gamma)$ its set of vertices and edges respectively. Since our graphs are not oriented, we will not distinguish between the edge $(a,b)$ and $(b,a)\in E(\Gamma)$. For $s\geq 3$  we say that a graph $\Gamma$ has an
$s$-cycle if we can find a collection of distinct vertices  $v_1, v_2, \dots , v_s\in V(\Gamma)$ such that $(v_1,v_2)$, $(v_2,v_3)$, ..., $(v_{s-1},v_s)$, $(v_s,v_1)\in E(\Gamma)$.  We say that the graph $\Gamma$ is cycle-free if it does not have any cycles (i.e. it is a forest).
We denote by $K_{n}$  the complete graph with $n$ vertices $V(K_{n})=\{1,2,\dots,n\}$, and $E(K_n)=\{(i,j)|1\leq i<j\leq n\}$.

\subsection{Cycle-free $d$-partitions of $K_{2d}$}

Next we recall from \cite{edge} a few definitions and results about $d$-partitions of the complete graph $K_{2d}$.

\begin{definition} A $d$-partition of $K_{n}$ is an ordered collection $\mathcal{P}=(\Gamma_1,\Gamma_2,\dots,\Gamma_d)$ of sub-graphs $\Gamma_i$ of  $K_{n}$ such that
\begin{enumerate}
\item $\displaystyle{V(\Gamma_i)=V(K_{n})=\{1,2,\dots,n\}}$ for all $1\leq i \leq d$,
\item $\displaystyle{E(\Gamma_i)\bigcap E(\Gamma_j)=\emptyset}$ for all $i\neq j$,
\item $\displaystyle{\bigcup_{i=1}^dE(\Gamma_i)=E(K_{n})}$.
\end{enumerate}
We say that the $d$-partition $(\Gamma_1,\Gamma_2,\dots,\Gamma_d)$ of $K_{n}$ is homogeneous if $|E(\Gamma_i)\vert=|E(\Gamma_j)\vert$ for all $1\leq i<j\leq d$. We say that the $d$-partition $(\Gamma_1,\Gamma_2,\dots,\Gamma_d)$  is cycle-free if each $\Gamma_i$ is cycle-free.
\end{definition}

\begin{example} Consider $(\Gamma_1,\Gamma_2,\Gamma_3)$ a $3$-partition of the complete graph $K_6$ determined by
$V(\Gamma_1)=V(\Gamma_2)=V(\Gamma_3)=\{1,2,3,4,5,6\}$, $E(\Gamma_1)=\{(1,2),(1,3),(1,5),(2,4),(2,6)\}$, $E(\Gamma_2)=\{(1,4),(2,3),(3,4),(3,5),(4,6)\}$, and $E(\Gamma_3)=\{(1,6),(2,5),(3,6),(4,5),(5,6)\}$.  Then $(\Gamma_1, \Gamma_2,\Gamma_3)$ is a homogeneous, cycle-free $3$-partition for $K_6$ (see  Figure \ref{fig1}).
\end{example}

\begin{figure}[h!]
	\centering
	\begin{tikzpicture}
		[scale=1.5,auto=left,every node/.style={shape = circle, draw, fill = white,minimum size = 1pt, inner sep=0.3pt}]%baseline=(a.center)]%{circle,fill=black}]
		%\tikzset{VertexStyle/.style = {shape = circle,fill = black,minimum size = 9mm,inner sep=2pt}}
		\node (n1) at (0,0) {1};
		\node (n2) at (0.5,0.85)  {2};
		\node (n3) at (1.5,0.85)  {3};
		\node (n4) at (2,0)  {4};
		\node (n5) at (1.5,-0.85)  {5};
		\node (n6) at (0.5,-0.85)  {6};
		%\node (front) at (-2,0.5)   {$\Gamma_1$}
		\foreach \from/\to in {n1/n2,n1/n3,n1/n5,n2/n4,n2/n6}
		\draw[line width=0.5mm,red]  (\from) -- (\to);	
		%\node[state,above of=B1] (C1) {$C_1$};
		\node (n11) at (3,0) {1};
		\node (n21) at (3.5,0.85)  {2};
		\node (n31) at (4.5,0.85)  {3};
		\node (n41) at (5,0)  {4};
		\node (n51) at (4.5,-0.85)  {5};
		\node (n61) at (3.5,-0.85)  {6};
		%\node (front) at (-2,0.5)   {$\Gamma_1$}   3     5     7    11    12
		\foreach \from/\to in {n11/n41,n21/n31,n31/n41,n31/n51,n41/n61}
		\draw[line width=0.5mm,orange]  (\from) -- (\to);	
		
		\node (n12) at (6,0) {1};
		\node (n22) at (6.5,0.85)  {2};
		\node (n32) at (7.5,0.85)  {3};
		\node (n42) at (8,0)  {4};
		\node (n52) at (7.5,-0.85)  {5};
		\node (n62) at (6.5,-0.85)  {6};
		%\node (front) at (-2,0.5)   {$\Gamma_1$} 2     4     8     9    14
		\foreach \from/\to in {n12/n62,n22/n52,n32/n62,n42/n52,n52/n62}
		\draw[line width=0.5mm,blue]  (\from) -- (\to);	
		
		%\node[state,above of=B1] (C1) {$C_1$};
	\end{tikzpicture}
	\caption{A homogeneous cycle-free $3$-partition of  $K_6$ } \label{fig1}
\end{figure}
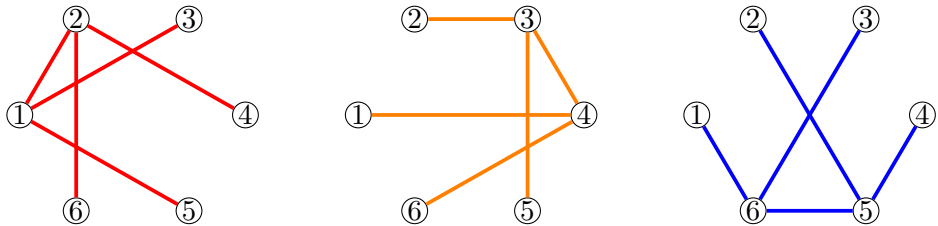

\begin{example} Consider $(\Delta_1, \Delta_2,\Delta_3)$  a $3$-partition of the complete graph $K_6$ determined by
 $V(\Delta_1)=V(\Delta_2)=V(\Delta_3)=\{1,2,3,4,5,6\}$, $E(\Delta_1)=\{(1,2),(1,4),(1,6),(2,4),(2,5)\}$, $E(\Delta_2)=\{(1,3),(2,3),(3,4),(3,6),(4,5)\}$, and $E(\Delta_3)=\{(1,5),(2,6),(3,5),(4,6),(5,6)\}$.  Then $(\Delta_1, \Delta_2,\Delta_3)$ is a homogeneous $3$-partition for $K_6$, but it is not cycle-free  since $\Delta_1$ has the cycle $(1,2,4)$ (see  Figure \ref{fig2}).
%\end{enumerate}
%For more example  of partition see \cite{edge}.
\end{example}

\begin{figure}[h!]
	\centering
	\begin{tikzpicture}
		[scale=1.5,auto=left,every node/.style={shape = circle, draw, fill = white,minimum size = 1pt, inner sep=0.3pt}]%baseline=(a.center)]%{circle,fill=black}]
		%\tikzset{VertexStyle/.style = {shape = circle,fill = black,minimum size = 9mm,inner sep=2pt}}
		\node (n1) at (0,0) {1};
		\node (n2) at (0.5,0.85)  {2};
		\node (n3) at (1.5,0.85)  {3};
		\node (n4) at (2,0)  {4};
		\node (n5) at (1.5,-0.85)  {5};
		\node (n6) at (0.5,-0.85)  {6};
		%\node (front) at (-2,0.5)   {$\Gamma_1$}
		\foreach \from/\to in {n1/n2,n1/n4,n1/n6,n2/n4,n2/n5}
		\draw[line width=0.5mm,red]  (\from) -- (\to);	
		%\node[state,above of=B1] (C1) {$C_1$};
		\node (n11) at (3,0) {1};
		\node (n21) at (3.5,0.85)  {2};
		\node (n31) at (4.5,0.85)  {3};
		\node (n41) at (5,0)  {4};
		\node (n51) at (4.5,-0.85)  {5};
		\node (n61) at (3.5,-0.85)  {6};
		%\node (front) at (-2,0.5)   {$\Gamma_1$}   3     5     7    11    12
		\foreach \from/\to in {n11/n31,n21/n31,n31/n41,n31/n61,n41/n51}
		\draw[line width=0.5mm,orange]  (\from) -- (\to);	
		
		\node (n12) at (6,0) {1};
		\node (n22) at (6.5,0.85)  {2};
		\node (n32) at (7.5,0.85)  {3};
		\node (n42) at (8,0)  {4};
		\node (n52) at (7.5,-0.85)  {5};
		\node (n62) at (6.5,-0.85)  {6};
		%\node (front) at (-2,0.5)   {$\Gamma_1$} 2     4     8     9    14
		\foreach \from/\to in {n12/n52,n22/n62,n32/n52,n42/n62,n52/n62}
		\draw[line width=0.5mm,blue]  (\from) -- (\to);	
		
		%\node[state,above of=B1] (C1) {$C_1$};
	\end{tikzpicture}
	\caption{A homogeneous $3$-partition of  $K_6$ that is not cycle-free } \label{fig2}
\end{figure}
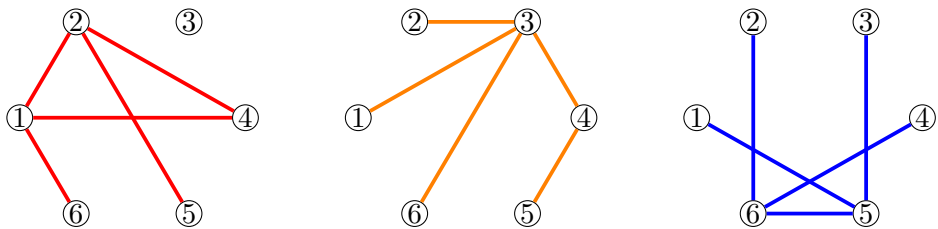

\begin{remark} One can easily see that if $(\Gamma_1,\dots,\Gamma_d)$ is a cycle-free $d$-partition of the complete graph $K_{2d}$ then each graph $\Gamma_i$ must be a spanning tree for $K_{2d}$ (and the other way around). So, one could use the terminology {\it spanning tree decomposition of $K_{2d}$} instead of {\it cycle-free $d$-partition of the complete graph $K_{2d}$}. In order to be consistent with the notation from \cite{edge}, we will use the cycle-free terminology.
\end{remark}

\begin{definition}
We denote by $\mathcal{P}_d^{cf}(K_{2d})$ the set of cycle-free $d$-partitions of the complete graph $K_{2d}$.
\end{definition}

\begin{remark} It is well know that a graph with $n$ vertices and at least $n$ edges must have a cycle.  The graph $K_{2d}$ has  $2d$ vertices and $d(2d-1)$ edges. If a $d$-partition $(\Gamma_1,\dots,\Gamma_d)$ of $K_{2d}$ is not homogeneous then one of the graphs $\Gamma_{i}$ will have at least $2d$ edges. This means that $\Gamma_{i}$ will have a cycle, and so the $d$-partition would not be cycle-free. In particular, this means that any cycle-free $d$-partition of $K_{2d}$ is automatically homogeneous.% While the term homogeneous is not necessary when we talk about cycle-free $d$-partitions of the complete graph $K_{2d}$  one should notice that it was used in  \cite{edge}.
\end{remark}

\begin{example} The twin-star graph $TS_d$ and the path graph $I_{2d}$ will play an important role in this paper, so, for convenience, we recall their presentation in Figure \ref{fig3} and Figure \ref{fig301} respectively. Notice that $TS_d$ is a particular case of the double-star graph $S_{x,y}$ where $x=y=d-1$ \cite{h}.
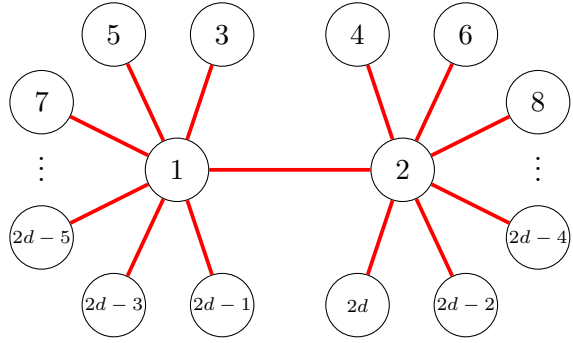
\begin{figure}[h]
\centering
\begin{tikzpicture}
  [scale=0.6,auto=left]%,every node/.style={shape = circle, draw, fill = white,minimum size = 14pt, inner sep=0.3pt}]%baseline=(a.center)]%{circle,fill=black}]
	%\tikzset{VertexStyle/.style = {shape = circle,fill = black,minimum size = 9mm,inner sep=2pt}}
	\node[shape=circle,draw=black,minimum size = 24pt,inner sep=0.3pt] (n1) at (3,0) {$1$};
	\node[shape=circle,draw=black,minimum size = 24pt,inner sep=0.3pt] (n2) at (8,0) {$2$};
	\node[shape=circle,draw=black,minimum size = 24pt,inner sep=0.3pt] (n3) at (7,3) {$4$};
  \node[shape=circle,draw=black,minimum size = 24pt,inner sep=0.3pt] (n4) at (4,3) {$3$};
	\node[shape=circle,draw=black,minimum size = 24pt,inner sep=0.3pt] (n5) at (9.4,3) {$6$};
  \node[shape=circle,draw=black,minimum size = 24pt,inner sep=0.3pt] (n6) at (1.6,3) {$5$};
	\node[shape=circle,draw=black,minimum size = 24pt,inner sep=0.3pt] (n51) at (11,1.5) {$8$};
  \node[shape=circle,draw=black,minimum size = 24pt,inner sep=0.3pt] (n61) at (0,1.5) {$7$};
	\node[shape=circle,draw=black,minimum size = 24pt,inner sep=0.3pt] (n71) at (11,-1.5) {{\tiny $2d-4$}};
  \node[shape=circle,draw=black,minimum size = 24pt,inner sep=0.3pt] (n81) at (0,-1.5) {{\tiny $2d-5$}};
	\node[shape=circle,draw=black,minimum size = 24pt,inner sep=0.3pt] (n7) at (9.4,-3) {{\tiny $2d-2$}};
  \node[shape=circle,draw=black,minimum size = 24pt,inner sep=0.3pt] (n8) at (1.6,-3) {{\tiny $2d-3$}};
	\node[shape=circle,draw=black,minimum size = 24pt,inner sep=0.3pt] (n9) at (7,-3) {{\tiny $2d$}};
  \node[shape=circle,draw=black,minimum size = 24pt,inner sep=0.3pt] (n10) at (4,-3) {{\tiny $2d-1$}};
	
	\node[shape=circle,minimum size = 24pt,inner sep=0.3pt] (m4) at (0,0.2) {{\large \vdots}};
	\node[shape=circle,minimum size = 24pt,inner sep=0.3pt] (m5) at (11,0.2) {{\large \vdots}};

	  \draw[line width=0.5mm,red]  (n1) -- (n2)  ;
		\draw[line width=0.5mm,red]  (n1) -- (n4)  ;
		\draw[line width=0.5mm,red]  (n1) -- (n6)  ;
		\draw[line width=0.5mm,red]  (n1) -- (n61)  ;
		\draw[line width=0.5mm,red]  (n1) -- (n8)  ;
		\draw[line width=0.5mm,red]  (n1) -- (n81)  ;
		\draw[line width=0.5mm,red]  (n1) -- (n10)  ;
		\draw[line width=0.5mm,red]  (n3) -- (n2)  ;
	  \draw[line width=0.5mm,red]  (n5) -- (n2)  ;
		\draw[line width=0.5mm,red]  (n51) -- (n2)  ;
	  \draw[line width=0.5mm,red]  (n7) -- (n2)  ;
	  \draw[line width=0.5mm,red]  (n71) -- (n2)  ;
		\draw[line width=0.5mm,red]  (n9) -- (n2)  ;
		%\draw[line width=0.5mm,dashed]  (n2) -- (n3);	
\end{tikzpicture}
\caption{$TS_d$ the Twin-Star graph  with $2d$ vertices} \label{fig3}
\end{figure}

\begin{figure}[h]
\centering
\begin{tikzpicture}
  [scale=0.6,auto=left]%,every node/.style={shape = circle, draw, fill = white,minimum size = 14pt, inner sep=0.3pt}]%baseline=(a.center)]%{circle,fill=black}]
	%\tikzset{VertexStyle/.style = {shape = circle,fill = black,minimum size = 9mm,inner sep=2pt}}
	
	\node[shape=circle,draw=black,minimum size = 24pt,inner sep=0.3pt] (n1) at (-1,0) {$1$};
	\node[shape=circle,draw=black,minimum size = 24pt,inner sep=0.3pt] (n2) at (2,0) {$2$};
	\node[shape=circle,draw=black,minimum size = 24pt,inner sep=0.3pt] (n3) at (5,0) {$3$};

		\node[shape=circle,minimum size = 24pt,inner sep=0.3pt] (m4) at (8,0) {{\large \dots}};

	\node[shape=circle,draw=black,minimum size = 24pt,inner sep=0.3pt] (n5) at (11,0) {{\tiny $2d-2$}};
	\node[shape=circle,draw=black,minimum size = 24pt,inner sep=0.3pt] (n6) at (14,0) {{\tiny $2d-1$}};
	\node[shape=circle,draw=black,minimum size = 24pt,inner sep=0.3pt] (n7) at (17,0) {{\tiny $2d$}};
  \node[] (n8) at (6.5,0)  {};
	\node[] (n9) at (9.5,0)  {};
	
	  \draw[line width=0.5mm,red]  (n1) -- (n2)  ;
		\draw[line width=0.5mm,red]  (n2) -- (n3)  ;
		\draw[line width=0.5mm,red]  (n3) -- (n8)  ;
		\draw[line width=0.5mm,red]  (n9) -- (n5)  ;
		\draw[line width=0.5mm,red]  (n5) -- (n6)  ;
		\draw[line width=0.5mm,red]  (n6) -- (n7)  ;
	
		%\draw[line width=0.5mm,dashed]  (n2) -- (n3);	
\end{tikzpicture}
\caption{$I_{2d}$ the path graph with $2d$ vertices} \label{fig301}
\end{figure}
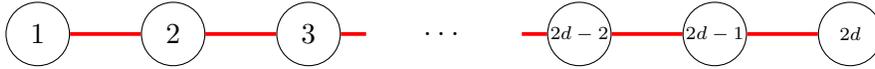

\end{example}

Next, we give an explicit example of a cycle-free $d$-partition of the complete graph $K_{2d}$ that will be used through the paper. In particular this shows that for all $d\geq 1$ the set $\mathcal{P}_d^{cf}(K_{2d})$ of cycle-free $d$-partition of $K_{2d}$ is not empty.

\begin{definition} For $1\leq a\leq d$ let $S_a=\{2a-1,2a\}$. For each $1\leq t\leq d$ we define the graphs $\Omega_t^{(d)}$ as follows.
\begin{enumerate}
\item $V(\Omega_t^{(d)})=\{1,2,\dots, 2d\}$ for all $1\leq t\leq d$.
\item If $1\leq i<j\leq 2d$, $i+j=0\; (mod\; 2)$  and $i\in S_a$ then  $(i,j)\in E(\Omega_a^{(d)})$.
\item If $1\leq i<j\leq 2d$, $i+j=1\; (mod\; 2)$ and $j\in S_b$ then  $(i,j)\in E(\Omega_b^{(d)})$.
\end{enumerate}
\end{definition}

\begin{remark}
The graph $\Omega_1^{(d)}$ is presented in Figure \ref{fig3}, while $\Omega_d^{(d)}$ is presented in Figure \ref{fig31}.
\begin{figure}[h]
\centering
\begin{tikzpicture}
  [scale=0.6,auto=left]%,every node/.style={shape = circle, draw, fill = white,minimum size = 14pt, inner sep=0.3pt}]%baseline=(a.center)]%{circle,fill=black}]
	%\tikzset{VertexStyle/.style = {shape = circle,fill = black,minimum size = 9mm,inner sep=2pt}}
	\node[shape=circle,draw=black,minimum size = 24pt,inner sep=0.3pt] (n1) at (3,0) {{\tiny $2d$}};
	\node[shape=circle,draw=black,minimum size = 24pt,inner sep=0.3pt] (n2) at (8,0) {{\tiny $2d-1$}};
	\node[shape=circle,draw=black,minimum size = 24pt,inner sep=0.3pt] (n3) at (7,3) {$2$};
  \node[shape=circle,draw=black,minimum size = 24pt,inner sep=0.3pt] (n4) at (4,3) {$1$};
	\node[shape=circle,draw=black,minimum size = 24pt,inner sep=0.3pt] (n5) at (9.4,3) {$4$};
  \node[shape=circle,draw=black,minimum size = 24pt,inner sep=0.3pt] (n6) at (1.6,3) {$3$};
	\node[shape=circle,draw=black,minimum size = 24pt,inner sep=0.3pt] (n51) at (11,1.5) {$6$};
  \node[shape=circle,draw=black,minimum size = 24pt,inner sep=0.3pt] (n61) at (0,1.5) {$5$};
	\node[shape=circle,draw=black,minimum size = 24pt,inner sep=0.3pt] (n71) at (11,-1.5) {{\tiny $2d-6$}};
  \node[shape=circle,draw=black,minimum size = 24pt,inner sep=0.3pt] (n81) at (0,-1.5) {{\tiny $2d-7$}};
	\node[shape=circle,draw=black,minimum size = 24pt,inner sep=0.3pt] (n7) at (9.4,-3) {{\tiny $2d-4$}};
  \node[shape=circle,draw=black,minimum size = 24pt,inner sep=0.3pt] (n8) at (1.6,-3) {{\tiny $2d-5$}};
	\node[shape=circle,draw=black,minimum size = 24pt,inner sep=0.3pt] (n9) at (7,-3) {{\tiny $2d-2$}};
  \node[shape=circle,draw=black,minimum size = 24pt,inner sep=0.3pt] (n10) at (4,-3) {{\tiny $2d-3$}};
	
	\node[shape=circle,minimum size = 24pt,inner sep=0.3pt] (m4) at (0,0.2) {{\large \vdots}};
	\node[shape=circle,minimum size = 24pt,inner sep=0.3pt] (m5) at (11,0.2) {{\large \vdots}};

	  \draw[line width=0.5mm,blue]  (n1) -- (n2)  ;
		\draw[line width=0.5mm,blue]  (n1) -- (n4)  ;
		\draw[line width=0.5mm,blue]  (n1) -- (n6)  ;
		\draw[line width=0.5mm,blue]  (n1) -- (n61)  ;
		\draw[line width=0.5mm,blue]  (n1) -- (n8)  ;
		\draw[line width=0.5mm,blue]  (n1) -- (n81)  ;
		\draw[line width=0.5mm,blue]  (n1) -- (n10)  ;
		\draw[line width=0.5mm,blue]  (n3) -- (n2)  ;
	  \draw[line width=0.5mm,blue]  (n5) -- (n2)  ;
		\draw[line width=0.5mm,blue]  (n51) -- (n2)  ;
	  \draw[line width=0.5mm,blue]  (n7) -- (n2)  ;
	  \draw[line width=0.5mm,blue]  (n71) -- (n2)  ;
		\draw[line width=0.5mm,blue]  (n9) -- (n2)  ;
		%\draw[line width=0.5mm,dashed]  (n2) -- (n3);	
\end{tikzpicture}
\caption{Labeling for $\Omega_d^{(d)}$} \label{fig31}
\end{figure}
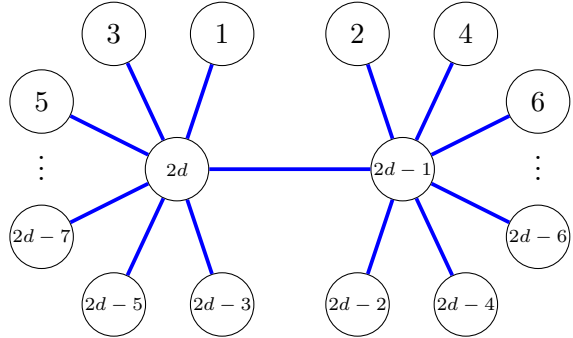
\end{remark}

\begin{lemma} \label{omegad}
For all $1\leq t\leq d$ the graph $\Omega_t^{(d)}$  is isomorphic to $TS_d$. In particular $E_d=(\Omega_1^{(d)},\dots,\Omega_d^{(d)})$ is a cycle-free $d$-partition of $K_{2d}$.
\end{lemma}
\begin{proof}
Indeed, for each $1\leq t\leq d$  the vertices $2t-1$ and $2t$ are the centers of the twin star (connected by the edge $(2t-1,2t)$), while half of the other vertices  of $\Omega_t^{(d)}$ are connected to $2t-1$, and the other half are connected to $2t$. This shows that the graph $\Omega_t^{(d)}$ is isomorphic to $TS_d$. Moreover, every edge $(i,j)\in E(K_{2d})$ belongs to a unique graph $\Omega_t^{(d)}$ determined by the conditions $i\in S_a$, $j\in S_b$ and the value of $i+j \; (mod\; 2)$. In particular, this shows that $E_d$ is a homogeneous cycle-free $d$-partition of $K_{2d}$, and so the set $\mathcal{P}^{cf}_d(K_{2d})$ is non empty for all $d\geq 1$.
\end{proof}

\begin{remark}
When $d=3$ the partition $E_3$ is presented in Figure \ref{fig1}, and when $d=4$ the partition $E_4$ is presented in Figure \ref{figd4}. %One of the main results of this paper is showing that any cycle-free $4$-partition of $K_8$ is involution equivalent to $E_4$.
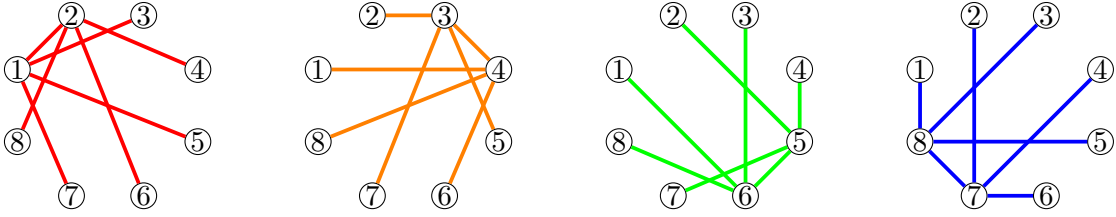
\begin{figure}[h!]
	\centering
	\begin{tikzpicture}
		[scale=0.8,auto=left,every node/.style={shape = circle, draw, fill = white,minimum size = 1pt, inner sep=0.3pt}]%baseline=(a.center)]%{circle,fill=black}]
		%\tikzset{VertexStyle/.style = {shape = circle,fill = black,minimum size = 9mm,inner sep=2pt}}
		\node (n1) at (0,2.1) {1};
		\node (n2) at (0.9,3)  {2};
		\node (n3) at (2.1,3)  {3};
		\node (n4) at (3,2.1)  {4};
		\node (n5) at (3,0.9)  {5};
		\node (n6) at (2.1,0)  {6};
		\node (n7) at (0.9,0)  {7};
		\node (n8) at (0,0.9)  {8};
		%\node (front) at (-2,0.5)   {$\Gamma_1$}
		\foreach \from/\to in {n1/n2,n1/n3,n1/n5,n1/n7,n2/n4,n2/n6,n2/n8}
		\draw[line width=0.5mm,red]  (\from) -- (\to);	
		\node (n12) at (5,2.1) {1};
		\node (n22) at (5.9,3)  {2};
		\node (n32) at (7.1,3)  {3};
		\node (n42) at (8,2.1)  {4};
		\node (n52) at (8,0.9)  {5};
		\node (n62) at (7.1,0)  {6};
		\node (n72) at (5.9,0)  {7};
		\node (n82) at (5,0.9)  {8};
		%\node (front) at (-2,0.5)   {$\Gamma_1$}
		\foreach \from/\to in {n32/n22,n32/n42,n32/n52,n32/n72,n42/n12,n42/n62,n42/n82}
		\draw[line width=0.5mm,orange]  (\from) -- (\to);	
		
		\node (n13) at (10,2.1) {1};
		\node (n23) at (10.9,3)  {2};
		\node (n33) at (12.1,3)  {3};
		\node (n43) at (13,2.1)  {4};
		\node (n53) at (13,0.9)  {5};
		\node (n63) at (12.1,0)  {6};
		\node (n73) at (10.9,0)  {7};
		\node (n83) at (10,0.9)  {8};
		%\node (front) at (-2,0.5)   {$\Gamma_1$}
		\foreach \from/\to in {n53/n63,n23/n53,n43/n53,n53/n73,n13/n63,n33/n63,n63/n83}
		\draw[line width=0.5mm,green]  (\from) -- (\to);	
		
		\node (n14) at (15,2.1) {1};
		\node (n24) at (15.9,3)  {2};
		\node (n34) at (17.1,3)  {3};
		\node (n44) at (18,2.1)  {4};
		\node (n54) at (18,0.9)  {5};
		\node (n64) at (17.1,0)  {6};
		\node (n74) at (15.9,0)  {7};
		\node (n84) at (15,0.9)  {8};
		%\node (front) at (-2,0.5)   {$\Gamma_1$}
		\foreach \from/\to in {n74/n84,n74/n24,n44/n74,n64/n74,n14/n84,n34/n84,n54/n84}
		\draw[line width=0.5mm,blue]  (\from) -- (\to);
		
		%\node[state,above of=B1] (C1) {$C_1$};
	\end{tikzpicture}
	\caption{$E_4=(\Omega_1^{(4)},\Omega_2^{(4)},\Omega_3^{(4)},\Omega_4^{(4)})$, a homogeneous cycle-free $4$-partition of  $K_8$ } \label{figd4}
\end{figure}

\end{remark}

The following result was proved in \cite{edge} and it is the main technical ingredient of this paper. %For completeness we recall the main steps of the proof.

\begin{lemma} Let $(\Gamma_1,\dots, \Gamma_d)$ be a cycle-free $d$-partition of $K_{2d}$, and take $1\leq x,y,z\leq 2d$ three distinct vertices of $K_{2d}$. Then, there exist $(\Lambda_1,\dots, \Lambda_d)$, a unique cycle-free  $d$-partition of $K_{2d}$  such that the two partitions $(\Gamma_1,\dots,\Gamma_d)$  and $(\Lambda_1,\dots,\Lambda_d)$ coincide on every edge of $K_{2d}$ except on the edges  $(x,y)$, $(x,z)$ and $(y,z)$ where they are different on at least two edges.
\label{keylemma}
\end{lemma}

\begin{definition}
 For $(\Gamma_1,\dots, \Gamma_d)\in \mathcal{P}_d^{cf}(K_{2d})$ and $1\leq x<y<z\leq 2d$, we denote by $$(\Gamma_1,\dots, \Gamma_d)^{(x,y,z)}\in \mathcal{P}_d^{cf}(K_{2d}),$$ the unique cycle-free $d$-partition of the complete graph $K_{2d}$ given by Lemma \ref{keylemma}.
\end{definition}

\begin{example}
Let's see an example on how Lemma \ref{keylemma} works when $d=2$. Consider $(\Gamma_1,\Gamma_2)$ the  cycle-free $2$-partition of $K_4$ from the left side of Figure \ref{figkeylemma}, and  take $(x,y,z)=(1,2,3)$.  By rearranging  the edges $(1,2)$, $(1,3)$ and $(2,3)$, we obtain the two homogeneous $2$-partitions of $K_4$ on the right side of Figure \ref{figkeylemma}. Since the partition on the bottom right corner has a cycle, it must be that $(\Gamma_1,\Gamma_2)^{(1,2,3)}$ is the partition on the top right corner of Figure \ref{figkeylemma}. %Later in the paper we will use several computations of this type.

\begin{figure}[h]
\centering
\begin{tikzpicture}
  [scale=0.8,auto=left]%,every node/.style={shape = circle, draw, fill = white,minimum size = 14pt, inner sep=0.3pt}]%baseline=(a.center)]%{circle,fill=black}]
	%\tikzset{VertexStyle/.style = {shape = circle,fill = black,minimum size = 9mm,inner sep=2pt}}

	\node[shape=circle,draw=black,minimum size = 15pt,inner sep=0.3pt] (n1) at (0,0.2) {{$1$}};
  \node[shape=circle,draw=black,minimum size = 15pt,inner sep=0.3pt] (n2) at (2,0.2) {{$2$}};
	\node[shape=circle,draw=black,minimum size = 15pt,inner sep=0.3pt] (n3) at (2,-1.8) {{$3$}};
  \node[shape=circle,draw=black,minimum size = 15pt,inner sep=0.3pt] (n4) at (0,-1.8) {{$4$}};

	  \draw[line width=0.5mm,red]  (n1) -- (n2)  ;
		\draw[line width=0.5mm,red]  (n1) -- (n4)  ;
		\draw[line width=0.5mm,red]  (n2) -- (n3)  ;
		
			\node[shape=circle,draw=black,minimum size = 15pt,inner sep=0.3pt] (n11) at (3.5,0.2) {{$1$}};
  \node[shape=circle,draw=black,minimum size = 15pt,inner sep=0.3pt] (n21) at (5.5,0.2) {{$2$}};
	\node[shape=circle,draw=black,minimum size = 15pt,inner sep=0.3pt] (n31) at (5.5,-1.8) {{$3$}};
  \node[shape=circle,draw=black,minimum size = 15pt,inner sep=0.3pt] (n41) at (3.5,-1.8) {{$4$}};

	  \draw[line width=0.5mm,blue]  (n11) -- (n31)  ;
		\draw[line width=0.5mm,blue]  (n41) -- (n21)  ;
		\draw[line width=0.5mm,blue]  (n31) -- (n41)  ;

%	\node[shape=circle,minimum size = 24pt,inner sep=0.3pt] (m4) at (7,1) {{\large $\longrightarrow$}};
 %       \node[shape=circle,minimum size = 24pt,inner sep=0.3pt] (mu4) at (7,1.5) {{ $(x,z,t)$}};
	
		\node[shape=circle,draw=black,minimum size = 15pt,inner sep=0.3pt] (n13) at (8.5,2) {{$1$}};
  \node[shape=circle,draw=black,minimum size = 15pt,inner sep=0.3pt] (n23) at (10.5,2) {{$2$}};
	\node[shape=circle,draw=black,minimum size = 15pt,inner sep=0.3pt] (n33) at (10.5,0) {{$3$}};
  \node[shape=circle,draw=black,minimum size = 15pt,inner sep=0.3pt] (n43) at (8.5,0) {{$4$}};
	
	  \draw[line width=0.5mm,red]  (n13) -- (n33)  ;
		\draw[line width=0.5mm,red]  (n43) -- (n13)  ;
		\draw[line width=0.5mm,red]  (n33) -- (n23)  ;
		
			\node[shape=circle,draw=black,minimum size = 15pt,inner sep=0.3pt] (n14) at (12,2) {{$1$}};
  \node[shape=circle,draw=black,minimum size = 15pt,inner sep=0.3pt] (n24) at (14,2) {{$2$}};
	\node[shape=circle,draw=black,minimum size = 15pt,inner sep=0.3pt] (n34) at (14,0) {{$3$}};
  \node[shape=circle,draw=black,minimum size = 15pt,inner sep=0.3pt] (n44) at (12,0) {{$4$}};
	
	%\node[shape=circle,draw=black,minimum size = 24pt,inner sep=0.3pt] (n14) at (12,2) {{\tiny $2a-1$}};
  %\node[shape=circle,draw=black,minimum size = 24pt,inner sep=0.3pt] (n24) at (14,2) {{\tiny $2a$}};
	%\node[shape=circle,draw=black,minimum size = 24pt,inner sep=0.3pt] (n34) at (14,0) {{\tiny $2b-1$}};
  %\node[shape=circle,draw=black,minimum size = 24pt,inner sep=0.3pt] (n44) at (12,0) {{\tiny $2b$}};

	  \draw[line width=0.5mm,blue]  (n14) -- (n24)  ;
		\draw[line width=0.5mm,blue]  (n44) -- (n24)  ;
		\draw[line width=0.5mm,blue]  (n34) -- (n44)  ;

		%\draw[line width=0.5mm,dashed]  (n2) -- (n3);	

%%%%%%%%%%%%%%%%%%%%%%%%%%%%%%%%%%%%%%%%%%%%%%%%%%

\node[shape=circle,minimum size = 24pt,inner sep=0.3pt] (m4) at (7,-1) {{\large $\longrightarrow$}};
        \node[shape=circle,minimum size = 24pt,inner sep=0.3pt] (mu4) at (7,-0.5) {{ $(1,2,3)$}};
	
		\node[shape=circle,draw=black,minimum size = 15pt,inner sep=0.3pt] (n13) at (8.5,-1.5) {{$1$}};
  \node[shape=circle,draw=black,minimum size = 15pt,inner sep=0.3pt] (n23) at (10.5,-1.5) {{$2$}};
	\node[shape=circle,draw=black,minimum size = 15pt,inner sep=0.3pt] (n33) at (10.5,-3.5) {{$3$}};
  \node[shape=circle,draw=black,minimum size = 15pt,inner sep=0.3pt] (n43) at (8.5,-3.5) {{$4$}};
	
	  \draw[line width=0.5mm,red]  (n13) -- (n23)  ;
		\draw[line width=0.5mm,red]  (n13) -- (n43)  ;
		\draw[line width=0.5mm,red]  (n13) -- (n33)  ;
		
			\node[shape=circle,draw=black,minimum size = 15pt,inner sep=0.3pt] (n14) at (12,-1.5) {{$1$}};
  \node[shape=circle,draw=black,minimum size = 15pt,inner sep=0.3pt] (n24) at (14,-1.5) {{$2$}};
	\node[shape=circle,draw=black,minimum size = 15pt,inner sep=0.3pt] (n34) at (14,-3.5) {{$3$}};
  \node[shape=circle,draw=black,minimum size = 15pt,inner sep=0.3pt] (n44) at (12,-3.5) {{$4$}};
	
	%\node[shape=circle,draw=black,minimum size = 24pt,inner sep=0.3pt] (n14) at (12,2) {{\tiny $2a-1$}};
  %\node[shape=circle,draw=black,minimum size = 24pt,inner sep=0.3pt] (n24) at (14,2) {{\tiny $2a$}};
	%\node[shape=circle,draw=black,minimum size = 24pt,inner sep=0.3pt] (n34) at (14,0) {{\tiny $2b-1$}};
  %\node[shape=circle,draw=black,minimum size = 24pt,inner sep=0.3pt] (n44) at (12,0) {{\tiny $2b$}};

	  \draw[line width=0.5mm,blue]  (n24) -- (n34)  ;
		\draw[line width=0.5mm,blue]  (n44) -- (n24)  ;
		\draw[line width=0.5mm,blue]  (n44) -- (n34)  ;

		%\draw[line width=0.5mm,dashed]  (n2) -- (n3);	

  %%%%%%%%%%%%%%%%%%%%%%%%%%%%%%%%%%%%%%%%%%%%%%%%

 % \node[shape=circle,minimum size = 24pt,inner sep=0.3pt] (m4) at (7,-6.25) {{\large $\longrightarrow$}};
  %      \node[shape=circle,minimum size = 24pt,inner sep=0.3pt] (mu4) at (7,-5.75) {{ $(x,y,z)$}};

		%\draw[line width=0.5mm,dashed]  (n2) -- (n3);	

\end{tikzpicture}
\caption{Computing  $(\Gamma_1,\Gamma_2)^{(1,2,3)}$} \label{figkeylemma}
\end{figure}
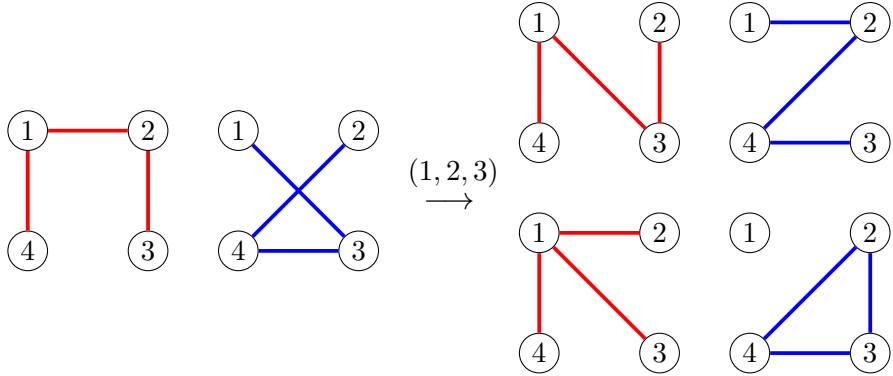

\end{example}

%\begin{example}
%Add an example for $d=3$.
%\end{example}

\begin{remark} \label{localglobal} From the definition we know that  the partitions $(\Gamma_1,\dots, \Gamma_d)$ and $(\Gamma_1,\dots, \Gamma_d)^{(x,y,z)}$ are different only on the edges $(x,y)$, $(x,z)$ and $(y,z)$. However, the way these two partitions are different on those three edges depends not only on the value of $(x,y,z)$, but also on the partition $(\Gamma_1,\dots, \Gamma_d)$ itself (i.e. the change   on the edges  $(x,y)$, $(x,z)$ and $(y,z)$ is not a local but a global process).
\end{remark}

\begin{remark} For $(\Gamma_1,\dots, \Gamma_d)\in \mathcal{P}_d^{cf}(K_{2d})$ and $1\leq x<y<z\leq 2d$ we have  $$((\Gamma_1,\dots, \Gamma_d)^{(x,y,z)})^{(x,y,z)}=(\Gamma_1,\dots, \Gamma_d).$$
In particular there is an involution $(x,y,z):\mathcal{P}_d^{cf}(K_{2d})\to \mathcal{P}_d^{cf}(K_{2d})$ determined by
$$(\Gamma_1,\dots,\Gamma_d)\to (\Gamma_1,\dots,\Gamma_d)^{(x,y,z)}.$$
\end{remark}

\begin{definition}
We denote by $G_d$ the subgroup of the symmetric group $S_{\mathcal{P}_d^{cf}(K_{2d})}$ generated by the involutions $(x,y,z)$ for all
$1\leq x<y<z\leq 2d$. If $g\in G_d$ we will denote by $(\Gamma_1,\dots, \Gamma_d)^g$ its action on the cycle-free $d$-partition $(\Gamma_1,\dots, \Gamma_d)$. \label{defGd}
\end{definition}
\begin{example}
One can show that $\mathcal{P}_2^{cf}(K_{4})$ has $12$ elements, $G_2$ acts transitively on $\mathcal{P}_2^{cf}(K_{4})$, and $G_2$ is isomorphic to $S_5$. It was shown in \cite{edge} that  $\mathcal{P}_3^{cf}(K_{6})$ has  \num{66240} elements and  $G_3$ acts transitively on $\mathcal{P}_3^{cf}(K_{6})$. The isomorphism type and the order of the group $G_3$ are unknown at this point.% (by definition $G_3$ is a subgroup of the symmetric group $S_{66240}$).
\end{example}

\subsection{The map $det^{S^2}$}
\label{subsection21}
The results in this subsection are not needed in this paper. We provide them only as partial motivation for studying cycle-free $d$-partitions of $K_{2d}$.
\begin{remark} Let $V_d$ be a $k$ vector space over a field $k$, with $dim_k(V_d)=d$.  The determinant map is the unique  (up to a constant) nontrivial linear map  $$det:V_d^{\otimes d}\to k,$$ with the property that $det(\otimes_{1\leq i\leq d} (v_i))=0$ if there exist $1\leq x<y\leq d$ such that $v_x=v_y$.

It was shown in \cite{dets2} that for all $d\geq 1$ there exists a nontrivial linear map $$det^{S^2}:V_d^{\otimes d(2d-1)}\to k,$$ with the property that $det^{S^2}(\otimes_{1\leq i<j\leq 2d}(v_{i,j}))=0$ if there exist $1\leq  x<y<z\leq 2d$ such that $v_{x,y}=v_{x,z}=v_{y,z}$.

It was  conjectured in \cite{sta2} that such a map is unique up to a constant. The conjecture is known to be true for the case $d=2$ (\cite{sta2}) and $d=3$ (\cite{edge}). It follows from the results in \cite{edge} that the conjecture is true for $V_d$ if and only if the action of the group $G_d$ on $\mathcal{P}_d^{cf}(K_{2d})$ is transitive.
\end{remark}

Let's briefly recall from \cite{edge} the connection between the map $det^{S^2}$ and $\mathcal{P}_d^{cf}(K_{2d})$. Fix $\mathcal{B}_d=\{e_1,e_2,\dots,e_d\}$ a basis for $V_d$. Then, to every edge $d$-partition $(\Gamma_1,\dots, \Gamma_d)$  of $K_{2d}$ we can associate a unique  element $$f_{(\Gamma_1,\dots, \Gamma_d)}=\otimes_{1\leq i<j\leq 2d}(f_{i,j})\in V_d^{\otimes d(2d-1)},$$ determined by  $f_{i,j}=e_k$ if and only if $(i,j)\in V(\Gamma_k)$. The following result supports the conjecture about uniqueness of the map $det^{S^2}$.

\begin{theorem} \cite{dets2}
 Let $(\Gamma_1,\dots, \Gamma_d)$ be an edge $d$-partition of $K_{2d}$, and $f_{(\Gamma_1,\dots, \Gamma_d)}$ be the associated element in $V_d^{\otimes d(2d-1)}$. Then $(\Gamma_1,\dots, \Gamma_d)$ is cycle-free if and only if $det^{S^2}(f_{(\Gamma_1,\dots, \Gamma_d)})\neq 0$.
\end{theorem}
\begin{remark}
Finally, it follows from results in \cite{edge} that if the conjecture is true then the map $det^{S^2}$ can be written as
\begin{equation}
det^{S^2}(\otimes_{1\leq i<j\leq 2d}(v_{i,j}))=\sum_{(\Gamma_1,\dots,\Gamma_d)\in \mathcal{P}^{cf}_d(K_{2d})} \varepsilon^{S^2}_d((\Gamma_1,\dots,\Gamma_d)) \prod_{t=1}^d(\prod_{(i,j)\in \Gamma_t}v_{i,j}^t),\label{for1}
\end{equation}
 where the sum is taken over the set $\mathcal{P}^{cf}_d(K_{2d})$ of homogeneous, cycle-free, edge $d$-partitions of the complete graph $K_{2d}$,  $v_{i,j}=\sum_{t=1}^dv_{i,j}^te_t$, and $\varepsilon^{S^2}_d:\mathcal{P}^{cf}_d(K_{2d})\to \{-1,1\}$ is determined (up to a sign) by the fact that $\varepsilon^{S^2}_d((\Gamma_1,\dots,\Gamma_d)^{(x,y,z)})=-\varepsilon^{S^2}_d((\Gamma_1,\dots,\Gamma_d))$ for all $1\leq x<y<z\leq 2d$. One should notice  the similarity with the formula of the determinant.
\end{remark}
\begin{remark}
Besides its connection to combinatorics, the map $det^{S^2}$ has applications to geometry \cite{sv}, and to the equilibrium problem in physics \cite{req}.
\end{remark}

\section{Equivalence relations on the set of cycle-free $d$-partitions of $K_{2d}$}
\label{section2}

In this section we state the main conjecture concerning this paper. We also introduce two equivalence relations on the set $\mathcal{P}_d^{cf}(K_{2d})$  and discuss some of their properties.

\subsection{Involution Equivalence} First recall from Lemma \ref{omegad} that the set $\mathcal{P}_d^{cf}(K_{2d})$ is not empty.
\begin{definition} Let $\mathcal{P}=(\Gamma_1,\dots,\Gamma_d)$ and $\mathcal{Q}=(\Theta_1,\dots,\Theta_d)\in \mathcal{P}_d^{cf}(K_{2d})$ be two cycle-free $d$-partitions of $K_{2d}$. We say that they are {\it involution equivalent} if there exist $n\in \mathbb{N}$, cycle-free $d$-partitions $\mathcal{P}_i\in \mathcal{P}_d^{cf}(K_{2d})$ for $0\leq i\leq n$, and $1\leq x_i<y_i<z_i\leq 2d$ for $1\leq i\leq n$ such that
$$\mathcal{P}_0=\mathcal{P}, ~~~\mathcal{P}_n=\mathcal{Q},~~ {\rm and}  ~~\mathcal{P}_i^{(x_i,y_i,z_i)}=\mathcal{P}_{i-1} ~~{\rm for ~~all} ~~1\leq i\leq n.$$
\end{definition}

The results from \cite{edge,dets2} show that the uniqueness (up to a constant)  of the map $det^{S^2}$ is equivalent with the following purely combinatorial conjecture.
\begin{conjecture} Any two cycle-free $d$-partitions of $K_{2d}$ are involution equivalent. \label{mainconj}
\end{conjecture}

\begin{remark} It is known from \cite{edge} that the conjecture is true for $d=2$ and $d=3$. The main result of this paper is showing that  Conjecture \ref{mainconj} is implied by the so called twin-star hypothesis.  We will also check that the conjecture is true for $d=4$.
\end{remark}

\begin{remark} \label{remGd} Recall from Definition \ref{defGd} that the group $G_d$ is generated by the involutions  $(x,y,z):\mathcal{P}_d^{cf}(K_{2d})\to \mathcal{P}_d^{cf}(K_{2d})$ for $1\leq x<y<z\leq 2d$. This means that  two cycle free $d$-partitions $\mathcal{P}=(\Gamma_1,\dots,\Gamma_d)$ and $\mathcal{Q}=(\Theta_1,\dots,\Theta_d)$ of $K_{2d}$ are involution equivalent if and only if there exists $g\in G_d$ such that $\mathcal{P}=\mathcal{Q}^g$. In particular, Conjecture \ref{mainconj} is equivalent with $G_d$ acting transitively on the set $\mathcal{P}_d^{cf}(K_{2d})$.
\end{remark}

\subsection{Weak Equivalence}
Before we prove the main result of this paper we need another equivalence relation on the set $\mathcal{P}_d^{cf}(K_{2d})$.

 Notice that there is an obvious action of the group $S_{2d}\times S_d$ on the set $\mathcal{P}_d^{cf}(K_{2d})$ induced by the action of the group $S_{2d}$ on the vertices of the graph $K_{2d}$, and the action of the group $S_d$ on the order of the subgraphs $\Gamma_i$. More precisely, we have.
\begin{definition} If $\sigma\in S_{2d}$,  and $(\Gamma_1,\dots, \Gamma_d)\in \mathcal{P}_d^{cf}(K_{2d})$ we define
$$\sigma\circ(\Gamma_1,\dots, \Gamma_d)=(\sigma \circ\Gamma_1,\dots, \sigma\circ\Gamma_d),$$
where for a subgraph $\Gamma$ of $K_{2d}$ with $E(\Gamma)=\{(i_1,j_1),...,(i_{2d-1},j_{2d-1})\}$ we take the edges of $\sigma\circ\Gamma$ to be
$E(\sigma\circ\Gamma)=\{(\sigma(i_1),\sigma(j_1)),...,(\sigma(i_{2d-1}),\sigma(j_{2d-1}))\}$.

For $\tau\in S_d$ we take
$$\tau \odot(\Gamma_1,\dots, \Gamma_d)=(\Gamma_{\tau^{-1}(1)},\dots,\Gamma_{\tau^{-1}(d)}).$$
\end{definition}

These two actions commute with each other and so we get an action of the group $S_{2d}\times S_d$ on $\mathcal{P}_d^{cf}(K_{2d})$.

\begin{remark} One can see that this action is transitive for $d=2$. For $d=3$ it was shown in \cite{edge} that there are $19$ equivalence classes for the action of $S_6\times S_3$ on the set $\mathcal{P}_3^{cf}(K_{6})$.
\end{remark}

The above actions are compatible as follows.
\begin{lemma} Let $(\Gamma_1,\dots,\Gamma_d)$ be a cycle-free $d$-partition of $K_{2d}$, $\sigma\in S_{2d}$, $\tau\in S_d$, and $1\leq x<y<z\leq 2d$. Then $$(\sigma,\tau)\cdot((\Gamma_1,\dots,\Gamma_d)^{(x,y,z)})=((\sigma,\tau)\cdot(\Gamma_{1},\dots,\Gamma_{d}))^{(\sigma(x),\sigma(y),\sigma(z))}.$$
\label{lemma4}
\end{lemma}
\begin{proof}
It follows directly form the definitions.
\end{proof}

The following technical result will be used in the next section. In order to make the paper easier to read, we postpone its proof to Appendix \ref{appendix1}.

\begin{lemma} Let $E_d\in \mathcal{P}_d^{cf}(K_{2d})$ be the cycle-free $d$-partition of $K_{2d}$ introduced in Lemma \ref{omegad}.
Then for every $(\sigma, \tau) \in S_{2d}\times S_d$ there exists an element $g\in G_d$ such that $$(\sigma, \tau)\cdot E_d= (E_d)^g.$$ \label{lemma3}
In other words $(\sigma, \tau)\cdot E_d$ is involution equivalent to $E_d$.
\end{lemma}

Combining the action of $S_{2d}\times S_d$ on $\mathcal{P}_d^{cf}(K_{2d})$ with involution equivalence  we get the following definition.

\begin{definition} Let $\mathcal{P}=(\Gamma_1,\dots,\Gamma_d)$ and $\mathcal{Q}=(\Theta_1,\dots,\Theta_d)$ be two cycle-free $d$-partitions of $K_{2d}$. We say that they are {\it weakly equivalent} if there exist $n\geq 0$, cycle-free $d$-partitions $\mathcal{P}_i$ for $0\leq i\leq n$, integers $1\leq x_i<y_i<z_i\leq 2d$ for $1\leq i\leq n$, and $(\sigma_i,\tau_i)\in S_{2d}\times S_d$ for $1\leq i\leq n$  such that
$$\mathcal{P}_0=\mathcal{P}, ~~~\mathcal{P}_n=\mathcal{Q},~~ {\rm and}  ~~(\sigma_i,\tau_i)\cdot (\mathcal{P}_i^{(x_i,y_i,z_i)})=\mathcal{P}_{i-1} ~~{\rm for ~~all} ~~1\leq i\leq n.$$
\end{definition}

\begin{remark} Let $\mathcal{P}=(\Gamma_1,\dots,\Gamma_d)$ and $\mathcal{Q}=(\Theta_1,\dots,\Theta_d)$ be two cycle-free $d$-partitions of $K_{2d}$.  From Remark \ref{remGd} and Lemma \ref{lemma4}   we get that $\mathcal{P}$ and $\mathcal{Q}$ are weakly  equivalent if and only if there exist  $g\in G_d$ and $(\sigma, \tau)\in S_{2d}\times S_d$ such that
$$\mathcal{P}=((\sigma, \tau)\cdot \mathcal{Q})^g.$$
\label{remwe}
\end{remark}

\begin{lemma}  \label{lemma3eiw} Let $\mathcal{P}=(\Gamma_1,\dots,\Gamma_d)$ and $\mathcal{Q}=(\Theta_1,\dots,\Theta_d)$ be two cycle-free $d$-partitions of $K_{2d}$.\\
1) If $\mathcal{P}$ and $\mathcal{Q}$ are involution equivalent then they are weakly equivalent.\\
2) If $\mathcal{P}$ is weakly equivalent to $E_d$ then  $\mathcal{P}$ is involution equivalent to $E_d$.
\end{lemma}
\begin{proof} The first statement is obvious.  The second statement is a consequence of Lemma  \ref{lemma3} and Remark \ref{remwe}.
\end{proof}

\begin{remark}
 Notice that if Conjecture \ref{mainconj} is true then $\mathcal{P}$ and  $\mathcal{Q}$ are involution equivalent if and only if they are weakly equivalent.
\end{remark}

\section{Twin-star Hypothesis}
\label{section3}
In this section we introduce the twin-star hypothesis and discuss how it relates to Conjecture \ref{mainconj}.

\begin{definition} Let $d\in \mathbb{N}$. We say that the twin-star hypothesis $\mathcal{TS}(d)$ holds true if every cycle-free $d$-partition of $K_{2d}$ is weakly equivalent to a cycle-free $d$-partition $(\Gamma_1,\Gamma_2,\dots,\Gamma_d)$ with the property that  at least one of the graphs $\Gamma_i$ is isomorphic to the twin-star graph $TS_d$.
\end{definition}

Notice that if Conjecture \ref{mainconj} is true for a certain $d$ then the twin-star hypothesis is also true for that $d$. In particular, since Conjecture \ref{mainconj} is true is true for $d=2$, and $d=3$ we know that $\mathcal{TS}(2)$ and $\mathcal{TS}(3)$ hold true. The main result of this paper is to prove the converse of this statement.

For this we need a few lemmas.
\begin{lemma} Let $(\Gamma_1,\dots,\Gamma_d)$ be a cycle-free $d$-partition of $K_{2d}$ such that $\Gamma_d$ coincides with the graph $\Omega_d^{(d)}$. Then for each $1\leq i\leq d-1$ the graph $\Gamma_i$ has  exactly one edge  adjacent to the vertex $2d-1$, and exactly one edge  adjacent to the vertex $2d$.  \label{lemmagd}
\end{lemma}
\begin{proof}
First notice that  $\Gamma_d=\Omega_d^{(d)}$ has $d$ edges adjacent to the vertex $2d$. So, when we remove the vertex $2d$ from $K_{2d}$ the cycle-free $d$-partition $(\Gamma_1,\dots,\Gamma_d)$ will induce  a cycle-free $d$-partition $(\Theta_1,\dots,\Theta_d)$ of $K_{2d-1}$ such that $\Theta_d$ has $d-1$ edges. On the other hand the graph $K_{2d-1}$ has $(d-1)(2d-1)$ edges, so
\begin{equation}\sum_{i=1}^{d-1}|E(\Theta_i)|=|E(K_{2d-1})|-|E(\Theta_d)|=(d-1)(2d-1)-(d-1)=(d-1)(2d-2).\label{eqtheta}
\end{equation}
Next, one can see that for each $1\leq i\leq d-1$ the graph $\Theta_i$ is cycle-free and has $2d-1$ vertices, so it cannot have more than $2d-2$ edges.  Combining this with Equation (\ref{eqtheta}) we get that $$|E(\Theta_i)|=2d-2,$$
for all $1\leq i\leq d-1$.

Finally, because $|E(\Gamma_{i})|=2d-1$, $E|(\Theta_i)|=2d-2$ and the graph $\Theta_i$ is obtained from $\Gamma_i$ by removing the vertex $2d$  we get that graph $\Gamma_i$ has exactly one edge  adjacent to the vertex $2d$.

A similar argument shows that for each $1\leq i\leq d-1$ the graph $\Gamma_i$ has exactly one edge  adjacent to the vertex $2d-1$.
\end{proof}

%To summarize, we may assume that up to weak equivalence the graph $\Gamma_d$ coincides with the graph $\Omega_d^{(d)}$ (of the $d$-partition  $E_d$), and for each $1\leq i\leq d-1$ the graph $\Gamma_i$ has  exactly one edge  adjacent to the vertex $2d-1$, and exactly one edge  adjacent to the vertex $2d$. %Moreover, notice that for each $1\leq x\leq 2d-2$ exactly one of the edges $(x,2d-1)$ and $(x,2d)$ belongs to $\Gamma_d$ while the other one belongs to some $\Gamma_i$ with $1\leq i\leq d-1$.

%\begin{definition} DO WE NEED IT??
%We will denote by $\overline{G_{d-1}}$ the subgroup of $G_d$ generated by all involutions $(x,y,z)$ with $1\leq x<y<z\leq 2d-2$.
%\end{definition}

%\begin{remark} It is important to notice that $\overline{G_{d-1}}$ and $G_{d-1}$ are not the isomorphic.   However, the action of $G_{d-1}$ on $\mathcal{P}_{d-1}^{cf}(K_{2d-2})$ can be recovered from the action $\overline{G_{d-1}}$ on $\mathcal{P}_d^{cf}(K_{2d})$.
%\end{remark}

\begin{remark} Let $(\Gamma_1,\dots,\Gamma_d)$ be a cycle-free $d$-partition of $K_{2d}$ such that $\Gamma_d$ coincides with the graph $\Omega_d^{(d)}$. Take $K_{2d-2}$ be the complete graph obtained from $K_{2d}$ by removing vertices $2d-1$, $2d$ and all their adjacent edges. Then the restriction of $(\Gamma_1,\dots,\Gamma_d)$ to $K_{2d-2}$ is a cycle-free $(d-1)$-partition $(\Delta_1,\dots,\Delta_{d-1})$.

Indeed, if we remove the vertices $2d-1$ and $2d$ (and the adjacent edges) from the graph $K_{2d}$ we obtain the complete graph $K_{2d-2}$. Moreover the set of edges $E(\Gamma_d)$ intersects trivially the set of edges $E(K_{2d-2})$. So, the cycle-free $d$-partition $(\Gamma_1,\dots,\Gamma_d)$ of the graph $K_{2d}$ induces a cycle-free $(d-1)$-partition $(\Delta_1,\dots,\Delta_{d-1})$ of the complete graph $K_{2d-2}$ (where for $1\leq i\leq d-1$ the graph $\Delta_i$ is obtained from $\Gamma_i$ by removing the vertices $2d-1$ and  $2d$ and their adjacent edges).

Finally, notice that under the above assumption, if $1\leq x<y<z\leq 2d-2$ then $(\Gamma_1,\dots,\Gamma_d)^{(x,y,z)}$ is determined by $(\Delta_1,\dots,\Delta_{d-1})^{(x,y,z)}$ and the restriction of  $(\Gamma_1,\dots,\Gamma_d)$ to all the edges connected to vertices $2d-1$ and $2d$. For a more general situation see also Remark \ref{reKsubG}.
\label{remlocal}
\end{remark}

\begin{lemma} Let $(\Gamma_1,\dots,\Gamma_d)$ be a cycle-free $d$-partition of $K_{2d}$ such that $\Gamma_d$ coincide with the graph $\Omega_d^{(d)}$.  Assume that when we remove the vertices $2d-1$ and $2d$ from the graph $K_{2d}$ the cycle-free $d-1$-partition of $K_{2d-2}$induced by $(\Gamma_1,\dots,\Gamma_d)$ is involution equivalent to $E_{d-1}$.
Then for every  $1\leq a<b\leq 2d-2$ and $1\leq s\leq d-1$ there exists a cycle-free $d$-partition $(\Psi_1,\dots,\Psi_d)$ such that $(a,b)\in E(\Psi_s)$, and  $(\Gamma_1,\dots,\Gamma_d)$ is involution equivalent to $(\Psi_1,\dots,\Psi_d)$ via a sequence of involutions $(x,y,z)$ with $1\leq x<y<z\leq 2d-2$.  In particular $(\Gamma_1,\dots,\Gamma_d)$ and $(\Psi_1,\dots,\Psi_d)$ coincide on all the edges connected to vertices $2d-1$ or $2d$, and so $\Psi_d=\Omega_d^{(d)}$. \label{lemmaxys}
\end{lemma}
\begin{proof}
Our hypothesis states that the cycle-free $(d-1)$-partition $(\Delta_1,\dots,\Delta_{d-1})$ induced by $(\Gamma_1,\dots,\Gamma_d)$ on $K_{2d-2}$ is involution equivalent to
$E_{d-1}$ via a sequence of involutions $(x,y,z)$ with $1\leq x<y<z\leq 2d-2$. %and $\sigma\times \tau\in  S_{2d-2}\times S_{d-1}$.
By Lemma \ref{lemma3} we know $E_{d-1}$ is involution equivalent to $\tau \odot E_{d-1}$ for any $\tau\in S_{d-1}$ (via a sequence of involutions $(x,y,z)$ with $1\leq x<y<z\leq 2d-2$). By choosing the appropriate $\tau\in S_{d-1}$ we may assume that the edge $(a,b)$ belongs to the $s$-th component of the partition $\tau\odot E_{d-1}$.

We take $(\Psi_1,\dots,\Psi_d)$ to be the $d$-partition of $K_{2d}$ that coincides with $(\Gamma_1,\dots, \Gamma_d)$ on all the edges connected to vertex $2d-1$ and $2d$, and such that the restriction of $(\Psi_1,\dots,\Psi_d)$ to $K_{2d-2}$ is the partition $\tau \odot E_{d-1}$. By Remark \ref{remlocal} we know that $(\Psi_1,\dots,\Psi_d)$  is involution equivalent to $(\Gamma_1,\dots, \Gamma_d)$ via a sequence of involutions $(x,y,z)$ with $1\leq x<y<z\leq 2d-2$. Moreover, from construction we have that the edge $(a,b)$ belongs to $E(\Psi_s)$.
\end{proof}

\begin{lemma}  Let $(\Gamma_1,\dots,\Gamma_d)$ be a cycle-free $d$-partition of $K_{2d}$ such that $\Gamma_d$ coincides with the graph $\Omega_d^{(d)}$. Assume  that when we remove the vertices $2d-1$ and $2d$ from the graph $K_{2d}$ the cycle-free $(d-1)$-partition of $K_{2d-2}$ induced by $(\Gamma_1,\dots,\Gamma_d)$ is involution equivalent to $E_{d-1}$.
Then there exists a cycle-free $d$-partition $(\Psi_1,\dots,\Psi_d)$ of $K_{2d}$ that is involution equivalent  to $(\Gamma_1,\dots,\Gamma_d)$ and has the property that the partition $(\Psi_1,\dots,\Psi_d)$ coincides with $E_d$ on all edges connected to the vertices $2d-1$ and $2d$.
\label{lemmare2d}

\end{lemma}
\begin{proof}

Let $1\leq a\neq b\leq 2d-2$ and $1\leq i<j\leq d-1$ such that $(a,2d)\in \Gamma_i$ and $(b,2d)\in \Gamma_j$. We claim  that up to a sequence of involutions $(x,y,z)$, we may switch the edge $(a,2d)\in \Gamma_i$ with the edge $(b,2d)\in \Gamma_j$ without changing our partition on any of the other edges adjacent to $2d-1$ or $2d$ (including the graph $\Gamma_d$).

Indeed, using Lemma \ref{lemmaxys} we may assume that up to a sequence of involutions $(x,y,z)$ with $1\leq x<y<z\leq 2d-2$, we have that  edge $(a,b)\in E(\Gamma_i)$. Notice that these involutions will not affect any of the edges connected to vertices $2d-1$ or $2d$.

Next, consider the action of the involution $(a,b,2d)$ on the $d$-partition $(\Gamma_1,\dots,\Gamma_d)$.  On the left side of Figure \ref{figc2}, we have the initial partition $(\Gamma_1,\dots,\Gamma_d)$ restricted to the edges $(a,b), (a,2d)\in E(\Gamma_i)$, and $(b,2d)\in E(\Gamma_j)$. We want to compute $(\Gamma_1,\dots,\Gamma_d)^{(a,b,2d)}$. Notice that the $d$-partition corresponding to the bottom right  graph in Figure \ref{figc2} has no edge adjacent to the vertex $2d$ in the $j$-th component of the partition, and so the $j$-th component of that partition will have a cycle. This means that the  graph in the top right graph in Figure \ref{figc2} represents the cycle-free $d$-partition $(\Gamma_1,\dots,\Gamma_d)^{(a,b,2d)}$. So, the action of involution $(a,b,2d)$ on $(\Gamma_1,\dots,\Gamma_d)$ will switch the edges $(a,2d)$ and $(b,2d)$ between the graph $\Gamma_i$ and $\Gamma_j$, and will not change any other edge connected to vertex $2d-1$ or $2d$, which proves our claim.

\begin{figure}[ht]
\centering
\begin{tikzpicture}
  [scale=0.8,auto=left]%,every node/.style={shape = circle, draw, fill = white,minimum size = 5pt, inner sep=0.3pt}]%baseline=(a.center)]%{circle,fill=black}]
	%\tikzset{VertexStyle/.style = {shape = circle,fill = black,minimum size = 9mm,inner sep=2pt}}
	%\node[shape=circle,draw=black,minimum size = 14pt,inner sep=0.3pt] (n1) at (0,1) {$a$};
  \node[shape=circle,draw=black,minimum size = 16pt,inner sep=0.5pt] (nx) at (0,-1) {$a$};
  \node[shape=circle,draw=black,minimum size = 16pt,inner sep=0.5pt] (ny) at (3,-1)  {$b$};
  \node[shape=circle,draw=black,minimum size = 16pt,inner sep=0.5pt] (nd) at (1.5,1.5)  {{\tiny $2d$}};

\node[shape=circle,minimum size = 24pt,inner sep=0.3pt] (m4) at (4.5,0) {{\large $\rightarrow$}};

  %\foreach \from/\to in {n1/n2,n2/n3,n3/n4,n4/n5,n5/n6}
    %\draw[line width=0.5mm]  (\from) -- (\to);	
    \path[line width=0.5mm,red] (nx) edge[bend left=0] node [above] {$\Gamma_i$} (ny);
		\path[line width=0.5mm,red] (nx) edge[bend left=0] node [left] {$\Gamma_i$} (nd);
	  \path[line width=0.5mm,blue] (ny) edge[bend left=0] node [right] {$\Gamma_j$} (nd);
	
	\node[shape=circle,draw=black,minimum size = 16pt,inner sep=0.5pt] (nx1) at (6,1) {$a$};
  \node[shape=circle,draw=black,minimum size = 16pt,inner sep=0.5pt] (ny1) at (9,1)  {$b$};
  \node[shape=circle,draw=black,minimum size = 16pt,inner sep=0.5pt] (nd1) at (7.5,3.5)  {{\tiny $2d$}};
	\node[shape=circle,minimum size = 24pt,inner sep=0.3pt] (m4) at (12,2.2) {{\large $: \;(\Gamma_1,\dots,\Gamma_d)^{(a,b,2d)}$}};
	
	\node[shape=circle,minimum size = 24pt,inner sep=0.3pt] (m4) at (7.5,0) {{\large $or$}};

  %\foreach \from/\to in {n1/n2,n2/n3,n3/n4,n4/n5,n5/n6}
    %\draw[line width=0.5mm]  (\from) -- (\to);	
    \path[line width=0.5mm,red] (nx1) edge[bend left=0] node [above] {$\Gamma_i$} (ny1);
		\path[line width=0.5mm,red] (ny1) edge[bend left=0] node [right] {$\Gamma_i$} (nd1);
	  \path[line width=0.5mm,blue] (nx1) edge[bend left=0] node [left] {$\Gamma_j$} (nd1);
		
	\node[shape=circle,draw=black,minimum size = 16pt,inner sep=0.5pt] (nx2) at (6,-3.5) {$a$};
  \node[shape=circle,draw=black,minimum size = 16pt,inner sep=0.5pt] (ny2) at (9,-3.5)  {$b$};
  \node[shape=circle,draw=black,minimum size = 16pt,inner sep=0.5pt] (nd2) at (7.5,-1)  {{\tiny $2d$}};

  %\foreach \from/\to in {n1/n2,n2/n3,n3/n4,n4/n5,n5/n6}
    %\draw[line width=0.5mm]  (\from) -- (\to);	
    \path[line width=0.5mm,blue] (ny2) edge[bend left=0] node [above] {$\Gamma_j$} (nx2);
		\path[line width=0.5mm,red] (ny2) edge[bend left=0] node [right] {$\Gamma_i$} (nd2);
	  \path[line width=0.5mm,red] (nx2) edge[bend left=0] node [left] {$\Gamma_i$} (nd2);
		\node[shape=circle,minimum size = 24pt,inner sep=0.3pt] (m4) at (11.5,-2.2) {{\large $:$ Not cycle-free }};

%\node[state,above of=B1] (C1) {$C_1$};
%baseline=(a.center)]%{circle,fill=black}]
	%\tikzset{VertexStyle/.style = {shape = circle,fill = black,minimum size = 9mm,inner sep=2pt}}

\end{tikzpicture}
\caption{Finding the action of the involution $(a,b,2d)$}\label{figc2}
\end{figure}
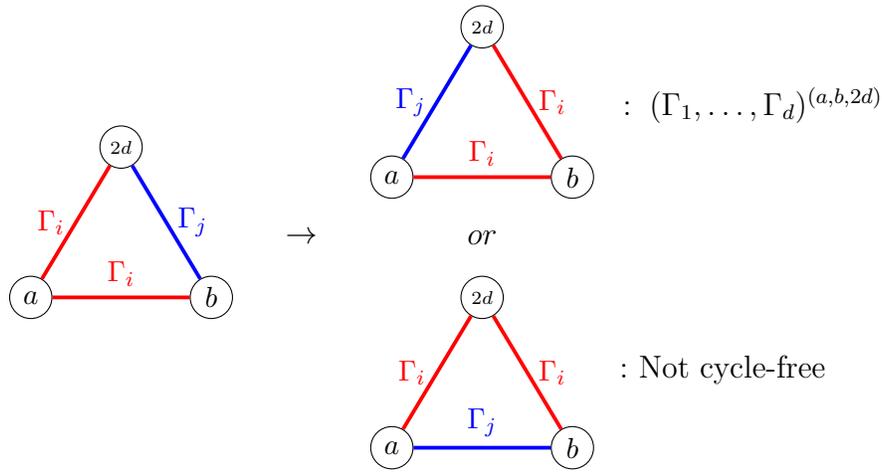

Similarly, if $1\leq a\neq b\leq 2d-2$ and $1\leq i<j\leq d-1$ such that $(a,2d-1)\in \Gamma_i$ and $(b,2d-1)\in \Gamma_j$, then up to a  sequence of involutions $(x,y,z)$, we may switch the edge $(a,2d-1)\in \Gamma_i$ with the edge $(b,2d-1)\in \Gamma_j$ without changing our partition on any of the other edges adjacent to $2d-1$ or $2d$ (including the graph $\Gamma_d$).

Next we will prove the existence of a cycle-free $d$-partition $(\Psi_1,\dots,\Psi_d)$ that is involution equivalent  to $(\Gamma_1,\dots,\Gamma_d)$ and has the property that the edges connected to the vertices $2d-1$ and $2d$ match the edges in the partition $E_d$. Indeed, we already know that $\Gamma_d$ coincides with with the graph $\Omega_d^{(d)}$. From Lemma \ref{lemmagd} we know that  for each $1\leq i\leq d-1$ there exist unique $1\leq u_i, v_i\leq 2d-2$ such that $(u_i,2d-1), (v_i,2d)\in E(\Gamma_i)$. Because the graph $\Omega_d^{(d)}$ accounts for all the edges of the form $(2t,2d-1)$ and $(2t-1,2d)$ it follows that $u_i$ is odd and $v_i$ is even.

Fix $1\leq i\leq d-1$ and let $(2a-1,2d-1)\in E(\Gamma_i)$. If $a\neq i$ then there exist $1\leq b\leq d-1$ such that $(2i-1,2d-1)\in E(\Gamma_b)$. From the above claim we know that up to involution equivalence (that does not change any of the other edges adjacent to $2d-1$ or $2d$) we may switch the edge $(2a-1,2d-1)\in E(\Gamma_i)$ with the edge $(2i-1,2d-1)\in E(\Gamma_b)$. This means that up to involution equivalence we may  assume  that  $(2i-1,2d-1)\in E(\Gamma_i)$  for all $1\leq i\leq d-1$.  A similar argument shows that  up to involution equivalence we may assume that $(2i,2d)\in E(\Gamma_i)$. In other words, up to involution equivalence the cycle-free $d$-partition $(\Gamma_1,\dots,\Gamma_d)$ coincides with $E_d$ on all edges connected to the vertices $2d-1$ and $2d$.
\end{proof}

We are now ready to prove the main result of this paper.

\begin{theorem} Suppose that the twin-star hypothesis $\mathcal{TS}(i)$ holds true for all $1\leq i\leq d$. Then Conjecture \ref{mainconj} holds true for $d$. \label{mainth}
\end{theorem}
\begin{proof} We will prove the statement by induction. Because Conjecture \ref{mainconj} is true for $d=2$ we have the first step of induction.

Assume that the statement is true for all integers smaller or equal to $d-1$. Take $(\Gamma_1,\dots,\Gamma_d)$ a cycle-free $d$-partition of $K_{2d}$. Since we assume  the twin-star hypothesis $\mathcal{TS}(d)$ holds true, we know that $(\Gamma_1,\dots,\Gamma_d)$ is weakly equivalent to $(\Psi_1,\dots,\Psi_d)$ such that the graph $\Psi_d$ is the twin-star graph $TS_d$ with the labeling of the graph $\Omega_d^{(d)}$ from  Figure \ref{fig31} (i.e. the last graph of the  $d$-partition  $E_d$).

%From Lemma \ref{lemmagd} we know that for each $1\leq i\leq d-1$ the graph $\Gamma_i$ has exactly one edge adjacent to the vertex $2d-1$ and exactly one edge adjacent to the vertex $2d$.

If we remove the vertices $2d-1$ and $2d$ from the  graph $K_{2d}$ the cycle-free $d$-partition $(\Psi_1,\dots,\Psi_d)$ of $K_{2d}$ will induce a cycle-free $(d-1)$-partition $(\Delta_1,\dots,\Delta_{d-1})$ of the complete graph $K_{2d-2}$. From the induction hypothesis we know that the cycle free $(d-1)$-partition $(\Delta_1,\dots,\Delta_{d-1})$ is involution equivalent to $E_{d-1}$.
From Lemma \ref{lemmare2d} we know that up to involution equivalence to the partition $(\Psi_1,\dots,\Psi_d)$ we can rearrange the edges connected to the vertex $2d-1$ and $2d$  such that they coincide with those of the partition $E_d$.

Moreover, notice that if we start with the cycle-free $d$-partition $E_d$ and remove the vertices $2d-1$ and $2d$ from the graph $K_{2d}$ then the induced cycle-free $(d-1)$-partition of $K_{2d-2}$ is  the cycle-free $(d-1)$-partition $E_{d-1}$. To summarize, the two partitions $(\Psi_1,\dots,\Psi_d)$ and $E_d$ coincide on the edges connected to vertices $2d-1$ and $2d$, and their restrictions to $K_{2d-2}$ are involution equivalent. So, from Remark \ref{remlocal} we get that $(\Psi_1,\dots,\Psi_d)$ is involution equivalent to $E_d$, which implies that $(\Gamma_1,\dots,\Gamma_d)$ is weakly equivalent to $E_d$.  Finally, from Lemma \ref{lemma3eiw} we get that $(\Gamma_1,\dots,\Gamma_d)$ is involution  equivalent to $E_d$, proving our statement.
\end{proof}

Before we move to the case $d=4$, we will prove one more general result about the cycle-free $d$-partitions of $K_{2d}$. We will need the following definitions.
\begin{definition} If $\Gamma$ is a tree then the diameter of $\Gamma$ is the length of the longest path in $\Gamma$. If $\Gamma$ is a tree and $x\in V(\Gamma)$ we say that $x$ is a branch vertex of $\Gamma$ if there are at least three distinct edges $(x,u)\in E(\Gamma)$ (i.e. the degree of $x$ is at least $3$).
\end{definition}

%\begin{example} The line graph $I_{2d}$ from Figure \ref{fig301} has diameter $2d-1$, and it has no branch vertices. The twin-star graph $TS_{d}$ from Figure \ref{fig3} has diameter $3$, and for $d\geq 3$ the only branch vertices are $1$ and $2$. \end{example}

\begin{theorem} Let $(\Gamma_1,\dots, \Gamma_d)$ a cycle-free $d$-partition of $K_{2d}$. Then there exists $(\Upsilon_1,\dots,\Upsilon_d)$ a cycle-free $d$-partition of $K_{2d}$ that is involution equivalent to $(\Gamma_1,\dots, \Gamma_d)$ and the graph $\Upsilon_d$ is isomorphic to the path graph $I_{2d}$. \label{theI2d}
\end{theorem}
\begin{proof} First notice that since $\Gamma_d$ is a tree with $2d-1$ edges if its diameter is $2d-1$ then $\Gamma_d$ is isomorphic to $I_{2d}$, which is what we want.

So let us assume that the diameter of $\Gamma_d$ is  $t<2d-1$. Take a path $(v_0,v_1,\dots,v_t)$ that is a diameter of the tree $\Gamma_d$, and consider a branch vertex $v_a$ along this diameter (such a vertex exists since $t<2d-1$). This means that $0<a<t$ and there exist a vertex $w\notin \{v_0,v_1,\dots,v_t\}$ such that $(v_a,w)\in E(\Gamma_d)$. We may assume that $v_a$ is the closest branch vertex to $v_t$ (i.e. $t-a$ is minimum among all possible branch vertices along the path $(v_0,v_1,\dots,v_t)$).  We will show that after applying a certain involution we can either increase  the diameter of $\Gamma_d$, or keep the same diameter but move the branch point closer to $v_t$.

Case I: If $t-a=1$, then $a=t-1$. Because $(v_0,v_1,\dots,v_{t-1},v_t)$ is a longest path in $\Gamma_d$  we must have that $(v_a,w)=(v_{t-1},w)$ is a leaf in $\Gamma_d$ (otherwise we would have a longer path in $\Gamma_d$). We can see from Figure \ref{figI1} that when we use the involution $(v_{t-1},v_t,w)$ on the cycle-free $d$-partition $(\Gamma_1,\dots, \Gamma_d)$ we create a cycle-free $d$-partition $$(\Theta_1,\dots,\Theta_d)=(\Gamma_1,\dots, \Gamma_d)^{(v_{t-1},v_t,w)},$$ with the property that $\Theta_d$ has the diameter of length  $t+1$.
Indeed, on the left hand side of Figure \ref{figI1} we have the graph $\Gamma_d$ that has the diameter $(v_0,v_1,\dots,v_{t-2},v_{t-1},v_t)$ with a branch point at $v_{t-1}$, and the edge $(w,v_t)\in E(\Gamma_j)$. After applying the involution $(v_{t-1},v_t,w)$ we have two options, either the top right graph $\Theta_d$ that has a  path $(v_0,v_1,\dots,v_{t-2},v_{t-1},v_t,w)$ of length $t+1$, or the bottom right graph $\Theta_d$ that has the path $(v_0,v_1,\dots,v_{t-2},v_{t-1},w,v_t)$ of length $t+1$. In both cases the diameter of $\Theta_d$ has increased to  $t+1$.
\begin{figure}[H]
\centering
\begin{tikzpicture}
  [scale=0.7,auto=left]%,every node/.style={shape = circle, draw, fill = white,minimum size = 5pt, inner sep=0.3pt}]%baseline=(a.center)]%{circle,fill=black}]
	%\tikzset{VertexStyle/.style = {shape = circle,fill = black,minimum size = 9mm,inner sep=2pt}}
	%\node[shape=circle,draw=black,minimum size = 14pt,inner sep=0.3pt] (n1) at (0,1) {$a$};
	\node[] (nq) at (-3,0) {};
	\node[shape=circle,draw=black,minimum size = 24pt,inner sep=0.5pt] (nt) at (-2,0) {$ v_{t-2}$};
  \node[shape=circle,draw=black,minimum size = 24pt,inner sep=0.5pt] (nx) at (1,0) {$v_{t-1}$};
  \node[shape=circle,draw=black,minimum size = 24pt,inner sep=0.5pt] (ny) at (4,0)  {$v_t$};
  \node[shape=circle,draw=black,minimum size = 24pt,inner sep=0.5pt] (nd) at (2.5,2.5)  { $w$};

\node[shape=circle,minimum size = 24pt,inner sep=0.3pt] (m4) at (6.5,0) {{\large $\rightarrow$}};

  %\foreach \from/\to in {n1/n2,n2/n3,n3/n4,n4/n5,n5/n6}
    %\draw[line width=0.5mm]  (\from) -- (\to);	
		\path[line width=0.5mm,red] (nt) edge[bend left=0] node [left] {$\Gamma_d$} (nq);
		\path[line width=0.5mm,red] (nx) edge[bend left=0] node [above] {$\Gamma_d$} (nt);
    \path[line width=0.5mm,red] (nx) edge[bend left=0] node [above] {$\Gamma_d$} (ny);
		\path[line width=0.5mm,red] (nx) edge[bend left=0] node [left] {$\Gamma_d$} (nd);
	  \path[line width=0.5mm,blue] (ny) edge[bend left=0] node [right] {$\Gamma_j$} (nd);
			\node[] (nq1) at (8,2) {};
	\node[shape=circle,draw=black,minimum size = 24pt,inner sep=0.5pt] (nt1) at (9,2) {$ v_{t-2}$};
	\node[shape=circle,draw=black,minimum size = 24pt,inner sep=0.5pt] (nx1) at (12,2) {$ v_{t-1}$};
  \node[shape=circle,draw=black,minimum size = 24pt,inner sep=0.5pt] (ny1) at (15,2)  {$v_t$};
  \node[shape=circle,draw=black,minimum size = 24pt,inner sep=0.5pt] (nd1) at (13.5,4.5)  { $w$};

	\node[shape=circle,minimum size = 24pt,inner sep=0.3pt] (m4) at (10.5,0) {{\large $or$}};

  %\foreach \from/\to in {n1/n2,n2/n3,n3/n4,n4/n5,n5/n6}
    %\draw[line width=0.5mm]  (\from) -- (\to);	
		\path[line width=0.5mm,red] (nt1) edge[bend left=0] node [left] {$\Theta_d$} (nq1);
			\path[line width=0.5mm,red] (nx1) edge[bend left=0] node [above] {$\Theta_d$} (nt1);
    \path[line width=0.5mm,red] (nx1) edge[bend left=0] node [above] {$\Theta_d$} (ny1);
		\path[line width=0.5mm,red] (ny1) edge[bend left=0] node [right] {$\Theta_d$} (nd1);
	  \path[line width=0.5mm,blue] (nx1) edge[bend left=0] node [left] {$\Theta_j$} (nd1);
	\node[] (nq2) at (8,-3) {};
	\node[shape=circle,draw=black,minimum size = 24pt,inner sep=0.5pt] (nt2) at (9,-3) {$v_{t-2}$};
	\node[shape=circle,draw=black,minimum size = 24pt,inner sep=0.5pt] (nx2) at (12,-3) {$v_{t-1}$};
  \node[shape=circle,draw=black,minimum size = 24pt,inner sep=0.5pt] (ny2) at (15,-3)  {$v_t$};
  \node[shape=circle,draw=black,minimum size = 24pt,inner sep=0.5pt] (nd2) at (13.5,-0.5)  { $w$};

  %\foreach \from/\to in {n1/n2,n2/n3,n3/n4,n4/n5,n5/n6}
    %\draw[line width=0.5mm]  (\from) -- (\to);	
		\path[line width=0.5mm,red] (nt2) edge[bend left=0] node [left] {$\Theta_d$} (nq2);
		\path[line width=0.5mm,red] (nx2) edge[bend left=0] node [above] {$\Theta_d$} (nt2);
    \path[line width=0.5mm,blue] (ny2) edge[bend left=0] node [above] {$\Theta_j$} (nx2);
		\path[line width=0.5mm,red] (ny2) edge[bend left=0] node [right] {$\Theta_d$} (nd2);
	  \path[line width=0.5mm,red] (nx2) edge[bend left=0] node [left] {$\Theta_d$} (nd2);

%\node[state,above of=B1] (C1) {$C_1$};
%baseline=(a.center)]%{circle,fill=black}]
	%\tikzset{VertexStyle/.style = {shape = circle,fill = black,minimum size = 9mm,inner sep=2pt}}
\end{tikzpicture}
\caption{The action of the involution $(v_{t-1},v_t,w)$ on $(\Gamma_1,\dots,\Gamma_d)$}\label{figI1}
\end{figure}
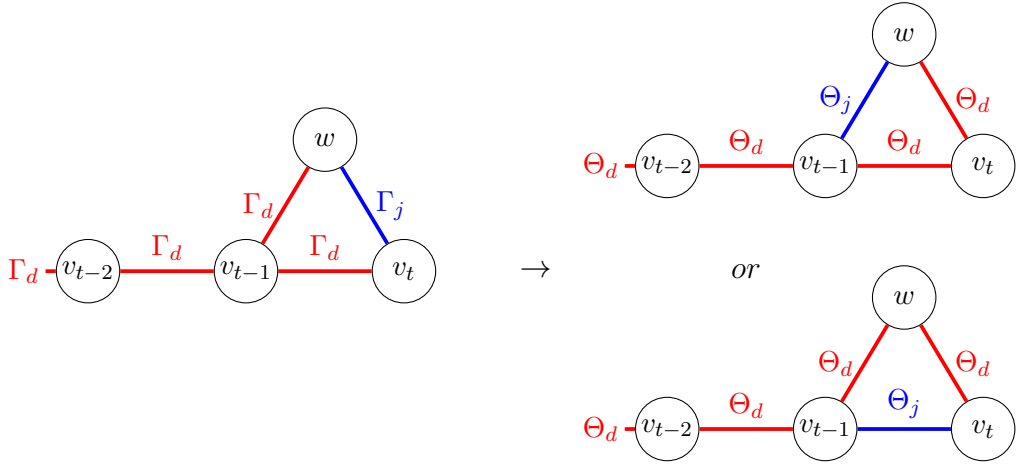

Case II: If $t-a=s\geq 2$, then $a+1<t$. We can see in Figure \ref{figI2} that after applying the involution $(v_a,v_{a+1},w)$ to the $d$-partition $(\Gamma_1,\dots, \Gamma_d)$ we create a cycle-free $d$-partition
$$(\Theta_1,\dots,\Theta_d)=(\Gamma_1,\dots, \Gamma_d)^{(v_{a},v_{a+1},w)},$$ such that $\Theta_d$ has diameter $t+1$, or $\Theta_d$ has diameter $t$ but the branch vertex is at distance $s-1$ from $v_t$. Indeed, on the left hand side of Figure \ref{figI2} we have the graph $\Gamma_d$ with the diameter $(v_0,v_1,\dots,v_{a},v_{a+1},\dots,v_t)$ that has a  branch point at $v_{a}$ which is at distance $s\geq 2$ from $v_t$, $(v_a,w)\in E(\Gamma_d)$ and  $(w,v_{a+1})\in E(\Gamma_j)$. After applying the involution $(v_a,v_{a+1},w)$  we have two options. The first one is the top right graph $\Theta_d$ from Figure \ref{figI2} that still contains the path $(v_0,v_1,\dots,v_{a-1},v_a,v_{a+1},\dots,v_t)$ of length $t$ but has a branch point at $v_{a+1}$ which is at distance $t-(a+1)=s-1$ to $v_t$. The second option is the bottom right graph $\Theta_d$ that has a path $(v_0,v_1,\dots,v_{a-1},v_{a},w,v_{a+1}\dots, v_{t})$ of length $t+1$.

And so, after applying the involution $(v_a,v_{a+1},w)$ either we increase the diameter of $\Theta_d$ to $t+1$, or we keep the diameter of $\Theta_d$ at $t$, but move the branch point closer to $v_t$.
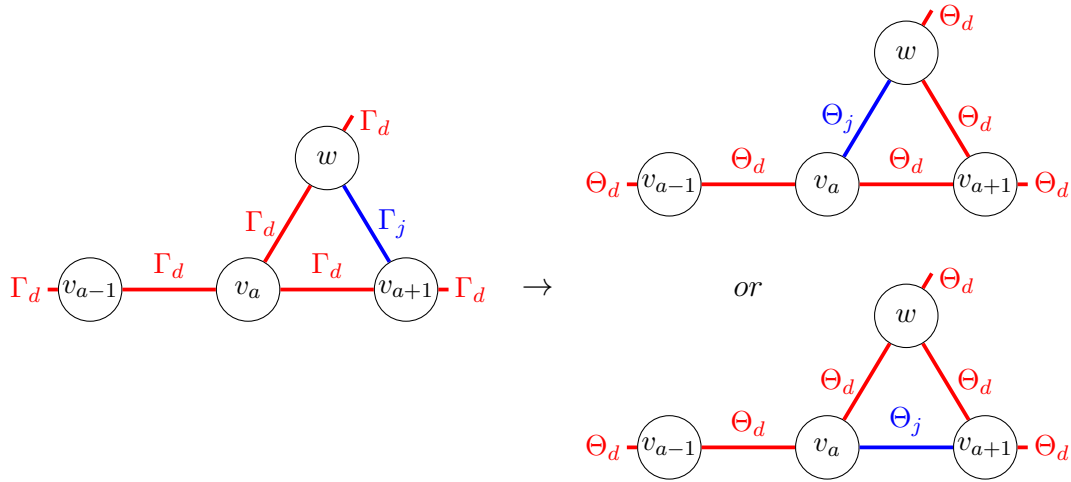
\begin{figure}[H]
\centering
\begin{tikzpicture}
  [scale=0.7,auto=left]%,every node/.style={shape = circle, draw, fill = white,minimum size = 5pt, inner sep=0.3pt}]%baseline=(a.center)]%{circle,fill=black}]
	%\tikzset{VertexStyle/.style = {shape = circle,fill = black,minimum size = 9mm,inner sep=2pt}}
	%\node[shape=circle,draw=black,minimum size = 14pt,inner sep=0.3pt] (n1) at (0,1) {$a$};
	\node[] (np) at (5,0) {};
	\node[] (nq) at (-3,0) {};
	\node[] (nr) at (3.1,3.5) {};
	\node[shape=circle,draw=black,minimum size = 24pt,inner sep=0.5pt] (nt) at (-2,0) {$ v_{a-1}$};
  \node[shape=circle,draw=black,minimum size = 24pt,inner sep=0.5pt] (nx) at (1,0) {$v_{a}$};
  \node[shape=circle,draw=black,minimum size = 24pt,inner sep=0.5pt] (ny) at (4,0)  {$v_{a+1}$};
  \node[shape=circle,draw=black,minimum size = 24pt,inner sep=0.5pt] (nd) at (2.5,2.5)  { $w$};

\node[shape=circle,minimum size = 24pt,inner sep=0.3pt] (m4) at (6.5,0) {{\large $\rightarrow$}};

  %\foreach \from/\to in {n1/n2,n2/n3,n3/n4,n4/n5,n5/n6}
    %\draw[line width=0.5mm]  (\from) -- (\to);	
		\path[line width=0.5mm,red] (ny) edge[bend left=0] node [right] {$\Gamma_d$} (np);
		\path[line width=0.5mm,red] (nd) edge[bend left=0] node [right] {$\Gamma_d$} (nr);
		\path[line width=0.5mm,red] (nt) edge[bend left=0] node [left] {$\Gamma_d$} (nq);
		\path[line width=0.5mm,red] (nx) edge[bend left=0] node [above] {$\Gamma_d$} (nt);
    \path[line width=0.5mm,red] (nx) edge[bend left=0] node [above] {$\Gamma_d$} (ny);
		\path[line width=0.5mm,red] (nx) edge[bend left=0] node [left] {$\Gamma_d$} (nd);
	  \path[line width=0.5mm,blue] (ny) edge[bend left=0] node [right] {$\Gamma_j$} (nd);
			\node[] (np1) at (16,2) {};
		  \node[] (nr1) at (14.1,5.5) {};
			\node[] (nq1) at (8,2) {};
	\node[shape=circle,draw=black,minimum size = 24pt,inner sep=0.5pt] (nt1) at (9,2) {$ v_{a-1}$};
	\node[shape=circle,draw=black,minimum size = 24pt,inner sep=0.5pt] (nx1) at (12,2) {$ v_{a}$};
  \node[shape=circle,draw=black,minimum size = 24pt,inner sep=0.5pt] (ny1) at (15,2)  {$v_{a+1}$};
  \node[shape=circle,draw=black,minimum size = 24pt,inner sep=0.5pt] (nd1) at (13.5,4.5)  { $w$};

	\node[shape=circle,minimum size = 24pt,inner sep=0.3pt] (m4) at (10.5,0) {{\large $or$}};

  %\foreach \from/\to in {n1/n2,n2/n3,n3/n4,n4/n5,n5/n6}
    %\draw[line width=0.5mm]  (\from) -- (\to);	
		\path[line width=0.5mm,red] (ny1) edge[bend left=0] node [right] {$\Theta_d$} (np1);
		\path[line width=0.5mm,red] (nt1) edge[bend left=0] node [left] {$\Theta_d$} (nq1);
			\path[line width=0.5mm,red] (nx1) edge[bend left=0] node [above] {$\Theta_d$} (nt1);
    \path[line width=0.5mm,red] (nx1) edge[bend left=0] node [above] {$\Theta_d$} (ny1);
		\path[line width=0.5mm,red] (ny1) edge[bend left=0] node [right] {$\Theta_d$} (nd1);
	  \path[line width=0.5mm,blue] (nx1) edge[bend left=0] node [left] {$\Theta_j$} (nd1);
		\path[line width=0.5mm,red] (nd1) edge[bend left=0] node [right] {$\Theta_d$} (nr1);
		\node[] (np2) at (16,-3) {};
		 \node[] (nr2) at (14.1,0.5) {};
	\node[] (nq2) at (8,-3) {};
	\node[shape=circle,draw=black,minimum size = 24pt,inner sep=0.5pt] (nt2) at (9,-3) {$v_{a-1}$};
	\node[shape=circle,draw=black,minimum size = 24pt,inner sep=0.5pt] (nx2) at (12,-3) {$v_{a}$};
  \node[shape=circle,draw=black,minimum size = 24pt,inner sep=0.5pt] (ny2) at (15,-3)  {$v_{a+1}$};
  \node[shape=circle,draw=black,minimum size = 24pt,inner sep=0.5pt] (nd2) at (13.5,-0.5)  { $w$};

  %\foreach \from/\to in {n1/n2,n2/n3,n3/n4,n4/n5,n5/n6}
    %\draw[line width=0.5mm]  (\from) -- (\to);	
		\path[line width=0.5mm,red] (ny2) edge[bend left=0] node [right] {$\Theta_d$} (np2);
		\path[line width=0.5mm,red] (nt2) edge[bend left=0] node [left] {$\Theta_d$} (nq2);
		\path[line width=0.5mm,red] (nx2) edge[bend left=0] node [above] {$\Theta_d$} (nt2);
    \path[line width=0.5mm,blue] (ny2) edge[bend left=0] node [above] {$\Theta_j$} (nx2);
		\path[line width=0.5mm,red] (ny2) edge[bend left=0] node [right] {$\Theta_d$} (nd2);
	  \path[line width=0.5mm,red] (nx2) edge[bend left=0] node [left] {$\Theta_d$} (nd2);
		\path[line width=0.5mm,red] (nd2) edge[bend left=0] node [right] {$\Theta_d$} (nr2);

%\node[state,above of=B1] (C1) {$C_1$};
%baseline=(a.center)]%{circle,fill=black}]
	%\tikzset{VertexStyle/.style = {shape = circle,fill = black,minimum size = 9mm,inner sep=2pt}}
\end{tikzpicture}
\caption{The action of the involution $(v_{a},v_{a+1},w)$}\label{figI2}
\end{figure}
Combining these two cases we get our statement.

\end{proof}

\begin{remark} Notice that $I_{2d}$ is the tree with the largest possible diameter in $K_{2d}$. One can easily show that the star graph $S_{2d}$ cannot appear as one of the graphs $\Gamma_i$ of a cycle-free $d$-partition $(\Gamma_1,\dots,\Gamma_d)$ of $K_{2d}$. This means that   the twin-star graph $TS_{d}$ has the smallest diameter among all possible graphs $\Gamma_i$ that can appear in a  cycle-free $d$-partition $(\Gamma_1,\dots,\Gamma_d)$ of $K_{2d}$. So, in a certain way, Theorem \ref{theI2d} is the opposite limit case of the twin-star hypothesis.
\end{remark}

\begin{remark} One would like to reverse the process from Theorem \ref{theI2d}. For example, start with a partition $(\Gamma_1,\dots,\Gamma_d)$ such that $\Gamma_d$ is $I_{2d}$, and use involution equivalence to get a partition where $\Gamma_d$ is equal to the twin-star $TS_{d}$. Unfortunately, as noticed in Remark \ref{localglobal}, the definition of the involution $(x,y,z)$ is not local, and so it's not clear if this process can be done.
\end{remark}

%%%%%%%%%%%%%%%%%%%%%%%%%%%%%%%%%%%

\section{The case $d=4$}

\label{section4}

In this section we show that Conjecture \ref{mainconj} is true for $d=4$. It is known from \cite{h} that there are $23$ isomorphism types of trees with $8$ vertices. For the convenience of the reader and in order to have a consistent notation we list them in Appendix \ref{appendix2}.

The proof has two distinct steps. First, in Subsection \ref{subsection41}, we  show  that every cycle-free $4$-partition $(\Gamma_1,\Gamma_2,\Gamma_3,\Gamma_4)$ of the complete graph $K_8$ is weakly equivalent with a cycle-free $4$-partition  $(\Delta_1,\Delta_2,\Delta_3,\Delta_4)$ such that $\Delta_4$ is the graph  $T_{19}$.  Next, in Subsection \ref{subsection42}, we show that if $\Gamma_4=T_{19}$ then the partition is weakly equivalent with a cycle-free $4$-partition  $(\Delta_1,\Delta_2,\Delta_3,\Delta_4)$ such that $\Delta_4$ is the  twin star graph $TS_4=T_{22}$ (see Appendix \ref{appendix2}). This shows that the twin-star hypothesis $\mathcal{TS}(4)$ holds true, and so Conjecture \ref{mainconj} is true for $d=4$.

\subsection{Reduction to  $T_{19}$}
\label{subsection41}

Recall from Theorem \ref{theI2d} that every cycle-free $4$-partition $(\Gamma_1,\Gamma_2,\Gamma_3,\Gamma_4)$ of the complete graph $K_8$ is weakly equivalent with a cycle-free $4$-partition  $(\Delta_1,\Delta_2,\Delta_3,\Delta_4)$ such that $\Delta_4$ is the path graph  $I_8$ (i.e. $T_1$ in Appendix \ref{appendix2}). We need the  following results.

\begin{lemma} \label{lemma161920}
Let $(\Gamma_1,\Gamma_2,\Gamma_3,\Gamma_4)$ be a cycle-free $4$-partition of the complete graph $K_8$  such that $\Gamma_4$ is the path graph  $I_8$. Then $(\Gamma_1,\Gamma_2,\Gamma_3,\Gamma_4)$ is weakly equivalent to a cycle-free $4$-partition $(\Delta_1,\Delta_2,\Delta_3,\Delta_4)$ such that $\Delta_4$ is one of the graphs $T_{16}, T_{19},$ or $T_{20}$ (see Appendix \ref{appendix2}).
\end{lemma}
\begin{proof} Up to weak equivalence we may assume that $\Gamma_4=I_8$ with the labeling from Figure \ref{I8T23}. We will present a proof by picture. For example, in Figure \ref{I8T23} we see that  using the involution $(5,6,7)$, our cycle free 4-partition with $\Gamma_4=I_8$ can be reduced to a cycle-free $4$-partition $(\Delta_1, \Delta_2, \Delta_3,\Delta_4)$  with  $\Delta_4=T_2$ or $\Delta_4=T_3$. To avoid pictures that are too busy, all edges that appear belong to $\Gamma_4$ or $\Delta_4$ except those that have an explicit different label (like $\Gamma_3$ or $\Delta_3$ in Figure \ref{I8T23}).

\begin{figure}[h]
\centering
\begin{tikzpicture}
  [scale=0.9,auto=left]%,every node/.style={shape = circle, draw, fill = white,minimum size = 14pt, inner sep=0.3pt}]%baseline=(a.center)]%{circle,fill=black}]
	%\tikzset{VertexStyle/.style = {shape = circle,fill = black,minimum size = 9mm,inner sep=2pt}}
	
	\node[shape=circle,minimum size = 24pt,inner sep=0.3pt] (m4) at (-2,0) {{\large $\Gamma_4=I_8:$}};
	\node[shape=circle,draw=black,minimum size = 12pt,inner sep=0.3pt] (n1) at (0,0) {$1$};
	\node[shape=circle,draw=black,minimum size = 12pt,inner sep=0.3pt] (n2) at (1,0) {$2$};
	\node[shape=circle,draw=black,minimum size = 12pt,inner sep=0.3pt] (n3) at (2,0) {$3$};
  \node[shape=circle,draw=black,minimum size = 12pt,inner sep=0.3pt] (n4) at (3,0) {$4$};
	\node[shape=circle,draw=black,minimum size = 12pt,inner sep=0.3pt] (n5) at (4,0) {$5$};
  \node[shape=circle,draw=black,minimum size = 12pt,inner sep=0.3pt] (n6) at (5,0) {$6$};
	\node[shape=circle,draw=black,minimum size = 12pt,inner sep=0.3pt] (n7) at (6,0) {$7$};
  \node[shape=circle,draw=black,minimum size = 12pt,inner sep=0.3pt] (n8) at (7,0) {$8$};

\node[shape=circle,minimum size = 24pt,inner sep=0.3pt] (m4) at (-2,-1) {{\large $\downarrow$}};
\path[line width=0.5mm,blue] (n1) edge[bend left=0] node [above] {} (n2);
\path[line width=0.5mm,blue] (n2) edge[bend left=0] node [above] {} (n3);
\path[line width=0.5mm,blue] (n3) edge[bend left=0] node [above] {} (n4);
\path[line width=0.5mm,blue] (n4) edge[bend left=0] node [above] {} (n5);
\path[line width=0.5mm,blue] (n5) edge[bend left=0] node [below] {$\Gamma_4$} (n6);
\path[line width=0.5mm,blue] (n6) edge[bend left=0] node [below] {$\Gamma_4$} (n7);
\path[line width=0.5mm,blue] (n7) edge[bend left=0] node [above] {} (n8);
\path[line width=0.5mm,red,dotted] (n5) edge[bend left=90] node [above] {$\Gamma_3$} (n7);

			\node[shape=circle,minimum size = 24pt,inner sep=0.3pt] (m4) at (-2,-2) {{\large $\Delta_4=T_2:$}};
	\node[shape=circle,draw=black,minimum size = 12pt,inner sep=0.3pt] (n11) at (0,-2) {$1$};
	\node[shape=circle,draw=black,minimum size = 12pt,inner sep=0.3pt] (n21) at (1,-2) {$2$};
	\node[shape=circle,draw=black,minimum size = 12pt,inner sep=0.3pt] (n31) at (2,-2) {$3$};
  \node[shape=circle,draw=black,minimum size = 12pt,inner sep=0.3pt] (n41) at (3,-2) {$4$};
	\node[shape=circle,draw=black,minimum size = 12pt,inner sep=0.3pt] (n51) at (4,-2) {$5$};
  \node[shape=circle,draw=black,minimum size = 12pt,inner sep=0.3pt] (n61) at (5,-2) {$6$};
	\node[shape=circle,draw=black,minimum size = 12pt,inner sep=0.3pt] (n71) at (6,-2) {$7$};
  \node[shape=circle,draw=black,minimum size = 12pt,inner sep=0.3pt] (n81) at (7,-2) {$8$};

\path[line width=0.5mm,blue] (n11) edge[bend left=0] node [above] {} (n21);
\path[line width=0.5mm,blue] (n21) edge[bend left=0] node [above] {} (n31);
\path[line width=0.5mm,blue] (n31) edge[bend left=0] node [above] {} (n41);
\path[line width=0.5mm,blue] (n41) edge[bend left=0] node [above] {} (n51);
\path[line width=0.5mm,red,dotted] (n51) edge[bend left=0] node [below] {$\Delta_3$} (n61);
\path[line width=0.5mm,blue] (n61) edge[bend left=0] node [below] {$\Delta_4$} (n71);
\path[line width=0.5mm,blue] (n71) edge[bend left=0] node [above] {} (n81);
\path[line width=0.5mm,blue] (n51) edge[bend left=90] node [above] {$\Delta_4$} (n71);

\node[shape=circle,minimum size = 24pt,inner sep=0.3pt] (m4) at (-2,-3) {{ or}};

			\node[shape=circle,minimum size = 24pt,inner sep=0.3pt] (m4) at (-2,-4) {{\large $\Delta_4=T_3:$}};
	\node[shape=circle,draw=black,minimum size = 12pt,inner sep=0.3pt] (n111) at (0,-4) {$1$};
	\node[shape=circle,draw=black,minimum size = 12pt,inner sep=0.3pt] (n211) at (1,-4) {$2$};
	\node[shape=circle,draw=black,minimum size = 12pt,inner sep=0.3pt] (n311) at (2,-4) {$3$};
  \node[shape=circle,draw=black,minimum size = 12pt,inner sep=0.3pt] (n411) at (3,-4) {$4$};
	\node[shape=circle,draw=black,minimum size = 12pt,inner sep=0.3pt] (n511) at (4,-4) {$5$};
  \node[shape=circle,draw=black,minimum size = 12pt,inner sep=0.3pt] (n611) at (5,-4) {$6$};
	\node[shape=circle,draw=black,minimum size = 12pt,inner sep=0.3pt] (n711) at (6,-4) {$7$};
  \node[shape=circle,draw=black,minimum size = 12pt,inner sep=0.3pt] (n811) at (7,-4) {$8$};

\path[line width=0.5mm,blue] (n111) edge[bend left=0] node [above] {} (n211);
\path[line width=0.5mm,blue] (n211) edge[bend left=0] node [above] {} (n311);
\path[line width=0.5mm,blue] (n311) edge[bend left=0] node [above] {} (n411);
\path[line width=0.5mm,blue] (n411) edge[bend left=0] node [above] {} (n511);
\path[line width=0.5mm,blue] (n511) edge[bend left=0] node [below] {$\Delta_4$} (n611);
\path[line width=0.5mm,red,dotted] (n611) edge[bend left=0] node [below] {$\Delta_3$} (n711);
\path[line width=0.5mm,blue] (n711) edge[bend left=0] node [above] {} (n811);
\path[line width=0.5mm,blue] (n511) edge[bend left=90] node [above] {$\Delta_4$} (n711);

\end{tikzpicture}
\caption{Reduction of $I_8=T_1$ to $T_2$ or $T_3$ \label{I8T23}}
\end{figure}
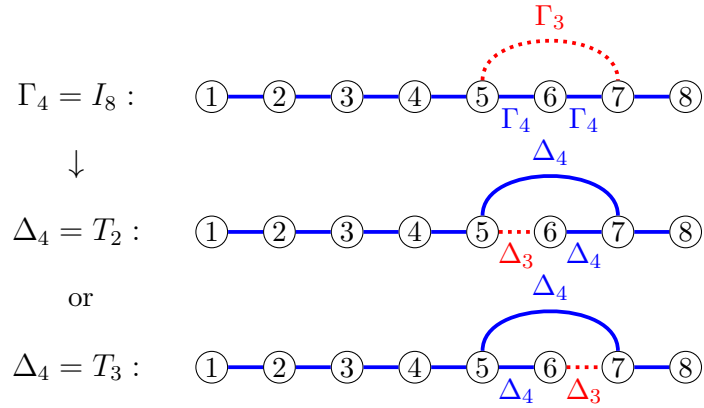

Next, in Figure \ref{T2T617} we see that if $\Gamma_4=T_2$ then using the involution $(4,5,6)$ we can reduce our partition  to a cycle-free $4$-partition with $\Delta_4=T_6$ or $\Delta_4=T_{17}$.
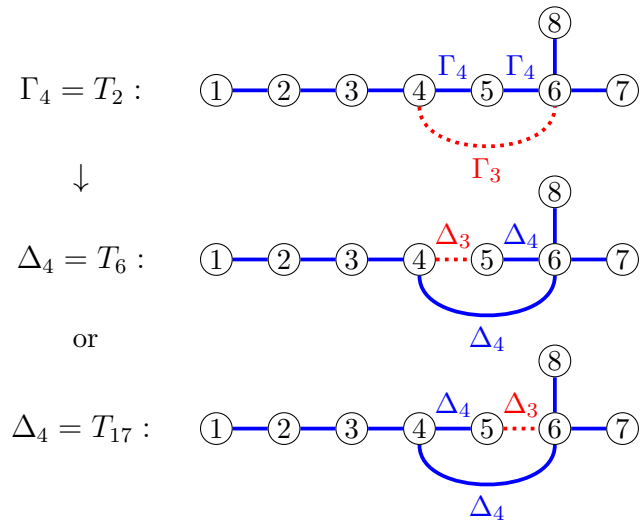
\begin{figure}[H]
\centering
\begin{tikzpicture}
  [scale=0.9,auto=left]%,every node/.style={shape = circle, draw, fill = white,minimum size = 14pt, inner sep=0.3pt}]%baseline=(a.center)]%{circle,fill=black}]
	%\tikzset{VertexStyle/.style = {shape = circle,fill = black,minimum size = 9mm,inner sep=2pt}}
	
	\node[shape=circle,minimum size = 24pt,inner sep=0.3pt] (m4) at (-2,0) {{\large $\Gamma_4=T_2:$}};
	\node[shape=circle,draw=black,minimum size = 12pt,inner sep=0.3pt] (n1) at (0,0) {$1$};
	\node[shape=circle,draw=black,minimum size = 12pt,inner sep=0.3pt] (n2) at (1,0) {$2$};
	\node[shape=circle,draw=black,minimum size = 12pt,inner sep=0.3pt] (n3) at (2,0) {$3$};
  \node[shape=circle,draw=black,minimum size = 12pt,inner sep=0.3pt] (n4) at (3,0) {$4$};
	\node[shape=circle,draw=black,minimum size = 12pt,inner sep=0.3pt] (n5) at (4,0) {$5$};
  \node[shape=circle,draw=black,minimum size = 12pt,inner sep=0.3pt] (n6) at (5,0) {$6$};
	\node[shape=circle,draw=black,minimum size = 12pt,inner sep=0.3pt] (n7) at (6,0) {$7$};
  \node[shape=circle,draw=black,minimum size = 12pt,inner sep=0.3pt] (n8) at (5,1) {$8$};

\node[shape=circle,minimum size = 24pt,inner sep=0.3pt] (m4) at (-2,-1.3) {{\large $\downarrow$}};
\path[line width=0.5mm,blue] (n1) edge[bend left=0] node [above] {} (n2);
\path[line width=0.5mm,blue] (n2) edge[bend left=0] node [above] {} (n3);
\path[line width=0.5mm,blue] (n3) edge[bend left=0] node [above] {} (n4);
\path[line width=0.5mm,blue] (n4) edge[bend left=0] node [above] {$\Gamma_4$} (n5);
\path[line width=0.5mm,blue] (n5) edge[bend left=0] node [above] {$\Gamma_4$} (n6);
\path[line width=0.5mm,blue] (n6) edge[bend left=0] node [below] {} (n7);
\path[line width=0.5mm,blue] (n6) edge[bend left=0] node [above] {} (n8);
\path[line width=0.5mm,red,dotted] (n4) edge[bend right=90] node [below] {$\Gamma_3$} (n6);
	
	\node[shape=circle,minimum size = 24pt,inner sep=0.3pt] (m41) at (-2,-2.5) {{\large $\Delta_4=T_6:$}};
	\node[shape=circle,draw=black,minimum size = 12pt,inner sep=0.3pt] (n11) at (0,-2.5) {$1$};
	\node[shape=circle,draw=black,minimum size = 12pt,inner sep=0.3pt] (n21) at (1,-2.5) {$2$};
	\node[shape=circle,draw=black,minimum size = 12pt,inner sep=0.3pt] (n31) at (2,-2.5) {$3$};
  \node[shape=circle,draw=black,minimum size = 12pt,inner sep=0.3pt] (n41) at (3,-2.5) {$4$};
	\node[shape=circle,draw=black,minimum size = 12pt,inner sep=0.3pt] (n51) at (4,-2.5) {$5$};
  \node[shape=circle,draw=black,minimum size = 12pt,inner sep=0.3pt] (n61) at (5,-2.5) {$6$};
	\node[shape=circle,draw=black,minimum size = 12pt,inner sep=0.3pt] (n71) at (6,-2.5) {$7$};
  \node[shape=circle,draw=black,minimum size = 12pt,inner sep=0.3pt] (n81) at (5,-1.5) {$8$};

\path[line width=0.5mm,blue] (n11) edge[bend left=0] node [above] {} (n21);
\path[line width=0.5mm,blue] (n21) edge[bend left=0] node [above] {} (n31);
\path[line width=0.5mm,blue] (n31) edge[bend left=0] node [above] {} (n41);
\path[line width=0.5mm,red,dotted] (n41) edge[bend left=0] node [above] {$\Delta_3$} (n51);
\path[line width=0.5mm,blue] (n51) edge[bend left=0] node [above] {$\Delta_4$} (n61);
\path[line width=0.5mm,blue] (n61) edge[bend left=0] node [below] {} (n71);
\path[line width=0.5mm,blue] (n61) edge[bend left=0] node [above] {} (n81);
\path[line width=0.5mm,blue] (n41) edge[bend right=90] node [below] {$\Delta_4$} (n61);

\node[shape=circle,minimum size = 24pt,inner sep=0.3pt] (m42) at (-2,-3.7) {{ or}};

\node[shape=circle,minimum size = 24pt,inner sep=0.3pt] (m4) at (-2,-5) {{\large $\Delta_4=T_{17}:$}};
	\node[shape=circle,draw=black,minimum size = 12pt,inner sep=0.3pt] (n12) at (0,-5) {$1$};
	\node[shape=circle,draw=black,minimum size = 12pt,inner sep=0.3pt] (n22) at (1,-5) {$2$};
	\node[shape=circle,draw=black,minimum size = 12pt,inner sep=0.3pt] (n32) at (2,-5) {$3$};
  \node[shape=circle,draw=black,minimum size = 12pt,inner sep=0.3pt] (n42) at (3,-5) {$4$};
	\node[shape=circle,draw=black,minimum size = 12pt,inner sep=0.3pt] (n52) at (4,-5) {$5$};
  \node[shape=circle,draw=black,minimum size = 12pt,inner sep=0.3pt] (n62) at (5,-5) {$6$};
	\node[shape=circle,draw=black,minimum size = 12pt,inner sep=0.3pt] (n72) at (6,-5) {$7$};
  \node[shape=circle,draw=black,minimum size = 12pt,inner sep=0.3pt] (n82) at (5,-4) {$8$};

\path[line width=0.5mm,blue] (n12) edge[bend left=0] node [above] {} (n22);
\path[line width=0.5mm,blue] (n22) edge[bend left=0] node [above] {} (n32);
\path[line width=0.5mm,blue] (n32) edge[bend left=0] node [above] {} (n42);
\path[line width=0.5mm,blue] (n42) edge[bend left=0] node [above] {$\Delta_4$} (n52);
\path[line width=0.5mm,red,dotted] (n52) edge[bend left=0] node [above] {$\Delta_3$} (n62);
\path[line width=0.5mm,blue] (n62) edge[bend left=0] node [below] {} (n72);
\path[line width=0.5mm,blue] (n62) edge[bend left=0] node [above] {} (n82);
\path[line width=0.5mm,blue] (n42) edge[bend right=90] node [below] {$\Delta_4$} (n62);

\end{tikzpicture}
\caption{Reduction of $T_2$ to $T_6$ or $T_{17}$ \label{T2T617}}
\end{figure}

In Figure \ref{T3T1416} we see that if $\Gamma_4=T_3$ then using the involution $(2,3,4)$ we can reduce it to a cycle-free $4$-partition with $\Delta_4=T_{14}$ or $\Delta_4=T_{16}$.

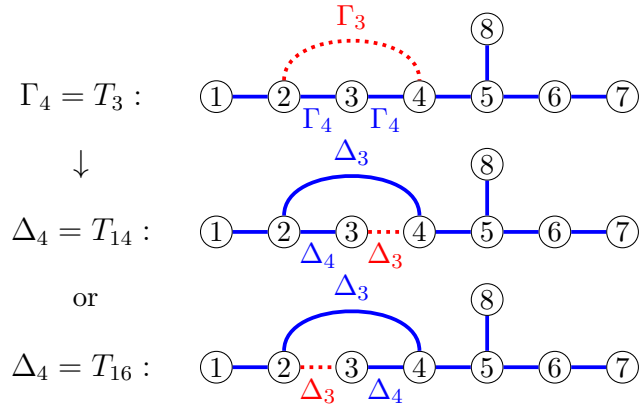
\begin{figure}[H]
\centering
\begin{tikzpicture}
  [scale=0.9,auto=left]%,every node/.style={shape = circle, draw, fill = white,minimum size = 14pt, inner sep=0.3pt}]%baseline=(a.center)]%{circle,fill=black}]
	%\tikzset{VertexStyle/.style = {shape = circle,fill = black,minimum size = 9mm,inner sep=2pt}}
	
	\node[shape=circle,minimum size = 24pt,inner sep=0.3pt] (m4) at (-2,0) {{\large $\Gamma_4=T_3:$}};
	\node[shape=circle,draw=black,minimum size = 12pt,inner sep=0.3pt] (n1) at (0,0) {$1$};
	\node[shape=circle,draw=black,minimum size = 12pt,inner sep=0.3pt] (n2) at (1,0) {$2$};
	\node[shape=circle,draw=black,minimum size = 12pt,inner sep=0.3pt] (n3) at (2,0) {$3$};
  \node[shape=circle,draw=black,minimum size = 12pt,inner sep=0.3pt] (n4) at (3,0) {$4$};
	\node[shape=circle,draw=black,minimum size = 12pt,inner sep=0.3pt] (n5) at (4,0) {$5$};
  \node[shape=circle,draw=black,minimum size = 12pt,inner sep=0.3pt] (n6) at (5,0) {$6$};
	\node[shape=circle,draw=black,minimum size = 12pt,inner sep=0.3pt] (n7) at (6,0) {$7$};
  \node[shape=circle,draw=black,minimum size = 12pt,inner sep=0.3pt] (n8) at (4,1) {$8$};

\node[shape=circle,minimum size = 24pt,inner sep=0.3pt] (m4) at (-2,-1) {{\large $\downarrow$}};
\path[line width=0.5mm,blue] (n1) edge[bend left=0] node [above] {} (n2);
\path[line width=0.5mm,blue] (n2) edge[bend left=0] node [below] {$\Gamma_4$} (n3);
\path[line width=0.5mm,blue] (n3) edge[bend left=0] node [below] {$\Gamma_4$} (n4);
\path[line width=0.5mm,blue] (n4) edge[bend left=0] node [above] {} (n5);
\path[line width=0.5mm,blue] (n5) edge[bend left=0] node [above] {} (n6);
\path[line width=0.5mm,blue] (n6) edge[bend left=0] node [below] {} (n7);
\path[line width=0.5mm,blue] (n5) edge[bend left=0] node [above] {} (n8);
\path[line width=0.5mm,red,dotted] (n2) edge[bend left=90] node [above] {$\Gamma_3$} (n4);

	\node[shape=circle,minimum size = 24pt,inner sep=0.3pt] (m41) at (-2,-4) {{\large $\Delta_4=T_{16}:$}};
	\node[shape=circle,draw=black,minimum size = 12pt,inner sep=0.3pt] (n11) at (0,-4) {$1$};
	\node[shape=circle,draw=black,minimum size = 12pt,inner sep=0.3pt] (n21) at (1,-4) {$2$};
	\node[shape=circle,draw=black,minimum size = 12pt,inner sep=0.3pt] (n31) at (2,-4) {$3$};
  \node[shape=circle,draw=black,minimum size = 12pt,inner sep=0.3pt] (n41) at (3,-4) {$4$};
	\node[shape=circle,draw=black,minimum size = 12pt,inner sep=0.3pt] (n51) at (4,-4) {$5$};
  \node[shape=circle,draw=black,minimum size = 12pt,inner sep=0.3pt] (n61) at (5,-4) {$6$};
	\node[shape=circle,draw=black,minimum size = 12pt,inner sep=0.3pt] (n71) at (6,-4) {$7$};
  \node[shape=circle,draw=black,minimum size = 12pt,inner sep=0.3pt] (n81) at (4,-3) {$8$};

\path[line width=0.5mm,blue] (n11) edge[bend left=0] node [above] {} (n21);
\path[line width=0.5mm,red,dotted] (n21) edge[bend left=0] node [below] {$\Delta_3$} (n31);
\path[line width=0.5mm,blue] (n31) edge[bend left=0] node [below] {$\Delta_4$} (n41);
\path[line width=0.5mm,blue] (n41) edge[bend left=0] node [above] {} (n51);
\path[line width=0.5mm,blue] (n51) edge[bend left=0] node [above] {} (n61);
\path[line width=0.5mm,blue] (n61) edge[bend left=0] node [below] {} (n71);
\path[line width=0.5mm,blue] (n51) edge[bend left=0] node [above] {} (n81);
\path[line width=0.5mm,blue] (n21) edge[bend left=90] node [above] {$\Delta_3$} (n41);

\node[shape=circle,minimum size = 24pt,inner sep=0.3pt] (m42) at (-2,-3) {{ or}};

	\node[shape=circle,minimum size = 24pt,inner sep=0.3pt] (m41) at (-2,-2) {{\large $\Delta_4=T_{14}:$}};
	\node[shape=circle,draw=black,minimum size = 12pt,inner sep=0.3pt] (n111) at (0,-2) {$1$};
	\node[shape=circle,draw=black,minimum size = 12pt,inner sep=0.3pt] (n211) at (1,-2) {$2$};
	\node[shape=circle,draw=black,minimum size = 12pt,inner sep=0.3pt] (n311) at (2,-2) {$3$};
  \node[shape=circle,draw=black,minimum size = 12pt,inner sep=0.3pt] (n411) at (3,-2) {$4$};
	\node[shape=circle,draw=black,minimum size = 12pt,inner sep=0.3pt] (n511) at (4,-2) {$5$};
  \node[shape=circle,draw=black,minimum size = 12pt,inner sep=0.3pt] (n611) at (5,-2) {$6$};
	\node[shape=circle,draw=black,minimum size = 12pt,inner sep=0.3pt] (n711) at (6,-2) {$7$};
  \node[shape=circle,draw=black,minimum size = 12pt,inner sep=0.3pt] (n811) at (4,-1) {$8$};

\path[line width=0.5mm,blue] (n111) edge[bend left=0] node [above] {} (n211);
\path[line width=0.5mm,blue] (n211) edge[bend left=0] node [below] {$\Delta_4$} (n311);
\path[line width=0.5mm,red,dotted] (n311) edge[bend left=0] node [below] {$\Delta_3$} (n411);
\path[line width=0.5mm,blue] (n411) edge[bend left=0] node [above] {} (n511);
\path[line width=0.5mm,blue] (n511) edge[bend left=0] node [above] {} (n611);
\path[line width=0.5mm,blue] (n611) edge[bend left=0] node [below] {} (n711);
\path[line width=0.5mm,blue] (n511) edge[bend left=0] node [above] {} (n811);
\path[line width=0.5mm,blue] (n211) edge[bend left=90] node [above] {$\Delta_3$} (n411);

\end{tikzpicture}
\caption{Reduction of $T_3$ to $T_{14}$ or $T_{16}$ \label{T3T1416}}
\end{figure}

In Figure \ref{T14T1920} we see that if $\Gamma_4=T_{14}$ then using the involution $(2,3,4)$ we can reduce it to a cycle-free $4$-partition with $\Delta_4=T_{19}$ or $\Delta_4=T_{20}$ (which takes care of the $T_3$ branch of our proof).

\begin{figure}[H]
\centering
\begin{tikzpicture}
  [scale=0.9,auto=left]%,every node/.style={shape = circle, draw, fill = white,minimum size = 14pt, inner sep=0.3pt}]%baseline=(a.center)]%{circle,fill=black}]
	%\tikzset{VertexStyle/.style = {shape = circle,fill = black,minimum size = 9mm,inner sep=2pt}}
	
	\node[shape=circle,minimum size = 24pt,inner sep=0.3pt] (m4) at (-2,0) {{\large $\Gamma_4=T_{14}:$}};
	\node[shape=circle,draw=black,minimum size = 12pt,inner sep=0.3pt] (n1) at (0,0) {$1$};
	\node[shape=circle,draw=black,minimum size = 12pt,inner sep=0.3pt] (n2) at (1,0) {$2$};
	\node[shape=circle,draw=black,minimum size = 12pt,inner sep=0.3pt] (n3) at (2,0) {$3$};
  \node[shape=circle,draw=black,minimum size = 12pt,inner sep=0.3pt] (n4) at (3,0) {$4$};
	\node[shape=circle,draw=black,minimum size = 12pt,inner sep=0.3pt] (n5) at (4,0) {$5$};
  \node[shape=circle,draw=black,minimum size = 12pt,inner sep=0.3pt] (n6) at (5,0) {$6$};
	\node[shape=circle,draw=black,minimum size = 12pt,inner sep=0.3pt] (n7) at (1,1) {$7$};
  \node[shape=circle,draw=black,minimum size = 12pt,inner sep=0.3pt] (n8) at (3,1) {$8$};

\node[shape=circle,minimum size = 24pt,inner sep=0.3pt] (m4) at (-2,-1.3) {{\large $\downarrow$}};
\path[line width=0.5mm,blue] (n1) edge[bend left=0] node [above] {} (n2);
\path[line width=0.5mm,blue] (n2) edge[bend left=0] node [above] {$\Gamma_4$} (n3);
\path[line width=0.5mm,blue] (n3) edge[bend left=0] node [above] {$\Gamma_4$} (n4);
\path[line width=0.5mm,blue] (n4) edge[bend left=0] node [above] {} (n5);
\path[line width=0.5mm,blue] (n5) edge[bend left=0] node [above] {} (n6);
\path[line width=0.5mm,blue] (n2) edge[bend left=0] node [below] {} (n7);
\path[line width=0.5mm,blue] (n4) edge[bend left=0] node [above] {} (n8);
\path[line width=0.5mm,red,dotted] (n2) edge[bend right=90] node [below] {$\Gamma_3$} (n4);

	\node[shape=circle,minimum size = 24pt,inner sep=0.3pt] (m41) at (-2,-5) {{\large $\Delta_4=T_{20}:$}};
	\node[shape=circle,draw=black,minimum size = 12pt,inner sep=0.3pt] (n11) at (0,-5) {$1$};
	\node[shape=circle,draw=black,minimum size = 12pt,inner sep=0.3pt] (n21) at (1,-5) {$2$};
	\node[shape=circle,draw=black,minimum size = 12pt,inner sep=0.3pt] (n31) at (2,-5) {$3$};
  \node[shape=circle,draw=black,minimum size = 12pt,inner sep=0.3pt] (n41) at (3,-5) {$4$};
	\node[shape=circle,draw=black,minimum size = 12pt,inner sep=0.3pt] (n51) at (4,-5) {$5$};
  \node[shape=circle,draw=black,minimum size = 12pt,inner sep=0.3pt] (n61) at (5,-5) {$6$};
	\node[shape=circle,draw=black,minimum size = 12pt,inner sep=0.3pt] (n71) at (1,-4) {$7$};
  \node[shape=circle,draw=black,minimum size = 12pt,inner sep=0.3pt] (n81) at (3,-4) {$8$};

\path[line width=0.5mm,blue] (n11) edge[bend left=0] node [above] {} (n21);
\path[line width=0.5mm,blue] (n21) edge[bend left=0] node [above] {$\Delta_4$} (n31);
\path[line width=0.5mm,red,dotted] (n31) edge[bend left=0] node [above] {$\Delta_3$} (n41);
\path[line width=0.5mm,blue] (n41) edge[bend left=0] node [above] {} (n51);
\path[line width=0.5mm,blue] (n51) edge[bend left=0] node [above] {} (n61);
\path[line width=0.5mm,blue] (n21) edge[bend left=0] node [below] {} (n71);
\path[line width=0.5mm,blue] (n41) edge[bend left=0] node [above] {} (n81);
\path[line width=0.5mm,blue] (n21) edge[bend right=90] node [below] {$\Delta_4$} (n41);

\node[shape=circle,minimum size = 24pt,inner sep=0.3pt] (m42) at (-2,-3.7) {{ or}};

		\node[shape=circle,minimum size = 24pt,inner sep=0.3pt] (m41) at (-2,-2.5) {{\large $\Delta_4=T_{19}:$}};
	\node[shape=circle,draw=black,minimum size = 12pt,inner sep=0.3pt] (n111) at (0,-2.5) {$1$};
	\node[shape=circle,draw=black,minimum size = 12pt,inner sep=0.3pt] (n211) at (1,-2.5) {$2$};
	\node[shape=circle,draw=black,minimum size = 12pt,inner sep=0.3pt] (n311) at (2,-2.5) {$3$};
  \node[shape=circle,draw=black,minimum size = 12pt,inner sep=0.3pt] (n411) at (3,-2.5) {$4$};
	\node[shape=circle,draw=black,minimum size = 12pt,inner sep=0.3pt] (n511) at (4,-2.5) {$5$};
  \node[shape=circle,draw=black,minimum size = 12pt,inner sep=0.3pt] (n611) at (5,-2.5) {$6$};
	\node[shape=circle,draw=black,minimum size = 12pt,inner sep=0.3pt] (n711) at (1,-1.5) {$7$};
  \node[shape=circle,draw=black,minimum size = 12pt,inner sep=0.3pt] (n811) at (3,-1.5) {$8$};

\path[line width=0.5mm,blue] (n111) edge[bend left=0] node [above] {} (n211);
\path[line width=0.5mm,red,dotted] (n211) edge[bend left=0] node [above] {$\Delta_3$} (n311);
\path[line width=0.5mm,blue] (n311) edge[bend left=0] node [above] {$\Delta_4$} (n411);
\path[line width=0.5mm,blue] (n411) edge[bend left=0] node [above] {} (n511);
\path[line width=0.5mm,blue] (n511) edge[bend left=0] node [above] {} (n611);
\path[line width=0.5mm,blue] (n211) edge[bend left=0] node [below] {} (n711);
\path[line width=0.5mm,blue] (n411) edge[bend left=0] node [above] {} (n811);
\path[line width=0.5mm,blue] (n211) edge[bend right=90] node [below] {$\Delta_4$} (n411);

\end{tikzpicture}
\caption{Reduction of $T_{14}$ to $T_{19}$ or $T_{20}$ \label{T14T1920}}
\end{figure}
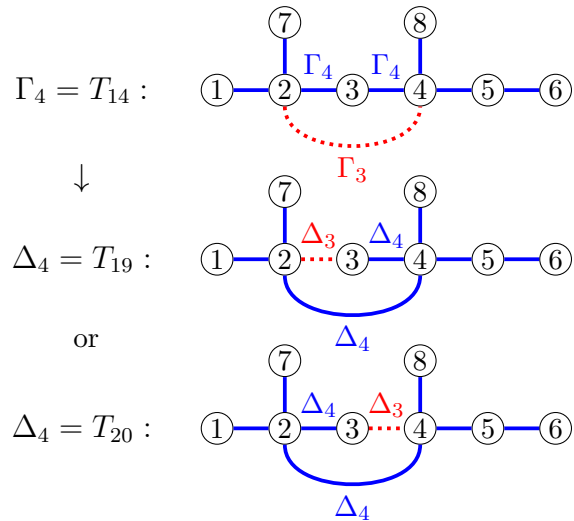

In Figure \ref{T17T1923} we see that if $\Gamma_4=T_{17}$ then using the involution $(3,4,5)$ we can reduce it to a cycle-free $4$-partition with $\Delta_4=T_{19}$ or $\Delta_4=T_{23}$

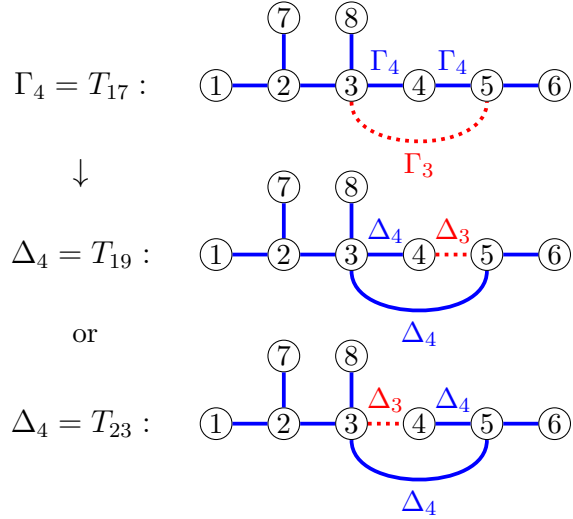
\begin{figure}[H]
\centering
\begin{tikzpicture}
  [scale=0.9,auto=left]%,every node/.style={shape = circle, draw, fill = white,minimum size = 14pt, inner sep=0.3pt}]%baseline=(a.center)]%{circle,fill=black}]
	%\tikzset{VertexStyle/.style = {shape = circle,fill = black,minimum size = 9mm,inner sep=2pt}}
	
	\node[shape=circle,minimum size = 24pt,inner sep=0.3pt] (m4) at (-2,0) {{\large $\Gamma_4=T_{17}:$}};
	\node[shape=circle,draw=black,minimum size = 12pt,inner sep=0.3pt] (n1) at (0,0) {$1$};
	\node[shape=circle,draw=black,minimum size = 12pt,inner sep=0.3pt] (n2) at (1,0) {$2$};
	\node[shape=circle,draw=black,minimum size = 12pt,inner sep=0.3pt] (n3) at (2,0) {$3$};
  \node[shape=circle,draw=black,minimum size = 12pt,inner sep=0.3pt] (n4) at (3,0) {$4$};
	\node[shape=circle,draw=black,minimum size = 12pt,inner sep=0.3pt] (n5) at (4,0) {$5$};
  \node[shape=circle,draw=black,minimum size = 12pt,inner sep=0.3pt] (n6) at (5,0) {$6$};
	\node[shape=circle,draw=black,minimum size = 12pt,inner sep=0.3pt] (n7) at (1,1) {$7$};
  \node[shape=circle,draw=black,minimum size = 12pt,inner sep=0.3pt] (n8) at (2,1) {$8$};

\node[shape=circle,minimum size = 24pt,inner sep=0.3pt] (m4) at (-2,-1.3) {{\large $\downarrow$}};
\path[line width=0.5mm,blue] (n1) edge[bend left=0] node [above] {} (n2);
\path[line width=0.5mm,blue] (n2) edge[bend left=0] node [above] {} (n3);
\path[line width=0.5mm,blue] (n3) edge[bend left=0] node [above] {$\Gamma_4$} (n4);
\path[line width=0.5mm,blue] (n4) edge[bend left=0] node [above] {$\Gamma_4$} (n5);
\path[line width=0.5mm,blue] (n5) edge[bend left=0] node [above] {} (n6);
\path[line width=0.5mm,blue] (n2) edge[bend left=0] node [below] {} (n7);
\path[line width=0.5mm,blue] (n3) edge[bend left=0] node [above] {} (n8);
\path[line width=0.5mm,red,dotted] (n3) edge[bend right=90] node [below] {$\Gamma_3$} (n5);

	\node[shape=circle,minimum size = 24pt,inner sep=0.3pt] (m4) at (-2,-2.5) {{\large $\Delta_4=T_{19}:$}};
	\node[shape=circle,draw=black,minimum size = 12pt,inner sep=0.3pt] (n11) at (0,-2.5) {$1$};
	\node[shape=circle,draw=black,minimum size = 12pt,inner sep=0.3pt] (n21) at (1,-2.5) {$2$};
	\node[shape=circle,draw=black,minimum size = 12pt,inner sep=0.3pt] (n31) at (2,-2.5) {$3$};
  \node[shape=circle,draw=black,minimum size = 12pt,inner sep=0.3pt] (n41) at (3,-2.5) {$4$};
	\node[shape=circle,draw=black,minimum size = 12pt,inner sep=0.3pt] (n51) at (4,-2.5) {$5$};
  \node[shape=circle,draw=black,minimum size = 12pt,inner sep=0.3pt] (n61) at (5,-2.5) {$6$};
	\node[shape=circle,draw=black,minimum size = 12pt,inner sep=0.3pt] (n71) at (1,-1.5) {$7$};
  \node[shape=circle,draw=black,minimum size = 12pt,inner sep=0.3pt] (n81) at (2,-1.5) {$8$};

\path[line width=0.5mm,blue] (n11) edge[bend left=0] node [above] {} (n21);
\path[line width=0.5mm,blue] (n21) edge[bend left=0] node [above] {} (n31);
\path[line width=0.5mm,blue] (n31) edge[bend left=0] node [above] {$\Delta_4$} (n41);
\path[line width=0.5mm,red,dotted] (n41) edge[bend left=0] node [above] {$\Delta_3$} (n51);
\path[line width=0.5mm,blue] (n51) edge[bend left=0] node [above] {} (n61);
\path[line width=0.5mm,blue] (n21) edge[bend left=0] node [below] {} (n71);
\path[line width=0.5mm,blue] (n31) edge[bend left=0] node [above] {} (n81);
\path[line width=0.5mm,blue] (n31) edge[bend right=90] node [below] {$\Delta_4$} (n51);

\node[shape=circle,minimum size = 24pt,inner sep=0.3pt] (m42) at (-2,-3.7) {{ or}};

	\node[shape=circle,minimum size = 24pt,inner sep=0.3pt] (m4) at (-2,-5) {{\large $\Delta_4=T_{23}:$}};
	\node[shape=circle,draw=black,minimum size = 12pt,inner sep=0.3pt] (n111) at (0,-5) {$1$};
	\node[shape=circle,draw=black,minimum size = 12pt,inner sep=0.3pt] (n211) at (1,-5) {$2$};
	\node[shape=circle,draw=black,minimum size = 12pt,inner sep=0.3pt] (n311) at (2,-5) {$3$};
  \node[shape=circle,draw=black,minimum size = 12pt,inner sep=0.3pt] (n411) at (3,-5) {$4$};
	\node[shape=circle,draw=black,minimum size = 12pt,inner sep=0.3pt] (n511) at (4,-5) {$5$};
  \node[shape=circle,draw=black,minimum size = 12pt,inner sep=0.3pt] (n611) at (5,-5) {$6$};
	\node[shape=circle,draw=black,minimum size = 12pt,inner sep=0.3pt] (n711) at (1,-4) {$7$};
  \node[shape=circle,draw=black,minimum size = 12pt,inner sep=0.3pt] (n811) at (2,-4) {$8$};

\path[line width=0.5mm,blue] (n111) edge[bend left=0] node [above] {} (n211);
\path[line width=0.5mm,blue] (n211) edge[bend left=0] node [above] {} (n311);
\path[line width=0.5mm,red,dotted] (n311) edge[bend left=0] node [above] {$\Delta_3$} (n411);
\path[line width=0.5mm,blue] (n411) edge[bend left=0] node [above] {$\Delta_4$} (n511);
\path[line width=0.5mm,blue] (n511) edge[bend left=0] node [above] {} (n611);
\path[line width=0.5mm,blue] (n211) edge[bend left=0] node [below] {} (n711);
\path[line width=0.5mm,blue] (n311) edge[bend left=0] node [above] {} (n811);
\path[line width=0.5mm,blue] (n311) edge[bend right=90] node [below] {$\Delta_4$} (n511);

\end{tikzpicture}
\caption{Reduction of $T_{17}$ to $T_{19}$ or $T_{23}$ \label{T17T1923}}
\end{figure}

In Figure \ref{T23T19}  we see that  if $\Gamma_4=T_{23}$ then using the involution $(2,3,4)$ we can reduce it to a cycle-free $4$-partition with $\Delta_4=T_{19}$.

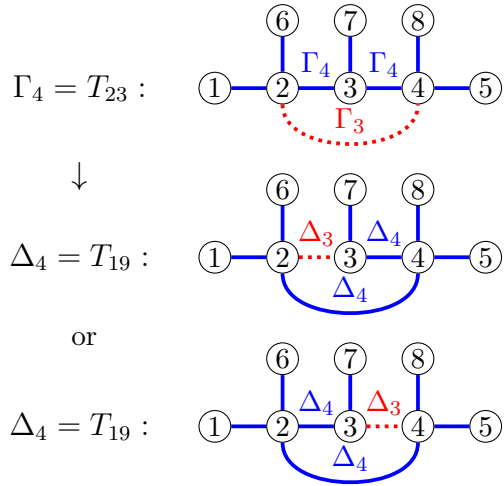
\begin{figure}[H]
\centering
\begin{tikzpicture}
  [scale=0.9,auto=left]%,every node/.style={shape = circle, draw, fill = white,minimum size = 14pt, inner sep=0.3pt}]%baseline=(a.center)]%{circle,fill=black}]
	%\tikzset{VertexStyle/.style = {shape = circle,fill = black,minimum size = 9mm,inner sep=2pt}}
	
	\node[shape=circle,minimum size = 24pt,inner sep=0.3pt] (m4) at (-2,0) {{\large $\Gamma_4=T_{23}:$}};
	\node[shape=circle,draw=black,minimum size = 12pt,inner sep=0.3pt] (n1) at (0,0) {$1$};
	\node[shape=circle,draw=black,minimum size = 12pt,inner sep=0.3pt] (n2) at (1,0) {$2$};
	\node[shape=circle,draw=black,minimum size = 12pt,inner sep=0.3pt] (n3) at (2,0) {$3$};
  \node[shape=circle,draw=black,minimum size = 12pt,inner sep=0.3pt] (n4) at (3,0) {$4$};
	\node[shape=circle,draw=black,minimum size = 12pt,inner sep=0.3pt] (n5) at (4,0) {$5$};
  \node[shape=circle,draw=black,minimum size = 12pt,inner sep=0.3pt] (n6) at (1,1) {$6$};
	\node[shape=circle,draw=black,minimum size = 12pt,inner sep=0.3pt] (n7) at (2,1) {$7$};
  \node[shape=circle,draw=black,minimum size = 12pt,inner sep=0.3pt] (n8) at (3,1) {$8$};

\node[shape=circle,minimum size = 24pt,inner sep=0.3pt] (m4) at (-2,-1.3) {{\large $\downarrow$}};
\path[line width=0.5mm,blue] (n1) edge[bend left=0] node [above] {} (n2);
\path[line width=0.5mm,blue] (n2) edge[bend left=0] node [above] {$\Gamma_4$} (n3);
\path[line width=0.5mm,blue] (n3) edge[bend left=0] node [above] {$\Gamma_4$} (n4);
\path[line width=0.5mm,blue] (n4) edge[bend left=0] node [above] {} (n5);
\path[line width=0.5mm,blue] (n2) edge[bend left=0] node [above] {} (n6);
\path[line width=0.5mm,blue] (n3) edge[bend left=0] node [below] {} (n7);
\path[line width=0.5mm,blue] (n4) edge[bend left=0] node [above] {} (n8);
\path[line width=0.5mm,red,dotted] (n2) edge[bend right=90] node [above] {$\Gamma_3$} (n4);

		\node[shape=circle,minimum size = 24pt,inner sep=0.3pt] (m4) at (-2,-2.5) {{\large $\Delta_4=T_{19}:$}};
	\node[shape=circle,draw=black,minimum size = 12pt,inner sep=0.3pt] (n11) at (0,-2.5) {$1$};
	\node[shape=circle,draw=black,minimum size = 12pt,inner sep=0.3pt] (n21) at (1,-2.5) {$2$};
	\node[shape=circle,draw=black,minimum size = 12pt,inner sep=0.3pt] (n31) at (2,-2.5) {$3$};
  \node[shape=circle,draw=black,minimum size = 12pt,inner sep=0.3pt] (n41) at (3,-2.5) {$4$};
	\node[shape=circle,draw=black,minimum size = 12pt,inner sep=0.3pt] (n51) at (4,-2.5) {$5$};
  \node[shape=circle,draw=black,minimum size = 12pt,inner sep=0.3pt] (n61) at (1,-1.5) {$6$};
	\node[shape=circle,draw=black,minimum size = 12pt,inner sep=0.3pt] (n71) at (2,-1.5) {$7$};
  \node[shape=circle,draw=black,minimum size = 12pt,inner sep=0.3pt] (n81) at (3,-1.5) {$8$};

\path[line width=0.5mm,blue] (n11) edge[bend left=0] node [above] {} (n21);
\path[line width=0.5mm,red,dotted] (n21) edge[bend left=0] node [above] {$\Delta_3$} (n31);
\path[line width=0.5mm,blue] (n31) edge[bend left=0] node [above] {$\Delta_4$} (n41);
\path[line width=0.5mm,blue] (n41) edge[bend left=0] node [above] {} (n51);
\path[line width=0.5mm,blue] (n21) edge[bend left=0] node [above] {} (n61);
\path[line width=0.5mm,blue] (n31) edge[bend left=0] node [below] {} (n71);
\path[line width=0.5mm,blue] (n41) edge[bend left=0] node [above] {} (n81);
\path[line width=0.5mm,blue] (n21) edge[bend right=90] node [above] {$\Delta_4$} (n41);
	
\node[shape=circle,minimum size = 24pt,inner sep=0.3pt] (m42) at (-2,-3.7) {{ or}};

		\node[shape=circle,minimum size = 24pt,inner sep=0.3pt] (m4) at (-2,-5) {{\large $\Delta_4=T_{19}:$}};
	\node[shape=circle,draw=black,minimum size = 12pt,inner sep=0.3pt] (n12) at (0,-5) {$1$};
	\node[shape=circle,draw=black,minimum size = 12pt,inner sep=0.3pt] (n22) at (1,-5) {$2$};
	\node[shape=circle,draw=black,minimum size = 12pt,inner sep=0.3pt] (n32) at (2,-5) {$3$};
  \node[shape=circle,draw=black,minimum size = 12pt,inner sep=0.3pt] (n42) at (3,-5) {$4$};
	\node[shape=circle,draw=black,minimum size = 12pt,inner sep=0.3pt] (n52) at (4,-5) {$5$};
  \node[shape=circle,draw=black,minimum size = 12pt,inner sep=0.3pt] (n62) at (1,-4) {$6$};
	\node[shape=circle,draw=black,minimum size = 12pt,inner sep=0.3pt] (n72) at (2,-4) {$7$};
  \node[shape=circle,draw=black,minimum size = 12pt,inner sep=0.3pt] (n82) at (3,-4) {$8$};

\path[line width=0.5mm,blue] (n12) edge[bend left=0] node [above] {} (n22);
\path[line width=0.5mm,blue] (n22) edge[bend left=0] node [above] {$\Delta_4$} (n32);
\path[line width=0.5mm,red,dotted] (n32) edge[bend left=0] node [above] {$\Delta_3$} (n42);
\path[line width=0.5mm,blue] (n42) edge[bend left=0] node [above] {} (n52);
\path[line width=0.5mm,blue] (n22) edge[bend left=0] node [above] {} (n62);
\path[line width=0.5mm,blue] (n32) edge[bend left=0] node [below] {} (n72);
\path[line width=0.5mm,blue] (n42) edge[bend left=0] node [above] {} (n82);
\path[line width=0.5mm,blue] (n22) edge[bend right=90] node [above] {$\Delta_4$} (n42);
        	
\end{tikzpicture}
\caption{Reduction of $T_{23}$ to $T_{19}$ \label{T23T19}}
\end{figure}

Finally, in Figure \ref{T6T20} using the involution $(3,4,5)$ we see that  any homogeneous cycle-free $4$-partition with  $\Gamma_4=T_{6}$ is involution equivalent to one with $\Delta_4=T_{20}$. Notice that the other graph  from Figure \ref{T6T20} is  $T_9$, which  cannot be part of a cycle-free $4$-partition of $K_8$ since it has a vertex of degree $5$ (and so not all the $\Delta_i$'s can be connected to vertex $5$).

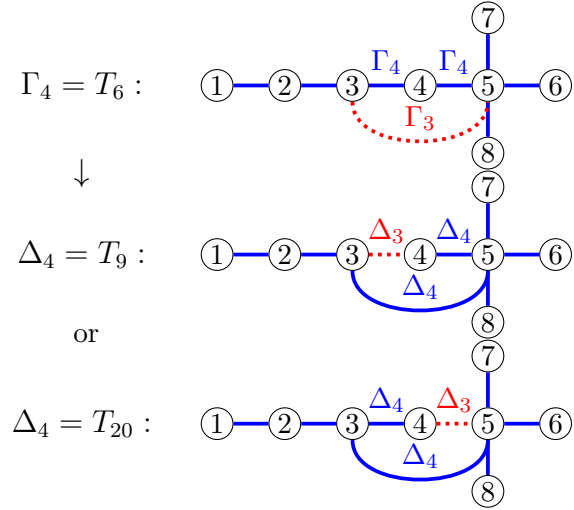
\begin{figure}[H]
\centering
\begin{tikzpicture}
  [scale=0.9,auto=left]%,every node/.style={shape = circle, draw, fill = white,minimum size = 14pt, inner sep=0.3pt}]%baseline=(a.center)]%{circle,fill=black}]
	%\tikzset{VertexStyle/.style = {shape = circle,fill = black,minimum size = 9mm,inner sep=2pt}}
	
	\node[shape=circle,minimum size = 24pt,inner sep=0.3pt] (m4) at (-2,0) {{\large $\Gamma_4=T_{6}:$}};
	\node[shape=circle,draw=black,minimum size = 12pt,inner sep=0.3pt] (n1) at (0,0) {$1$};
	\node[shape=circle,draw=black,minimum size = 12pt,inner sep=0.3pt] (n2) at (1,0) {$2$};
	\node[shape=circle,draw=black,minimum size = 12pt,inner sep=0.3pt] (n3) at (2,0) {$3$};
  \node[shape=circle,draw=black,minimum size = 12pt,inner sep=0.3pt] (n4) at (3,0) {$4$};
	\node[shape=circle,draw=black,minimum size = 12pt,inner sep=0.3pt] (n5) at (4,0) {$5$};
  \node[shape=circle,draw=black,minimum size = 12pt,inner sep=0.3pt] (n6) at (5,0) {$6$};
	\node[shape=circle,draw=black,minimum size = 12pt,inner sep=0.3pt] (n7) at (4,1) {$7$};
  \node[shape=circle,draw=black,minimum size = 12pt,inner sep=0.3pt] (n8) at (4,-1) {$8$};

\node[shape=circle,minimum size = 24pt,inner sep=0.3pt] (m4) at (-2,-1.3) {{\large $\downarrow$}};
\path[line width=0.5mm,blue] (n1) edge[bend left=0] node [above] {} (n2);
\path[line width=0.5mm,blue] (n2) edge[bend left=0] node [above] {} (n3);
\path[line width=0.5mm,blue] (n3) edge[bend left=0] node [above] {$\Gamma_4$} (n4);
\path[line width=0.5mm,blue] (n4) edge[bend left=0] node [above] {$\Gamma_4$} (n5);
\path[line width=0.5mm,blue] (n5) edge[bend left=0] node [above] {} (n6);
\path[line width=0.5mm,blue] (n7) edge[bend left=0] node [below] {} (n5);
\path[line width=0.5mm,blue] (n8) edge[bend left=0] node [above] {} (n5);
\path[line width=0.5mm,red,dotted] (n3) edge[bend right=90] node [above] {$\Gamma_3$} (n5);

	\node[shape=circle,minimum size = 24pt,inner sep=0.3pt] (m42) at (-2,-2.5) {{\large $\Delta_4=T_{9}:$}};
	\node[shape=circle,draw=black,minimum size = 12pt,inner sep=0.3pt] (n11) at (0,-2.5) {$1$};
	\node[shape=circle,draw=black,minimum size = 12pt,inner sep=0.3pt] (n21) at (1,-2.5) {$2$};
	\node[shape=circle,draw=black,minimum size = 12pt,inner sep=0.3pt] (n31) at (2,-2.5) {$3$};
  \node[shape=circle,draw=black,minimum size = 12pt,inner sep=0.3pt] (n41) at (3,-2.5) {$4$};
	\node[shape=circle,draw=black,minimum size = 12pt,inner sep=0.3pt] (n51) at (4,-2.5) {$5$};
  \node[shape=circle,draw=black,minimum size = 12pt,inner sep=0.3pt] (n61) at (5,-2.5) {$6$};
	\node[shape=circle,draw=black,minimum size = 12pt,inner sep=0.3pt] (n71) at (4,-1.5) {$7$};
  \node[shape=circle,draw=black,minimum size = 12pt,inner sep=0.3pt] (n81) at (4,-3.5) {$8$};

\path[line width=0.5mm,blue] (n11) edge[bend left=0] node [above] {} (n21);
\path[line width=0.5mm,blue] (n21) edge[bend left=0] node [above] {} (n31);
\path[line width=0.5mm,red,dotted] (n31) edge[bend left=0] node [above] {$\Delta_3$} (n41);
\path[line width=0.5mm,blue] (n41) edge[bend left=0] node [above] {$\Delta_4$} (n51);
\path[line width=0.5mm,blue] (n51) edge[bend left=0] node [above] {} (n61);
\path[line width=0.5mm,blue] (n71) edge[bend left=0] node [below] {} (n51);
\path[line width=0.5mm,blue] (n81) edge[bend left=0] node [above] {} (n51);
\path[line width=0.5mm,blue] (n31) edge[bend right=90] node [above] {$\Delta_4$} (n51);

\node[shape=circle,minimum size = 24pt,inner sep=0.3pt] (m42) at (-2,-3.7) {{ or}};

			\node[shape=circle,minimum size = 24pt,inner sep=0.3pt] (m42) at (-2,-5) {{\large $\Delta_4=T_{20}:$}};
	\node[shape=circle,draw=black,minimum size = 12pt,inner sep=0.3pt] (n111) at (0,-5) {$1$};
	\node[shape=circle,draw=black,minimum size = 12pt,inner sep=0.3pt] (n211) at (1,-5) {$2$};
	\node[shape=circle,draw=black,minimum size = 12pt,inner sep=0.3pt] (n311) at (2,-5) {$3$};
  \node[shape=circle,draw=black,minimum size = 12pt,inner sep=0.3pt] (n411) at (3,-5) {$4$};
	\node[shape=circle,draw=black,minimum size = 12pt,inner sep=0.3pt] (n511) at (4,-5) {$5$};
  \node[shape=circle,draw=black,minimum size = 12pt,inner sep=0.3pt] (n611) at (5,-5) {$6$};
	\node[shape=circle,draw=black,minimum size = 12pt,inner sep=0.3pt] (n711) at (4,-4) {$7$};
  \node[shape=circle,draw=black,minimum size = 12pt,inner sep=0.3pt] (n811) at (4,-6) {$8$};

\path[line width=0.5mm,blue] (n111) edge[bend left=0] node [above] {} (n211);
\path[line width=0.5mm,blue] (n211) edge[bend left=0] node [above] {} (n311);
\path[line width=0.5mm,blue] (n311) edge[bend left=0] node [above] {$\Delta_4$} (n411);
\path[line width=0.5mm,red,dotted] (n411) edge[bend left=0] node [above] {$\Delta_3$} (n511);
\path[line width=0.5mm,blue] (n511) edge[bend left=0] node [above] {} (n611);
\path[line width=0.5mm,blue] (n711) edge[bend left=0] node [below] {} (n511);
\path[line width=0.5mm,blue] (n811) edge[bend left=0] node [above] {} (n511);
\path[line width=0.5mm,blue] (n311) edge[bend right=90] node [above] {$\Delta_4$} (n511);

\end{tikzpicture}
\caption{Reduction of $T_{6}$ to $T_{20}$ \label{T6T20}}
\end{figure}

To summarize, we have the reduction diagram from Figure \ref{redDia}, which shows that a cycle-free $4$-partition of $K_8$ with $\Gamma_4=I_8$ is weakly equivalent to one where $\Delta_4=T_{16}$, or  $\Delta_4=T_{19}$, or $\Delta_4=T_{20}$, which concludes our proof.

\begin{figure}[H]
\centering
\begin{tikzpicture}
  [scale=0.9,auto=left]%,every node/.style={shape = circle, draw, fill = white,minimum size = 14pt, inner sep=0.3pt}]%baseline=(a.center)]%{circle,fill=black}]
	%\tikzset{VertexStyle/.style = {shape = circle,fill = black,minimum size = 9mm,inner sep=2pt}}
	
		\node[shape=circle,draw=black,minimum size = 24pt,inner sep=0.3pt] (n1) at (0,0) {$T_1$};
	\node[shape=circle,draw=black,minimum size = 24pt,inner sep=0.3pt] (n2) at (-2,-1.5) {$T_2$};
	\node[shape=circle,draw=black,minimum size = 24pt,inner sep=0.3pt] (n3) at (2,-1.5) {$T_3$};
	\node[shape=circle,draw=black,minimum size = 24pt,inner sep=0.3pt] (n6) at (-3,-3) {$T_6$};
	\node[shape=circle,draw=black,minimum size = 24pt,inner sep=0.3pt] (n17) at (-1,-3) {$T_{17}$};
	\node[shape=circle,draw=black,minimum size = 24pt,inner sep=0.3pt] (n14) at (1,-3) {$T_{14}$};
	\node[fill=yellow,shape=circle,draw=black,minimum size = 24pt,inner sep=0.3pt] (n16) at (3,-3) {$T_{16}$};
	\node[shape=circle,draw=black,minimum size = 24pt,inner sep=0.3pt] (n23) at (-2,-4.5) {$T_{23}$};
	\node[fill=yellow,shape=circle,draw=black,minimum size = 24pt,inner sep=0.3pt] (n19) at (0,-4.5) {$T_{19}$};
	\node[fill=yellow,shape=circle,draw=black,minimum size = 24pt,inner sep=0.3pt] (n20) at (2,-4.5) {$T_{20}$};
	
\path[line width=0.5mm,black,->] (n1) edge[bend left=0] node [above] {} (n2);
\path[line width=0.5mm,black,->] (n1) edge[bend left=0] node [above] {} (n3);
\path[line width=0.5mm,black,->] (n2) edge[bend left=0] node [above] {} (n6);
\path[line width=0.5mm,black,->] (n2) edge[bend left=0] node [above] {} (n17);
\path[line width=0.5mm,black,->] (n3) edge[bend left=0] node [above] {} (n14);
\path[line width=0.5mm,black,->] (n3) edge[bend left=0] node [above] {} (n16);
\path[line width=0.5mm,black,->] (n17) edge[bend left=0] node [above] {} (n23);
\path[line width=0.5mm,black,->] (n17) edge[bend left=0] node [above] {} (n19);
\path[line width=0.5mm,black,->] (n14) edge[bend left=0] node [above] {} (n19);
\path[line width=0.5mm,black,->] (n14) edge[bend left=0] node [above] {} (n20);
\path[line width=0.5mm,black,->] (n23) edge[bend left=0] node [above] {} (n19);
\path[line width=0.5mm,black,->] (n6) edge[bend right=120] node [above] {} (n20);
        	
\end{tikzpicture}
\caption{Reduction diagram from $T_1=I_8$ to $T_{16}$, $T_{19}$ or $T_{20}$ \label{redDia}}
\end{figure}
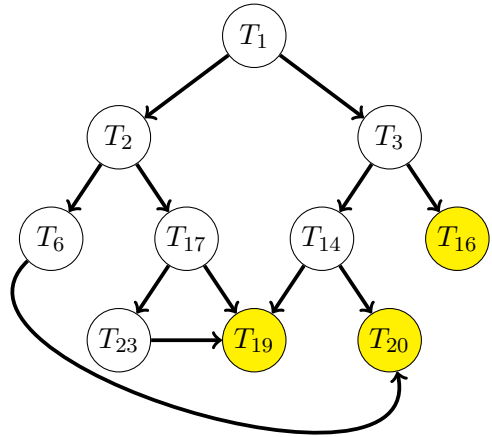

\end{proof}

\begin{remark} Notice that during the proof of Lemma \ref{lemma161920} we showed that if $(\Gamma_1,\Gamma_2,\Gamma_3,\Gamma_4)$ is a cycle-free $4$-partition of the complete graph $K_8$  such that $\Gamma_4$ is the graph $T_{17}$ or $T_{23}$ then it can be reduced to a cycle free $4$-partition $(\Delta_1,\Delta_2,\Delta_3,\Delta_4)$ where  $\Delta_4=T_{19}$ (see also Figure \ref{redDia}).
\label{remark1723}
\end{remark}

\begin{lemma}
Let $(\Gamma_1,\Gamma_2,\Gamma_3,\Gamma_4)$ be a cycle-free $4$-partition of the complete graph $K_8$  such that $\Gamma_4$ is the graph  $T_{16}$. Then $(\Gamma_1,\Gamma_2,\Gamma_3,\Gamma_4)$ is weakly equivalent to a cycle-free $4$-partition $(\Delta_1,\Delta_2,\Delta_3,\Delta_4)$ such that $\Delta_4$ is the graph $T_{19}$.
\end{lemma}
\begin{proof}  Up to weakly equivalence we may assume  that $\Gamma_4=T_{16}$ with the labeling from Figure \ref{T16}. Notice that none of the edges $(1,5)$, $(1,6)$, $(2,5)$ or $(2,6)$ are in $\Gamma_4$. Since our partition has only three other graphs $\Gamma_1$, $\Gamma_2$ and $\Gamma_3$, we know that at least two of these four edges must belong to the same graph. Up to weakly equivalence we may assume that they belong to $\Gamma_3$. There are six different cases.

\begin{figure}[H]
\centering
\begin{tikzpicture}
  [scale=0.9,auto=left]%,every node/.style={shape = circle, draw, fill = white,minimum size = 14pt, inner sep=0.3pt}]%baseline=(a.center)]%{circle,fill=black}]
	%\tikzset{VertexStyle/.style = {shape = circle,fill = black,minimum size = 9mm,inner sep=2pt}}
	
	\node[shape=circle,minimum size = 24pt,inner sep=0.3pt] (m4) at (-2,0) {{\large $\Gamma_4=T_{16}:$}};
	\node[shape=circle,draw=black,minimum size = 12pt,inner sep=0.3pt] (n1) at (0,0) {$1$};
	\node[shape=circle,draw=black,minimum size = 12pt,inner sep=0.3pt] (n2) at (1,0) {$2$};
	\node[shape=circle,draw=black,minimum size = 12pt,inner sep=0.3pt] (n3) at (2,0) {$3$};
  \node[shape=circle,draw=black,minimum size = 12pt,inner sep=0.3pt] (n4) at (3,0) {$4$};
	\node[shape=circle,draw=black,minimum size = 12pt,inner sep=0.3pt] (n5) at (4,0) {$5$};
  \node[shape=circle,draw=black,minimum size = 12pt,inner sep=0.3pt] (n6) at (5,0) {$6$};
	\node[shape=circle,draw=black,minimum size = 12pt,inner sep=0.3pt] (n7) at (2,1) {$7$};
  \node[shape=circle,draw=black,minimum size = 12pt,inner sep=0.3pt] (n8) at (3,1) {$8$};

%\node[shape=circle,minimum size = 24pt,inner sep=0.3pt] (m4) at (-1,-1.8) {{\large $\downarrow$}};
\path[line width=0.5mm,blue] (n1) edge[bend left=0] node [above] {} (n2);
\path[line width=0.5mm,blue] (n2) edge[bend left=0] node [above] {} (n3);
\path[line width=0.5mm,blue] (n3) edge[bend left=0] node [above] {} (n4);
\path[line width=0.5mm,blue] (n4) edge[bend left=0] node [above] {} (n5);
\path[line width=0.5mm,blue] (n5) edge[bend left=0] node [above] {} (n6);
\path[line width=0.5mm,blue] (n7) edge[bend left=0] node [below] {} (n3);
\path[line width=0.5mm,blue] (n8) edge[bend left=0] node [above] {} (n4);
%\path[line width=0.5mm,red,dotted] (n2) edge[bend right=90] node [below] {$\Gamma_3$} (n6);
%\path[line width=0.5mm,red,dotted] (n2) edge[bend right=90] node [above] {$\Gamma_3$} (n5);

\end{tikzpicture}
\caption{The graph $\Gamma_4=T_{16}$ \label{T16}}
\end{figure}
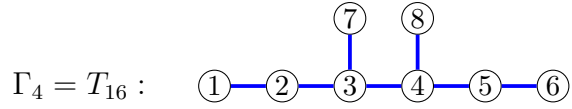

Case I: Suppose that the edges $(2,5)$ and $(2,6)\in E(\Gamma_3)$. From Figure \ref{T16T23} we can see that using the involution $(2,5,6)$  our partition can be reduced to one with $\Delta_4=T_{23}$ (the other possible partition in Figure \ref{T16T23} is not cycle-free). From Remark  \ref{remark1723} we know that this partition can be reduced further to one with $\Delta_4=T_{19}$.

\begin{figure}[H]
\centering
\begin{tikzpicture}
  [scale=0.9,auto=left]%,every node/.style={shape = circle, draw, fill = white,minimum size = 14pt, inner sep=0.3pt}]%baseline=(a.center)]%{circle,fill=black}]
	%\tikzset{VertexStyle/.style = {shape = circle,fill = black,minimum size = 9mm,inner sep=2pt}}
	
	\node[shape=circle,minimum size = 24pt,inner sep=0.3pt] (m4) at (-2,0) {{\large $\Gamma_4=T_{16}:$}};
	\node[shape=circle,draw=black,minimum size = 12pt,inner sep=0.3pt] (n1) at (0,0) {$1$};
	\node[shape=circle,draw=black,minimum size = 12pt,inner sep=0.3pt] (n2) at (1,0) {$2$};
	\node[shape=circle,draw=black,minimum size = 12pt,inner sep=0.3pt] (n3) at (2,0) {$3$};
  \node[shape=circle,draw=black,minimum size = 12pt,inner sep=0.3pt] (n4) at (3,0) {$4$};
	\node[shape=circle,draw=black,minimum size = 12pt,inner sep=0.3pt] (n5) at (4,0) {$5$};
  \node[shape=circle,draw=black,minimum size = 12pt,inner sep=0.3pt] (n6) at (5,0) {$6$};
	\node[shape=circle,draw=black,minimum size = 12pt,inner sep=0.3pt] (n7) at (2,1) {$7$};
  \node[shape=circle,draw=black,minimum size = 12pt,inner sep=0.3pt] (n8) at (3,1) {$8$};

\node[shape=circle,minimum size = 24pt,inner sep=0.3pt] (m4) at (-2,-1.8) {{\large $\downarrow$}};
\path[line width=0.5mm,blue] (n1) edge[bend left=0] node [above] {} (n2);
\path[line width=0.5mm,blue] (n2) edge[bend left=0] node [above] {} (n3);
\path[line width=0.5mm,blue] (n3) edge[bend left=0] node [above] {} (n4);
\path[line width=0.5mm,blue] (n4) edge[bend left=0] node [above] {} (n5);
\path[line width=0.5mm,blue] (n5) edge[bend left=0] node [above] {$\Gamma_4$} (n6);
\path[line width=0.5mm,blue] (n7) edge[bend left=0] node [below] {} (n3);
\path[line width=0.5mm,blue] (n8) edge[bend left=0] node [above] {} (n4);
\path[line width=0.5mm,red,dotted] (n2) edge[bend right=90] node [below] {$\Gamma_3$} (n6);
\path[line width=0.5mm,red,dotted] (n2) edge[bend right=90] node [above] {$\Gamma_3$} (n5);

	\node[shape=circle,minimum size = 24pt,inner sep=0.3pt] (m41) at (-1.9,-3.5) {{\large Not cycle-free :}};
	\node[shape=circle,draw=black,minimum size = 12pt,inner sep=0.3pt] (n11) at (0,-3.5) {$1$};
	\node[shape=circle,draw=black,minimum size = 12pt,inner sep=0.3pt] (n21) at (1,-3.5) {$2$};
	\node[shape=circle,draw=black,minimum size = 12pt,inner sep=0.3pt] (n31) at (2,-3.5) {$3$};
  \node[shape=circle,draw=black,minimum size = 12pt,inner sep=0.3pt] (n41) at (3,-3.5) {$4$};
	\node[shape=circle,draw=black,minimum size = 12pt,inner sep=0.3pt] (n51) at (4,-3.5) {$5$};
  \node[shape=circle,draw=black,minimum size = 12pt,inner sep=0.3pt] (n61) at (5,-3.5) {$6$};
	\node[shape=circle,draw=black,minimum size = 12pt,inner sep=0.3pt] (n71) at (2,-2.5) {$7$};
  \node[shape=circle,draw=black,minimum size = 12pt,inner sep=0.3pt] (n81) at (3,-2.5) {$8$};

\path[line width=0.5mm,blue] (n11) edge[bend left=0] node [above] {} (n21);
\path[line width=0.5mm,blue] (n21) edge[bend left=0] node [above] {} (n31);
\path[line width=0.5mm,blue] (n31) edge[bend left=0] node [above] {} (n41);
\path[line width=0.5mm,blue] (n41) edge[bend left=0] node [above] {} (n51);
\path[line width=0.5mm,red,dotted] (n51) edge[bend left=0] node [above] {$\Delta_3$} (n61);
\path[line width=0.5mm,blue] (n71) edge[bend left=0] node [below] {} (n31);
\path[line width=0.5mm,blue] (n81) edge[bend left=0] node [above] {} (n41);
\path[line width=0.5mm,red,dotted] (n21) edge[bend right=90] node [below] {$\Delta_3$} (n61);
\path[line width=0.5mm,blue] (n21) edge[bend right=90] node [above] {$\Delta_4$} (n51);

\node[shape=circle,minimum size = 24pt,inner sep=0.3pt] (m42) at (-2,-5.2) {{ or}};

	\node[shape=circle,minimum size = 24pt,inner sep=0.3pt] (m41) at (-2,-7) {{\large $\Delta_4=T_{23}:$}};
	\node[shape=circle,draw=black,minimum size = 12pt,inner sep=0.3pt] (n111) at (0,-7) {$1$};
	\node[shape=circle,draw=black,minimum size = 12pt,inner sep=0.3pt] (n211) at (1,-7) {$2$};
	\node[shape=circle,draw=black,minimum size = 12pt,inner sep=0.3pt] (n311) at (2,-7) {$3$};
  \node[shape=circle,draw=black,minimum size = 12pt,inner sep=0.3pt] (n411) at (3,-7) {$4$};
	\node[shape=circle,draw=black,minimum size = 12pt,inner sep=0.3pt] (n511) at (4,-7) {$5$};
  \node[shape=circle,draw=black,minimum size = 12pt,inner sep=0.3pt] (n611) at (5,-7) {$6$};
	\node[shape=circle,draw=black,minimum size = 12pt,inner sep=0.3pt] (n711) at (2,-6) {$7$};
  \node[shape=circle,draw=black,minimum size = 12pt,inner sep=0.3pt] (n811) at (3,-6) {$8$};

\path[line width=0.5mm,blue] (n111) edge[bend left=0] node [above] {} (n211);
\path[line width=0.5mm,blue] (n211) edge[bend left=0] node [above] {} (n311);
\path[line width=0.5mm,blue] (n311) edge[bend left=0] node [above] {} (n411);
\path[line width=0.5mm,blue] (n411) edge[bend left=0] node [above] {} (n511);
\path[line width=0.5mm,red,dotted] (n511) edge[bend left=0] node [above] {$\Delta_3$} (n611);
\path[line width=0.5mm,blue] (n711) edge[bend left=0] node [below] {} (n311);
\path[line width=0.5mm,blue] (n811) edge[bend left=0] node [above] {} (n411);
\path[line width=0.5mm,blue] (n211) edge[bend right=90] node [below] {$\Delta_4$} (n611);
\path[line width=0.5mm,red,dotted] (n211) edge[bend right=90] node [above] {$\Delta_3$} (n511);

\end{tikzpicture}
\caption{Reduction of $T_{16}$ when $(2,5)$, $(2,6)\in E(\Gamma_3)$\label{T16T23}}
\end{figure}
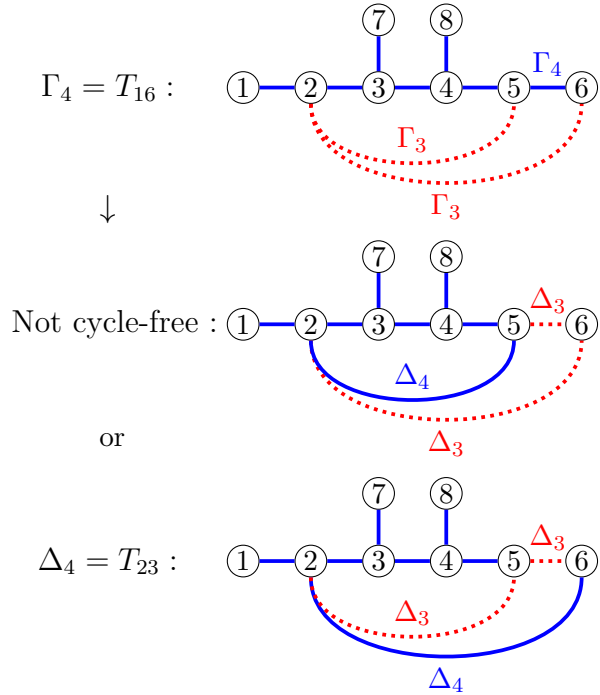

Case II: If  $(1,5)$, $(2,5)\in E(\Gamma_3)$, a  similar argument to Case I shows that  our partition can be reduced to one such that $\Delta_4$ is the graph $T_{19}$.

Case III: If $(1,6)$, $(2,6)\in E(\Gamma_3)$  then from  Figure \ref{T16T17} we can see that using the involution $(1,2,6)$  our partition can be reduced to one with $\Delta_4=T_{17}$. From Remark \ref{remark1723} we know that this partition can be reduced further to one with $\Delta_4=T_{19}$.

\begin{figure}[H]
\centering
\begin{tikzpicture}
  [scale=0.9,auto=left]%,every node/.style={shape = circle, draw, fill = white,minimum size = 14pt, inner sep=0.3pt}]%baseline=(a.center)]%{circle,fill=black}]
	%\tikzset{VertexStyle/.style = {shape = circle,fill = black,minimum size = 9mm,inner sep=2pt}}
	
	\node[shape=circle,minimum size = 24pt,inner sep=0.3pt] (m4) at (-2,0) {{\large $\Gamma_4=T_{16}:$}};
	\node[shape=circle,draw=black,minimum size = 12pt,inner sep=0.3pt] (n1) at (0,0) {$1$};
	\node[shape=circle,draw=black,minimum size = 12pt,inner sep=0.3pt] (n2) at (1,0) {$2$};
	\node[shape=circle,draw=black,minimum size = 12pt,inner sep=0.3pt] (n3) at (2,0) {$3$};
  \node[shape=circle,draw=black,minimum size = 12pt,inner sep=0.3pt] (n4) at (3,0) {$4$};
	\node[shape=circle,draw=black,minimum size = 12pt,inner sep=0.3pt] (n5) at (4,0) {$5$};
  \node[shape=circle,draw=black,minimum size = 12pt,inner sep=0.3pt] (n6) at (5,0) {$6$};
	\node[shape=circle,draw=black,minimum size = 12pt,inner sep=0.3pt] (n7) at (2,1) {$7$};
  \node[shape=circle,draw=black,minimum size = 12pt,inner sep=0.3pt] (n8) at (3,1) {$8$};

\node[shape=circle,minimum size = 24pt,inner sep=0.3pt] (m4) at (-2,-1.8) {{\large $\downarrow$}};
\path[line width=0.5mm,blue] (n1) edge[bend left=0] node [above] {$\Gamma_4$} (n2);
\path[line width=0.5mm,blue] (n2) edge[bend left=0] node [above] {} (n3);
\path[line width=0.5mm,blue] (n3) edge[bend left=0] node [above] {} (n4);
\path[line width=0.5mm,blue] (n4) edge[bend left=0] node [above] {} (n5);
\path[line width=0.5mm,blue] (n5) edge[bend left=0] node [above] {} (n6);
\path[line width=0.5mm,blue] (n7) edge[bend left=0] node [below] {} (n3);
\path[line width=0.5mm,blue] (n8) edge[bend left=0] node [above] {} (n4);
\path[line width=0.5mm,red,dotted] (n2) edge[bend right=90] node [above] {$\Gamma_3$} (n6);
\path[line width=0.5mm,red,dotted] (n1) edge[bend right=90] node [below] {$\Gamma_3$} (n6);

	\node[shape=circle,minimum size = 24pt,inner sep=0.3pt] (m41) at (-2,-3.5) {{\large $\Delta_4=T_{17}$ :}};
	\node[shape=circle,draw=black,minimum size = 12pt,inner sep=0.3pt] (n11) at (0,-3.5) {$1$};
	\node[shape=circle,draw=black,minimum size = 12pt,inner sep=0.3pt] (n21) at (1,-3.5) {$2$};
	\node[shape=circle,draw=black,minimum size = 12pt,inner sep=0.3pt] (n31) at (2,-3.5) {$3$};
  \node[shape=circle,draw=black,minimum size = 12pt,inner sep=0.3pt] (n41) at (3,-3.5) {$4$};
	\node[shape=circle,draw=black,minimum size = 12pt,inner sep=0.3pt] (n51) at (4,-3.5) {$5$};
  \node[shape=circle,draw=black,minimum size = 12pt,inner sep=0.3pt] (n61) at (5,-3.5) {$6$};
	\node[shape=circle,draw=black,minimum size = 12pt,inner sep=0.3pt] (n71) at (2,-2.5) {$7$};
  \node[shape=circle,draw=black,minimum size = 12pt,inner sep=0.3pt] (n81) at (3,-2.5) {$8$};

\path[line width=0.5mm,red,dotted] (n11) edge[bend left=0] node [above] {$\Delta_3$} (n21);
\path[line width=0.5mm,blue] (n21) edge[bend left=0] node [above] {} (n31);
\path[line width=0.5mm,blue] (n31) edge[bend left=0] node [above] {} (n41);
\path[line width=0.5mm,blue] (n41) edge[bend left=0] node [above] {} (n51);
\path[line width=0.5mm,blue] (n51) edge[bend left=0] node [above] {} (n61);
\path[line width=0.5mm,blue] (n71) edge[bend left=0] node [below] {} (n31);
\path[line width=0.5mm,blue] (n81) edge[bend left=0] node [above] {} (n41);
\path[line width=0.5mm,red,dotted] (n21) edge[bend right=90] node [above] {$\Delta_3$} (n61);
\path[line width=0.5mm,blue] (n11) edge[bend right=90] node [below] {$\Delta_4$} (n61);

\node[shape=circle,minimum size = 24pt,inner sep=0.3pt] (m42) at (-2,-5.2) {{ or}};

	\node[shape=circle,minimum size = 24pt,inner sep=0.3pt] (m41) at (-2.4,-7) {{\large Not cycle-free :}};
	\node[shape=circle,draw=black,minimum size = 12pt,inner sep=0.3pt] (n111) at (0,-7) {$1$};
	\node[shape=circle,draw=black,minimum size = 12pt,inner sep=0.3pt] (n211) at (1,-7) {$2$};
	\node[shape=circle,draw=black,minimum size = 12pt,inner sep=0.3pt] (n311) at (2,-7) {$3$};
  \node[shape=circle,draw=black,minimum size = 12pt,inner sep=0.3pt] (n411) at (3,-7) {$4$};
	\node[shape=circle,draw=black,minimum size = 12pt,inner sep=0.3pt] (n511) at (4,-7) {$5$};
  \node[shape=circle,draw=black,minimum size = 12pt,inner sep=0.3pt] (n611) at (5,-7) {$6$};
	\node[shape=circle,draw=black,minimum size = 12pt,inner sep=0.3pt] (n711) at (2,-6) {$7$};
  \node[shape=circle,draw=black,minimum size = 12pt,inner sep=0.3pt] (n811) at (3,-6) {$8$};

\path[line width=0.5mm,red,dotted] (n111) edge[bend left=0] node [above] {$\Delta_3$} (n211);
\path[line width=0.5mm,blue] (n211) edge[bend left=0] node [above] {} (n311);
\path[line width=0.5mm,blue] (n311) edge[bend left=0] node [above] {} (n411);
\path[line width=0.5mm,blue] (n411) edge[bend left=0] node [above] {} (n511);
\path[line width=0.5mm,blue] (n511) edge[bend left=0] node [above] {} (n611);
\path[line width=0.5mm,blue] (n711) edge[bend left=0] node [below] {} (n311);
\path[line width=0.5mm,blue] (n811) edge[bend left=0] node [above] {} (n411);
\path[line width=0.5mm,blue] (n211) edge[bend right=90] node [above] {$\Delta_4$} (n611);
\path[line width=0.5mm,red,dotted] (n111) edge[bend right=90] node [below] {$\Delta_3$} (n611);

\end{tikzpicture}
\caption{Reduction of $T_{16}$ when $(1,6)$, $(2,6)\in E(\Gamma_3)$\label{T16T17}}
\end{figure}
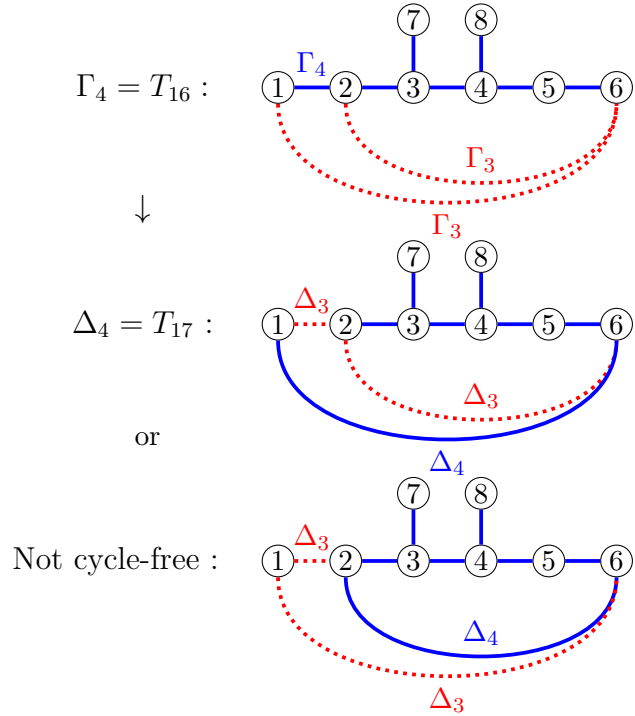

Case IV: If $(1,6)$, $(1,5)\in E(\Gamma_3)$, a similar argument to the one for Case III shows that our partition can be reduced to one such that $\Delta_4$ is the graph $T_{19}$.

Case V:  Assume that $(1,5), (2,6)\in E(\Gamma_3)$, and $(2,5)\in E(\Gamma_2)$. From Figure \ref{T16T23G23} we can see that using the involution $(2,5,6)$ we can either reduce our partition to a partition with $\Delta_4=T_{23}$ (which by  Remark  \ref{remark1723} can be reduced further to one with $\Delta_4=T_{19}$), or we can interchange the edges $(2,5)$ and $(2,6)$.  This gives us that $(1,5), (2,5)\in E(\Delta_3)$, and  $(2,6)\in E(\Delta_2)$, a situation that was already covered in Case II.

\begin{figure}[H]
\centering
\begin{tikzpicture}
  [scale=0.9,auto=left]%,every node/.style={shape = circle, draw, fill = white,minimum size = 14pt, inner sep=0.3pt}]%baseline=(a.center)]%{circle,fill=black}]
	%\tikzset{VertexStyle/.style = {shape = circle,fill = black,minimum size = 9mm,inner sep=2pt}}
	
	\node[shape=circle,minimum size = 24pt,inner sep=0.3pt] (m4) at (-1,0) {{\large $T_{16}$ :}};
	\node[shape=circle,draw=black,minimum size = 12pt,inner sep=0.3pt] (n1) at (0,0) {$1$};
	\node[shape=circle,draw=black,minimum size = 12pt,inner sep=0.3pt] (n2) at (1,0) {$2$};
	\node[shape=circle,draw=black,minimum size = 12pt,inner sep=0.3pt] (n3) at (2,0) {$3$};
  \node[shape=circle,draw=black,minimum size = 12pt,inner sep=0.3pt] (n4) at (3,0) {$4$};
	\node[shape=circle,draw=black,minimum size = 12pt,inner sep=0.3pt] (n5) at (4,0) {$5$};
  \node[shape=circle,draw=black,minimum size = 12pt,inner sep=0.3pt] (n6) at (5,0) {$6$};
	\node[shape=circle,draw=black,minimum size = 12pt,inner sep=0.3pt] (n7) at (2,1) {$7$};
  \node[shape=circle,draw=black,minimum size = 12pt,inner sep=0.3pt] (n8) at (3,1) {$8$};

\node[shape=circle,minimum size = 24pt,inner sep=0.3pt] (m4) at (-1,-1.8) {{\large $\downarrow$}};
\path[line width=0.5mm,blue] (n1) edge[bend left=0] node [above] {} (n2);
\path[line width=0.5mm,blue] (n2) edge[bend left=0] node [above] {} (n3);
\path[line width=0.5mm,blue] (n3) edge[bend left=0] node [above] {} (n4);
\path[line width=0.5mm,blue] (n4) edge[bend left=0] node [above] {} (n5);
\path[line width=0.5mm,blue] (n5) edge[bend left=0] node [above] {$\Gamma_4$} (n6);
\path[line width=0.5mm,blue] (n7) edge[bend left=0] node [below] {} (n3);
\path[line width=0.5mm,blue] (n8) edge[bend left=0] node [above] {} (n4);
\path[line width=0.5mm,red,dotted] (n2) edge[bend right=90] node [below] {$\Gamma_3$} (n6);
\path[line width=0.5mm,green,dash dot] (n2) edge[bend right=90] node [above] {$\Gamma_2$} (n5);
\path[line width=0.5mm,red,dotted] (n1) edge[bend left=90] node [above] {$\Gamma_3$} (n5);

	\node[shape=circle,minimum size = 24pt,inner sep=0.3pt] (m41) at (-2,-3.5) {{\large Not cycle-free :}};
	\node[shape=circle,draw=black,minimum size = 12pt,inner sep=0.3pt] (n11) at (0,-3.5) {$1$};
	\node[shape=circle,draw=black,minimum size = 12pt,inner sep=0.3pt] (n21) at (1,-3.5) {$2$};
	\node[shape=circle,draw=black,minimum size = 12pt,inner sep=0.3pt] (n31) at (2,-3.5) {$3$};
  \node[shape=circle,draw=black,minimum size = 12pt,inner sep=0.3pt] (n41) at (3,-3.5) {$4$};
	\node[shape=circle,draw=black,minimum size = 12pt,inner sep=0.3pt] (n51) at (4,-3.5) {$5$};
  \node[shape=circle,draw=black,minimum size = 12pt,inner sep=0.3pt] (n61) at (5,-3.5) {$6$};
	\node[shape=circle,draw=black,minimum size = 12pt,inner sep=0.3pt] (n71) at (2,-2.5) {$7$};
  \node[shape=circle,draw=black,minimum size = 12pt,inner sep=0.3pt] (n81) at (3,-2.5) {$8$};

\path[line width=0.5mm,blue] (n11) edge[bend left=0] node [above] {} (n21);
\path[line width=0.5mm,blue] (n21) edge[bend left=0] node [above] {} (n31);
\path[line width=0.5mm,blue] (n31) edge[bend left=0] node [above] {} (n41);
\path[line width=0.5mm,blue] (n41) edge[bend left=0] node [above] {} (n51);
\path[line width=0.5mm,green,dash dot] (n51) edge[bend left=0] node [above] {$\Delta_2$} (n61);
\path[line width=0.5mm,blue] (n71) edge[bend left=0] node [below] {} (n31);
\path[line width=0.5mm,blue] (n81) edge[bend left=0] node [above] {} (n41);
\path[line width=0.5mm,red,dotted] (n21) edge[bend right=90] node [below] {$\Delta_3$} (n61);
\path[line width=0.5mm,blue] (n21) edge[bend right=90] node [above] {$\Delta_4$} (n51);
\path[line width=0.5mm,red,dotted] (n11) edge[bend left=90] node [above] {$\Delta_3$} (n51);
	
\node[shape=circle,minimum size = 24pt,inner sep=0.3pt] (m421) at (-1,-5.2) {{ or}};

\node[shape=circle,minimum size = 24pt,inner sep=0.3pt] (m411) at (-2,-7) {{\large Not cycle-free :}};
	\node[shape=circle,draw=black,minimum size = 12pt,inner sep=0.3pt] (n111) at (0,-7) {$1$};
	\node[shape=circle,draw=black,minimum size = 12pt,inner sep=0.3pt] (n211) at (1,-7) {$2$};
	\node[shape=circle,draw=black,minimum size = 12pt,inner sep=0.3pt] (n311) at (2,-7) {$3$};
  \node[shape=circle,draw=black,minimum size = 12pt,inner sep=0.3pt] (n411) at (3,-7) {$4$};
	\node[shape=circle,draw=black,minimum size = 12pt,inner sep=0.3pt] (n511) at (4,-7) {$5$};
  \node[shape=circle,draw=black,minimum size = 12pt,inner sep=0.3pt] (n611) at (5,-7) {$6$};
	\node[shape=circle,draw=black,minimum size = 12pt,inner sep=0.3pt] (n711) at (2,-6) {$7$};
  \node[shape=circle,draw=black,minimum size = 12pt,inner sep=0.3pt] (n811) at (3,-6) {$8$};

\path[line width=0.5mm,blue] (n111) edge[bend left=0] node [above] {} (n211);
\path[line width=0.5mm,blue] (n211) edge[bend left=0] node [above] {} (n311);
\path[line width=0.5mm,blue] (n311) edge[bend left=0] node [above] {} (n411);
\path[line width=0.5mm,blue] (n411) edge[bend left=0] node [above] {} (n511);
\path[line width=0.5mm,red,dotted] (n511) edge[bend left=0] node [above] {$\Delta_3$} (n611);
\path[line width=0.5mm,blue] (n711) edge[bend left=0] node [below] {} (n311);
\path[line width=0.5mm,blue] (n811) edge[bend left=0] node [above] {} (n411);
\path[line width=0.5mm,green,dash dot] (n211) edge[bend right=90] node [below] {$\Delta_2$} (n611);
\path[line width=0.5mm,blue] (n211) edge[bend right=90] node [above] {$\Delta_4$} (n511);
\path[line width=0.5mm,red,dotted] (n111) edge[bend left=90] node [above] {$\Delta_3$} (n511);
	
\node[shape=circle,minimum size = 24pt,inner sep=0.3pt] (m422) at (-1,-8.7) {{ or}};

	\node[shape=circle,minimum size = 24pt,inner sep=0.3pt] (m41) at (-1,-10.5) {{\large $T_{23}$ :}};
	\node[shape=circle,draw=black,minimum size = 12pt,inner sep=0.3pt] (n1112) at (0,-10.5) {$1$};
	\node[shape=circle,draw=black,minimum size = 12pt,inner sep=0.3pt] (n2112) at (1,-10.5) {$2$};
	\node[shape=circle,draw=black,minimum size = 12pt,inner sep=0.3pt] (n3112) at (2,-10.5) {$3$};
  \node[shape=circle,draw=black,minimum size = 12pt,inner sep=0.3pt] (n4112) at (3,-10.5) {$4$};
	\node[shape=circle,draw=black,minimum size = 12pt,inner sep=0.3pt] (n5112) at (4,-10.5) {$5$};
  \node[shape=circle,draw=black,minimum size = 12pt,inner sep=0.3pt] (n6112) at (5,-10.5) {$6$};
	\node[shape=circle,draw=black,minimum size = 12pt,inner sep=0.3pt] (n7112) at (2,-9.5) {$7$};
  \node[shape=circle,draw=black,minimum size = 12pt,inner sep=0.3pt] (n8112) at (3,-9.5) {$8$};

\path[line width=0.5mm,blue] (n1112) edge[bend left=0] node [above] {} (n2112);
\path[line width=0.5mm,blue] (n2112) edge[bend left=0] node [above] {} (n3112);
\path[line width=0.5mm,blue] (n3112) edge[bend left=0] node [above] {} (n4112);
\path[line width=0.5mm,blue] (n4112) edge[bend left=0] node [above] {} (n5112);
\path[line width=0.5mm,green,dash dot] (n5112) edge[bend left=0] node [above] {$\Delta_2$} (n6112);
\path[line width=0.5mm,blue] (n7112) edge[bend left=0] node [below] {} (n3112);
\path[line width=0.5mm,blue] (n8112) edge[bend left=0] node [above] {} (n4112);
\path[line width=0.5mm,blue] (n2112) edge[bend right=90] node [below] {$\Delta_4$} (n6112);
\path[line width=0.5mm,red,dotted] (n2112) edge[bend right=90] node [above] {$\Delta_3$} (n5112);
\path[line width=0.5mm,red,dotted] (n1112) edge[bend left=90] node [above] {$\Delta_3$} (n5112);
	
\node[shape=circle,minimum size = 24pt,inner sep=0.3pt] (m423) at (-1,-12.2) {{ or}};

	\node[shape=circle,minimum size = 24pt,inner sep=0.3pt] (m413) at (-1,-14) {{\large $T_{23}$ :}};
	\node[shape=circle,draw=black,minimum size = 12pt,inner sep=0.3pt] (n1113) at (0,-14) {$1$};
	\node[shape=circle,draw=black,minimum size = 12pt,inner sep=0.3pt] (n2113) at (1,-14) {$2$};
	\node[shape=circle,draw=black,minimum size = 12pt,inner sep=0.3pt] (n3113) at (2,-14) {$3$};
  \node[shape=circle,draw=black,minimum size = 12pt,inner sep=0.3pt] (n4113) at (3,-14) {$4$};
	\node[shape=circle,draw=black,minimum size = 12pt,inner sep=0.3pt] (n5113) at (4,-14) {$5$};
  \node[shape=circle,draw=black,minimum size = 12pt,inner sep=0.3pt] (n6113) at (5,-14) {$6$};
	\node[shape=circle,draw=black,minimum size = 12pt,inner sep=0.3pt] (n7113) at (2,-13) {$7$};
  \node[shape=circle,draw=black,minimum size = 12pt,inner sep=0.3pt] (n8113) at (3,-13) {$8$};

\path[line width=0.5mm,blue] (n1113) edge[bend left=0] node [above] {} (n2113);
\path[line width=0.5mm,blue] (n2113) edge[bend left=0] node [above] {} (n3113);
\path[line width=0.5mm,blue] (n3113) edge[bend left=0] node [above] {} (n4113);
\path[line width=0.5mm,blue] (n4113) edge[bend left=0] node [above] {} (n5113);
\path[line width=0.5mm,red,dotted] (n5113) edge[bend left=0] node [above] {$\Delta_3$} (n6113);
\path[line width=0.5mm,blue] (n7113) edge[bend left=0] node [below] {} (n3113);
\path[line width=0.5mm,blue] (n8113) edge[bend left=0] node [above] {} (n4113);
\path[line width=0.5mm,blue] (n2113) edge[bend right=90] node [below] {$\Delta_4$} (n6113);
\path[line width=0.5mm,green,dash dot] (n2113) edge[bend right=90] node [above] {$\Delta_2$} (n5113);
\path[line width=0.5mm,red,dotted] (n1113) edge[bend left=90] node [above] {$\Delta_3$} (n5113);

\node[shape=circle,minimum size = 24pt,inner sep=0.3pt] (m423) at (-1,-15.7) {{ or}};

\node[shape=circle,minimum size = 24pt,inner sep=0.3pt] (m4114) at (-3,-17.5) {{\large Switch $(2,5)$ with $(2,6)$ :}};

	\node[shape=circle,draw=black,minimum size = 12pt,inner sep=0.3pt] (n14) at (0,-17.5) {$1$};
	\node[shape=circle,draw=black,minimum size = 12pt,inner sep=0.3pt] (n24) at (1,-17.5) {$2$};
	\node[shape=circle,draw=black,minimum size = 12pt,inner sep=0.3pt] (n34) at (2,-17.5) {$3$};
  \node[shape=circle,draw=black,minimum size = 12pt,inner sep=0.3pt] (n44) at (3,-17.5) {$4$};
	\node[shape=circle,draw=black,minimum size = 12pt,inner sep=0.3pt] (n54) at (4,-17.5) {$5$};
  \node[shape=circle,draw=black,minimum size = 12pt,inner sep=0.3pt] (n64) at (5,-17.5) {$6$};
	\node[shape=circle,draw=black,minimum size = 12pt,inner sep=0.3pt] (n74) at (2,-16.5) {$7$};
  \node[shape=circle,draw=black,minimum size = 12pt,inner sep=0.3pt] (n84) at (3,-16.5) {$8$};

\node[shape=circle,minimum size = 24pt,inner sep=0.3pt] (m44) at (-1,-1.8) {{\large $\downarrow$}};
\path[line width=0.5mm,blue] (n14) edge[bend left=0] node [above] {} (n24);
\path[line width=0.5mm,blue] (n24) edge[bend left=0] node [above] {} (n34);
\path[line width=0.5mm,blue] (n34) edge[bend left=0] node [above] {} (n44);
\path[line width=0.5mm,blue] (n44) edge[bend left=0] node [above] {} (n54);
\path[line width=0.5mm,blue] (n54) edge[bend left=0] node [above] {$\Delta_4$} (n64);
\path[line width=0.5mm,blue] (n74) edge[bend left=0] node [below] {} (n34);
\path[line width=0.5mm,blue] (n84) edge[bend left=0] node [above] {} (n44);
\path[line width=0.5mm,green,dash dot] (n24) edge[bend right=90] node [below] {$\Delta_2$} (n64);
\path[line width=0.5mm,red,dotted] (n24) edge[bend right=90] node [above] {$\Delta_3$} (n54);
\path[line width=0.5mm,red,dotted] (n14) edge[bend left=90] node [above] {$\Delta_3$} (n54);

\end{tikzpicture}
\caption{Reduction of $T_{16}$ when $(2,5)\in E(\Gamma_2)$, $(2,6)\in E(\Gamma_3)$\label{T16T23G23}}
\end{figure}
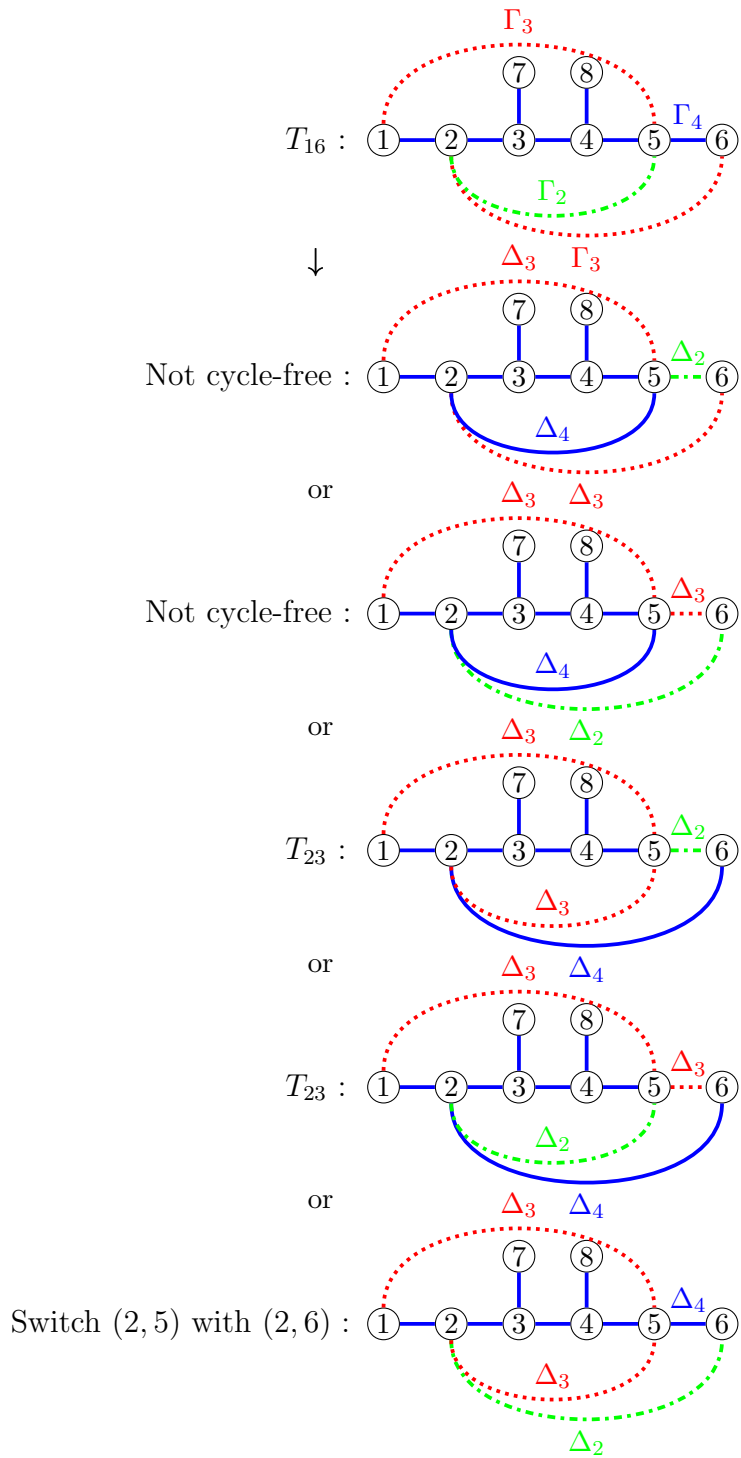

Case VI: If $(1,6), (2,5)\in E(\Gamma_3)$, and $(1,5)\in E(\Gamma_2)$, then a similar argument to Case V shows that  our partition can be reduced to one such that $\Delta_4$ is the graph $T_{19}$, which concludes our proof.

\end{proof}

\begin{lemma}
Let $(\Gamma_1,\Gamma_2,\Gamma_3,\Gamma_4)$ be a cycle-free $4$-partition of the complete graph $K_8$  such that $\Gamma_4$ is the graph  $T_{20}$. Then $(\Gamma_1,\Gamma_2,\Gamma_3,\Gamma_4)$ is weakly  equivalent to a cycle-free $4$-partition $(\Omega_1,\Omega_2,\Omega_3,\Omega_4)$ such that $\Omega_4$ is the graph $T_{19}$.
\end{lemma}
\begin{proof}
Up to weak equivalence we may assume that $\Gamma_4=T_{20}$ with the labeling from Figure \ref{T20}. Notice that none of the edges $(1,4)$, $(1,8)$, $(2,4)$ or $(2,8)$ are in $\Gamma_4$. Since our partition has only three other graphs $\Gamma_1$, $\Gamma_2$ and $\Gamma_3$, we know that at least two of these four edges must belong to the same graph. Up to weakly equivalence we may assume that they belong to $\Gamma_3$. There are six different cases.

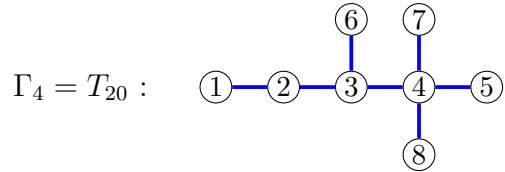
\begin{figure}[H]
\centering
\begin{tikzpicture}
  [scale=0.9,auto=left]%,every node/.style={shape = circle, draw, fill = white,minimum size = 14pt, inner sep=0.3pt}]%baseline=(a.center)]%{circle,fill=black}]
	%\tikzset{VertexStyle/.style = {shape = circle,fill = black,minimum size = 9mm,inner sep=2pt}}
	
	\node[shape=circle,minimum size = 24pt,inner sep=0.3pt] (m4) at (-2,0) {{\large $\Gamma_4=T_{20}$ :}};
	\node[shape=circle,draw=black,minimum size = 12pt,inner sep=0.3pt] (n1) at (0,0) {$1$};
	\node[shape=circle,draw=black,minimum size = 12pt,inner sep=0.3pt] (n2) at (1,0) {$2$};
	\node[shape=circle,draw=black,minimum size = 12pt,inner sep=0.3pt] (n3) at (2,0) {$3$};
  \node[shape=circle,draw=black,minimum size = 12pt,inner sep=0.3pt] (n4) at (3,0) {$4$};
	\node[shape=circle,draw=black,minimum size = 12pt,inner sep=0.3pt] (n5) at (4,0) {$5$};
  \node[shape=circle,draw=black,minimum size = 12pt,inner sep=0.3pt] (n6) at (2,1) {$6$};
	\node[shape=circle,draw=black,minimum size = 12pt,inner sep=0.3pt] (n7) at (3,1) {$7$};
  \node[shape=circle,draw=black,minimum size = 12pt,inner sep=0.3pt] (n8) at (3,-1) {$8$};

\path[line width=0.5mm,blue] (n1) edge[bend left=0] node [above] {} (n2);
\path[line width=0.5mm,blue] (n2) edge[bend left=0] node [above] {} (n3);
\path[line width=0.5mm,blue] (n3) edge[bend left=0] node [above] {} (n4);
\path[line width=0.5mm,blue] (n4) edge[bend left=0] node [above] {} (n5);
\path[line width=0.5mm,blue] (n3) edge[bend left=0] node [above] {} (n6);
\path[line width=0.5mm,blue] (n4) edge[bend left=0] node [below] {} (n7);
\path[line width=0.5mm,blue] (n4) edge[bend left=0] node [above] {} (n8);
%\path[line width=0.5mm,red,dotted] (n1) edge[bend right=90] node [below] {$\Gamma_3$} (n8);
%\path[line width=0.5mm,red,dotted] (n2) edge[bend right=90] node [above] {$\Gamma_3$} (n8);

\end{tikzpicture}
\caption{The graph $\Gamma_4=T_{20}$ \label{T20}}
\end{figure}

Case I: Assume that  $(1,8)$ and $(2,8)\in E(\Gamma_3)$. From Figure \ref{T20T19} we can see that using the involution $(1,2,8)$  our partition can be reduced to one with $\Delta_4=T_{19}$ (the other possible partition in Figure \ref{T20T19} is not cycle-free). %From Remark \ref{remark1723} we know that this partition can be reduced further to one with $\Delta_4=T_{19}$.

\begin{figure}[H]
\centering
\begin{tikzpicture}
  [scale=0.9,auto=left]%,every node/.style={shape = circle, draw, fill = white,minimum size = 14pt, inner sep=0.3pt}]%baseline=(a.center)]%{circle,fill=black}]
	%\tikzset{VertexStyle/.style = {shape = circle,fill = black,minimum size = 9mm,inner sep=2pt}}
	
	\node[shape=circle,minimum size = 24pt,inner sep=0.3pt] (m4) at (-2,0) {{\large $\Gamma_4=T_{20}$ :}};
	\node[shape=circle,draw=black,minimum size = 12pt,inner sep=0.3pt] (n1) at (0,0) {$1$};
	\node[shape=circle,draw=black,minimum size = 12pt,inner sep=0.3pt] (n2) at (1,0) {$2$};
	\node[shape=circle,draw=black,minimum size = 12pt,inner sep=0.3pt] (n3) at (2,0) {$3$};
  \node[shape=circle,draw=black,minimum size = 12pt,inner sep=0.3pt] (n4) at (3,0) {$4$};
	\node[shape=circle,draw=black,minimum size = 12pt,inner sep=0.3pt] (n5) at (4,0) {$5$};
  \node[shape=circle,draw=black,minimum size = 12pt,inner sep=0.3pt] (n6) at (2,1) {$6$};
	\node[shape=circle,draw=black,minimum size = 12pt,inner sep=0.3pt] (n7) at (3,1) {$7$};
  \node[shape=circle,draw=black,minimum size = 12pt,inner sep=0.3pt] (n8) at (3,-1) {$8$};

\node[shape=circle,minimum size = 24pt,inner sep=0.3pt] (m4) at (-2,-1.8) {{\large $\downarrow$}};
\path[line width=0.5mm,blue] (n1) edge[bend left=0] node [above] {$\Gamma_4$} (n2);
\path[line width=0.5mm,blue] (n2) edge[bend left=0] node [above] {} (n3);
\path[line width=0.5mm,blue] (n3) edge[bend left=0] node [above] {} (n4);
\path[line width=0.5mm,blue] (n4) edge[bend left=0] node [above] {} (n5);
\path[line width=0.5mm,blue] (n3) edge[bend left=0] node [above] {} (n6);
\path[line width=0.5mm,blue] (n4) edge[bend left=0] node [below] {} (n7);
\path[line width=0.5mm,blue] (n4) edge[bend left=0] node [above] {} (n8);
\path[line width=0.5mm,red,dotted] (n1) edge[bend right=90] node [below] {$\Gamma_3$} (n8);
\path[line width=0.5mm,red,dotted] (n2) edge[bend right=90] node [above] {$\Gamma_3$} (n8);

	\node[shape=circle,minimum size = 24pt,inner sep=0.3pt] (m41) at (-2,-3.5) {{\large Not cycle-free :}};
	\node[shape=circle,draw=black,minimum size = 12pt,inner sep=0.3pt] (n11) at (0,-3.5) {$1$};
	\node[shape=circle,draw=black,minimum size = 12pt,inner sep=0.3pt] (n21) at (1,-3.5) {$2$};
	\node[shape=circle,draw=black,minimum size = 12pt,inner sep=0.3pt] (n31) at (2,-3.5) {$3$};
  \node[shape=circle,draw=black,minimum size = 12pt,inner sep=0.3pt] (n41) at (3,-3.5) {$4$};
	\node[shape=circle,draw=black,minimum size = 12pt,inner sep=0.3pt] (n51) at (4,-3.5) {$5$};
  \node[shape=circle,draw=black,minimum size = 12pt,inner sep=0.3pt] (n61) at (2,-2.5) {$6$};
	\node[shape=circle,draw=black,minimum size = 12pt,inner sep=0.3pt] (n71) at (3,-2.5) {$7$};
  \node[shape=circle,draw=black,minimum size = 12pt,inner sep=0.3pt] (n81) at (3,-4.5) {$8$};

\path[line width=0.5mm,red,dotted] (n11) edge[bend left=0] node [above] {$\Delta_3$} (n21);
\path[line width=0.5mm,blue] (n21) edge[bend left=0] node [above] {} (n31);
\path[line width=0.5mm,blue] (n31) edge[bend left=0] node [above] {} (n41);
\path[line width=0.5mm,blue] (n41) edge[bend left=0] node [above] {} (n51);
\path[line width=0.5mm,blue] (n31) edge[bend left=0] node [above] {} (n61);
\path[line width=0.5mm,blue] (n41) edge[bend left=0] node [below] {} (n71);
\path[line width=0.5mm,blue] (n41) edge[bend left=0] node [above] {} (n81);
\path[line width=0.5mm,red,dotted] (n11) edge[bend right=90] node [below] {$\Delta_3$} (n81);
\path[line width=0.5mm,blue] (n21) edge[bend right=90] node [above] {$\Delta_4$} (n81);

\node[shape=circle,minimum size = 24pt,inner sep=0.3pt] (m42) at (-2,-5.3) {{ or}};

	\node[shape=circle,minimum size = 24pt,inner sep=0.3pt] (m412) at (-2,-7) {{\large $\Delta_4=T_{19}$ :}};
	\node[shape=circle,draw=black,minimum size = 12pt,inner sep=0.3pt] (n112) at (0,-7) {$1$};
	\node[shape=circle,draw=black,minimum size = 12pt,inner sep=0.3pt] (n212) at (1,-7) {$2$};
	\node[shape=circle,draw=black,minimum size = 12pt,inner sep=0.3pt] (n312) at (2,-7) {$3$};
  \node[shape=circle,draw=black,minimum size = 12pt,inner sep=0.3pt] (n412) at (3,-7) {$4$};
	\node[shape=circle,draw=black,minimum size = 12pt,inner sep=0.3pt] (n512) at (4,-7) {$5$};
  \node[shape=circle,draw=black,minimum size = 12pt,inner sep=0.3pt] (n612) at (2,-6) {$6$};
	\node[shape=circle,draw=black,minimum size = 12pt,inner sep=0.3pt] (n712) at (3,-6) {$7$};
  \node[shape=circle,draw=black,minimum size = 12pt,inner sep=0.3pt] (n812) at (3,-8) {$8$};

\path[line width=0.5mm,red,dotted] (n112) edge[bend left=0] node [above] {$\Delta_3$} (n212);
\path[line width=0.5mm,blue] (n212) edge[bend left=0] node [above] {} (n312);
\path[line width=0.5mm,blue] (n312) edge[bend left=0] node [above] {} (n412);
\path[line width=0.5mm,blue] (n412) edge[bend left=0] node [above] {} (n512);
\path[line width=0.5mm,blue] (n312) edge[bend left=0] node [above] {} (n612);
\path[line width=0.5mm,blue] (n412) edge[bend left=0] node [below] {} (n712);
\path[line width=0.5mm,blue] (n412) edge[bend left=0] node [above] {} (n812);
\path[line width=0.5mm,blue] (n112) edge[bend right=90] node [below] {$\Delta_4$} (n812);
\path[line width=0.5mm,red,dotted] (n212) edge[bend right=90] node [above] {$\Delta_3$} (n812);

\end{tikzpicture}
\caption{Reduction of $T_{20}$ when $(1,8)$, $(2,8)\in E(\Gamma_3)$\label{T20T19}}
\end{figure}
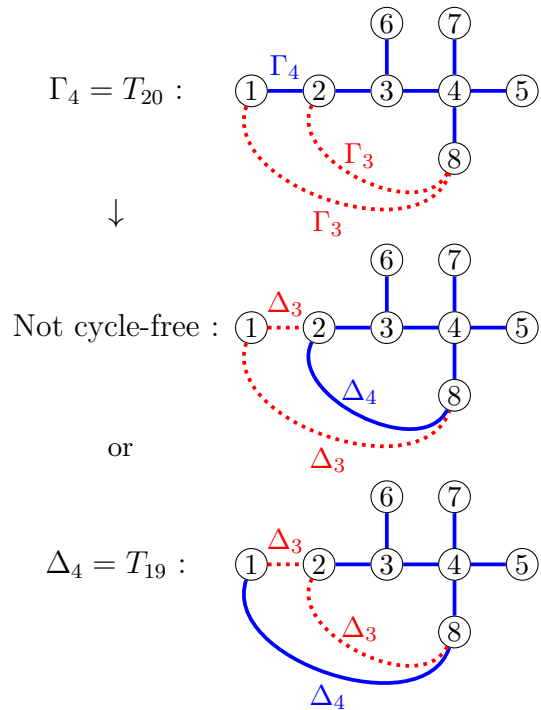

Case II: If $(1,4)$ and $(1,8)\in E(\Gamma_3)$ then using the involution $(1,4,8)$ our partition can be reduced to one with $\Delta_4=T_{17}$. And so by Remark  \ref{remark1723} we know that this partition can be reduced further to one with $\Delta_4=T_{19}$.

Case III: If $(2,4)$ and $(2,8)\in E(\Gamma_3)$ then using the involution $(2,4,8)$ our partition can be reduced to one with $\Delta_4=T_{23}$, which by Remark  \ref{remark1723}  can be reduced further to one with $\Delta_4=T_{19}$.

Case IV: If $(1,4)$ and $(2,4)\in E(\Gamma_3)$ then our partition is not cycle-free. Indeed, in this case there is only one other edge connected to vertex $4$ that does not appear in Figure \ref{T20TNCF} (namely  the edge $(4,6)$).  However, there are two more graphs $\Gamma_1$ and $\Gamma_2$ in our partition, so one of these two graphs will not be connected to vertex $4$. This implies that a partition with $\Gamma_3$ and $\Gamma_4$ depicted in  Figure \ref{T20TNCF} cannot be cycle free.

\begin{figure}[h!]
\centering
\begin{tikzpicture}
  [scale=0.9,auto=left]%,every node/.style={shape = circle, draw, fill = white,minimum size = 14pt, inner sep=0.3pt}]%baseline=(a.center)]%{circle,fill=black}]
	%\tikzset{VertexStyle/.style = {shape = circle,fill = black,minimum size = 9mm,inner sep=2pt}}
	
	\node[shape=circle,minimum size = 24pt,inner sep=0.3pt] (m4) at (-2,0) {{\large $\Gamma_4=T_{20}$ :}};
	\node[shape=circle,draw=black,minimum size = 12pt,inner sep=0.3pt] (n1) at (0,0) {$1$};
	\node[shape=circle,draw=black,minimum size = 12pt,inner sep=0.3pt] (n2) at (1,0) {$2$};
	\node[shape=circle,draw=black,minimum size = 12pt,inner sep=0.3pt] (n3) at (2,0) {$3$};
  \node[shape=circle,draw=black,minimum size = 12pt,inner sep=0.3pt] (n4) at (3,0) {$4$};
	\node[shape=circle,draw=black,minimum size = 12pt,inner sep=0.3pt] (n5) at (4,0) {$5$};
  \node[shape=circle,draw=black,minimum size = 12pt,inner sep=0.3pt] (n6) at (2,1) {$6$};
	\node[shape=circle,draw=black,minimum size = 12pt,inner sep=0.3pt] (n7) at (3,1) {$7$};
  \node[shape=circle,draw=black,minimum size = 12pt,inner sep=0.3pt] (n8) at (3,-1) {$8$};

\path[line width=0.5mm,blue] (n1) edge[bend left=0] node [above] {$\Gamma_4$} (n2);
\path[line width=0.5mm,blue] (n2) edge[bend left=0] node [above] {} (n3);
\path[line width=0.5mm,blue] (n3) edge[bend left=0] node [above] {} (n4);
\path[line width=0.5mm,blue] (n4) edge[bend left=0] node [above] {} (n5);
\path[line width=0.5mm,blue] (n3) edge[bend left=0] node [above] {} (n6);
\path[line width=0.5mm,blue] (n4) edge[bend left=0] node [below] {} (n7);
\path[line width=0.5mm,blue] (n4) edge[bend left=0] node [above] {} (n8);
\path[line width=0.5mm,red,dotted] (n1) edge[bend right=90] node [below] {$\Gamma_3$} (n4);
\path[line width=0.5mm,red,dotted] (n2) edge[bend right=90] node [above] {$\Gamma_3$} (n4);

\end{tikzpicture}
\caption{Case of $\Gamma_4=T_{20}$ and $(1,4)$, $(2,4)\in E(\Gamma_3)$ is not cycle-free\label{T20TNCF}}
\end{figure}
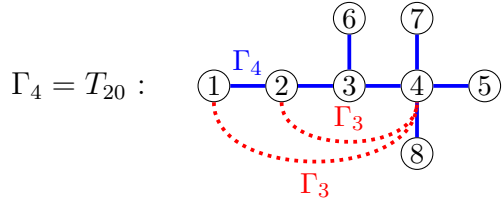

Case V: Assume that  $(2,4), (1,8)\in E(\Gamma_3)$ and $(1,4)\in E(\Gamma_2)$. Then we can see from Figure \ref{T20G3G2} that after using the involution $(1,2,4)$ we can switch the edges $(1,4)$ and $(2,4)$, which gives  $(1,4), (1,8)\in E(\Gamma_3)$ and $(2,4)\in E(\Gamma_2)$, a situation that was discussed  in Case II.

\begin{figure}[H]
\centering
\begin{tikzpicture}
  [scale=0.9,auto=left]%,every node/.style={shape = circle, draw, fill = white,minimum size = 14pt, inner sep=0.3pt}]%baseline=(a.center)]%{circle,fill=black}]
	%\tikzset{VertexStyle/.style = {shape = circle,fill = black,minimum size = 9mm,inner sep=2pt}}
	
	\node[shape=circle,minimum size = 24pt,inner sep=0.3pt] (m4) at (-1,0) {{\large $T_{20}$ :}};
	\node[shape=circle,draw=black,minimum size = 12pt,inner sep=0.3pt] (n1) at (0,0) {$1$};
	\node[shape=circle,draw=black,minimum size = 12pt,inner sep=0.3pt] (n2) at (1,0) {$2$};
	\node[shape=circle,draw=black,minimum size = 12pt,inner sep=0.3pt] (n3) at (2,0) {$3$};
  \node[shape=circle,draw=black,minimum size = 12pt,inner sep=0.3pt] (n4) at (3,0) {$4$};
	\node[shape=circle,draw=black,minimum size = 12pt,inner sep=0.3pt] (n5) at (4,0) {$5$};
  \node[shape=circle,draw=black,minimum size = 12pt,inner sep=0.3pt] (n6) at (2,1) {$6$};
	\node[shape=circle,draw=black,minimum size = 12pt,inner sep=0.3pt] (n7) at (3,1) {$7$};
  \node[shape=circle,draw=black,minimum size = 12pt,inner sep=0.3pt] (n8) at (3,-1) {$8$};

\node[shape=circle,minimum size = 24pt,inner sep=0.3pt] (m4) at (-1,-1.8) {{\large $\downarrow$}};
\path[line width=0.5mm,blue] (n1) edge[bend left=0] node [above] {$\Gamma_4$} (n2);
\path[line width=0.5mm,blue] (n2) edge[bend left=0] node [above] {} (n3);
\path[line width=0.5mm,blue] (n3) edge[bend left=0] node [above] {} (n4);
\path[line width=0.5mm,blue] (n4) edge[bend left=0] node [above] {} (n5);
\path[line width=0.5mm,blue] (n3) edge[bend left=0] node [above] {} (n6);
\path[line width=0.5mm,blue] (n4) edge[bend left=0] node [below] {} (n7);
\path[line width=0.5mm,blue] (n4) edge[bend left=0] node [above] {} (n8);
\path[line width=0.5mm,green,dash dot] (n1) edge[bend right=90] node [below] {$\Gamma_2$} (n4);
\path[line width=0.5mm,red,dotted] (n2) edge[bend right=90] node [above] {$\Gamma_3$} (n4);
\path[line width=0.5mm,red,dotted] (n1) edge[bend right=90] node [below] {$\Gamma_3$} (n8);

	\node[shape=circle,minimum size = 24pt,inner sep=0.3pt] (m41) at (-2,-3.5) {{\large Not cycle-free :}};
	\node[shape=circle,draw=black,minimum size = 12pt,inner sep=0.3pt] (n11) at (0,-3.5) {$1$};
	\node[shape=circle,draw=black,minimum size = 12pt,inner sep=0.3pt] (n21) at (1,-3.5) {$2$};
	\node[shape=circle,draw=black,minimum size = 12pt,inner sep=0.3pt] (n31) at (2,-3.5) {$3$};
  \node[shape=circle,draw=black,minimum size = 12pt,inner sep=0.3pt] (n41) at (3,-3.5) {$4$};
	\node[shape=circle,draw=black,minimum size = 12pt,inner sep=0.3pt] (n51) at (4,-3.5) {$5$};
  \node[shape=circle,draw=black,minimum size = 12pt,inner sep=0.3pt] (n61) at (2,-2.5) {$6$};
	\node[shape=circle,draw=black,minimum size = 12pt,inner sep=0.3pt] (n71) at (3,-2.5) {$7$};
  \node[shape=circle,draw=black,minimum size = 12pt,inner sep=0.3pt] (n81) at (3,-4.5) {$8$};

\path[line width=0.5mm,red,dotted] (n11) edge[bend left=0] node [above] {$\Delta_3$} (n21);
\path[line width=0.5mm,blue] (n21) edge[bend left=0] node [above] {} (n31);
\path[line width=0.5mm,blue] (n31) edge[bend left=0] node [above] {} (n41);
\path[line width=0.5mm,blue] (n41) edge[bend left=0] node [above] {} (n51);
\path[line width=0.5mm,blue] (n31) edge[bend left=0] node [above] {} (n61);
\path[line width=0.5mm,blue] (n41) edge[bend left=0] node [below] {} (n71);
\path[line width=0.5mm,blue] (n41) edge[bend left=0] node [above] {} (n81);
\path[line width=0.5mm,green,dash dot] (n11) edge[bend right=90] node [below] {$\Delta_2$} (n41);
\path[line width=0.5mm,blue] (n21) edge[bend right=90] node [above] {$\Delta_4$} (n41);
\path[line width=0.5mm,red,dotted] (n11) edge[bend right=90] node [below] {$\Delta_3$} (n81);
	
\node[shape=circle,minimum size = 24pt,inner sep=0.3pt] (m42) at (-1,-5.3) {{ or}};

	\node[shape=circle,minimum size = 24pt,inner sep=0.3pt] (m41) at (-2,-7) {{\large Not cycle-free :}};
	\node[shape=circle,draw=black,minimum size = 12pt,inner sep=0.3pt] (n112) at (0,-7) {$1$};
	\node[shape=circle,draw=black,minimum size = 12pt,inner sep=0.3pt] (n212) at (1,-7) {$2$};
	\node[shape=circle,draw=black,minimum size = 12pt,inner sep=0.3pt] (n312) at (2,-7) {$3$};
  \node[shape=circle,draw=black,minimum size = 12pt,inner sep=0.3pt] (n412) at (3,-7) {$4$};
	\node[shape=circle,draw=black,minimum size = 12pt,inner sep=0.3pt] (n512) at (4,-7) {$5$};
  \node[shape=circle,draw=black,minimum size = 12pt,inner sep=0.3pt] (n612) at (2,-6) {$6$};
	\node[shape=circle,draw=black,minimum size = 12pt,inner sep=0.3pt] (n712) at (3,-6) {$7$};
  \node[shape=circle,draw=black,minimum size = 12pt,inner sep=0.3pt] (n812) at (3,-8) {$8$};

\path[line width=0.5mm,green,dash dot] (n112) edge[bend left=0] node [above] {$\Delta_2$} (n212);
\path[line width=0.5mm,blue] (n212) edge[bend left=0] node [above] {} (n312);
\path[line width=0.5mm,blue] (n312) edge[bend left=0] node [above] {} (n412);
\path[line width=0.5mm,blue] (n412) edge[bend left=0] node [above] {} (n512);
\path[line width=0.5mm,blue] (n312) edge[bend left=0] node [above] {} (n612);
\path[line width=0.5mm,blue] (n412) edge[bend left=0] node [below] {} (n712);
\path[line width=0.5mm,blue] (n412) edge[bend left=0] node [above] {} (n812);
\path[line width=0.5mm,red,dotted] (n112) edge[bend right=90] node [below] {$\Delta_3$} (n412);
\path[line width=0.5mm,blue] (n212) edge[bend right=90] node [above] {$\Delta_4$} (n412);
\path[line width=0.5mm,red,dotted] (n112) edge[bend right=90] node [below] {$\Delta_3$} (n812);
	
\node[shape=circle,minimum size = 24pt,inner sep=0.3pt] (m42) at (-1,-8.3) {{ or}};

	\node[shape=circle,minimum size = 24pt,inner sep=0.3pt] (m413) at (-3.3,-10.5) {{\large  In $\Delta_4$ vertex $4$ has degree $5$ :}};
	\node[shape=circle,draw=black,minimum size = 12pt,inner sep=0.3pt] (n113) at (0,-10.5) {$1$};
	\node[shape=circle,draw=black,minimum size = 12pt,inner sep=0.3pt] (n213) at (1,-10.5) {$2$};
	\node[shape=circle,draw=black,minimum size = 12pt,inner sep=0.3pt] (n313) at (2,-10.5) {$3$};
  \node[shape=circle,draw=black,minimum size = 12pt,inner sep=0.3pt] (n413) at (3,-10.5) {$4$};
	\node[shape=circle,draw=black,minimum size = 12pt,inner sep=0.3pt] (n513) at (4,-10.5) {$5$};
  \node[shape=circle,draw=black,minimum size = 12pt,inner sep=0.3pt] (n613) at (2,-9.5) {$6$};
	\node[shape=circle,draw=black,minimum size = 12pt,inner sep=0.3pt] (n713) at (3,-9.5) {$7$};
  \node[shape=circle,draw=black,minimum size = 12pt,inner sep=0.3pt] (n813) at (3,-11.5) {$8$};

\path[line width=0.5mm,red,dotted] (n113) edge[bend left=0] node [above] {$\Delta_3$} (n213);
\path[line width=0.5mm,blue] (n213) edge[bend left=0] node [above] {} (n313);
\path[line width=0.5mm,blue] (n313) edge[bend left=0] node [above] {} (n413);
\path[line width=0.5mm,blue] (n413) edge[bend left=0] node [above] {} (n513);
\path[line width=0.5mm,blue] (n313) edge[bend left=0] node [above] {} (n613);
\path[line width=0.5mm,blue] (n413) edge[bend left=0] node [below] {} (n713);
\path[line width=0.5mm,blue] (n413) edge[bend left=0] node [above] {} (n813);
\path[line width=0.5mm,blue] (n113) edge[bend right=90] node [below] {$\Delta_4$} (n413);
\path[line width=0.5mm,green,dash dot] (n213) edge[bend right=90] node [above] {$\Delta_2$} (n413);
\path[line width=0.5mm,red,dotted] (n113) edge[bend right=90] node [below] {$\Delta_3$} (n813);
	
\node[shape=circle,minimum size = 24pt,inner sep=0.3pt] (m424) at (-1,-5.3) {{ or}};

	\node[shape=circle,minimum size = 24pt,inner sep=0.3pt] (m4134) at (-3.3,-14) {{\large  In $\Delta_4$ vertex $4$ has degree $5$ :}};
	\node[shape=circle,draw=black,minimum size = 12pt,inner sep=0.3pt] (n1134) at (0,-14) {$1$};
	\node[shape=circle,draw=black,minimum size = 12pt,inner sep=0.3pt] (n2134) at (1,-14) {$2$};
	\node[shape=circle,draw=black,minimum size = 12pt,inner sep=0.3pt] (n3134) at (2,-14) {$3$};
  \node[shape=circle,draw=black,minimum size = 12pt,inner sep=0.3pt] (n4134) at (3,-14) {$4$};
	\node[shape=circle,draw=black,minimum size = 12pt,inner sep=0.3pt] (n5134) at (4,-14) {$5$};
  \node[shape=circle,draw=black,minimum size = 12pt,inner sep=0.3pt] (n6134) at (2,-13) {$6$};
	\node[shape=circle,draw=black,minimum size = 12pt,inner sep=0.3pt] (n7134) at (3,-13) {$7$};
  \node[shape=circle,draw=black,minimum size = 12pt,inner sep=0.3pt] (n8134) at (3,-15) {$8$};

\path[line width=0.5mm,green,dash dot] (n1134) edge[bend left=0] node [above] {$\Delta_2$} (n2134);
\path[line width=0.5mm,blue] (n2134) edge[bend left=0] node [above] {} (n3134);
\path[line width=0.5mm,blue] (n3134) edge[bend left=0] node [above] {} (n4134);
\path[line width=0.5mm,blue] (n4134) edge[bend left=0] node [above] {} (n5134);
\path[line width=0.5mm,blue] (n3134) edge[bend left=0] node [above] {} (n6134);
\path[line width=0.5mm,blue] (n4134) edge[bend left=0] node [below] {} (n7134);
\path[line width=0.5mm,blue] (n4134) edge[bend left=0] node [above] {} (n8134);
\path[line width=0.5mm,blue] (n1134) edge[bend right=90] node [below] {$\Delta_4$} (n4134);
\path[line width=0.5mm,red,dotted] (n2134) edge[bend right=90] node [above] {$\Delta_3$} (n4134);
\path[line width=0.5mm,red,dotted] (n1134) edge[bend right=90] node [below] {$\Delta_3$} (n8134);
	
\node[shape=circle,minimum size = 24pt,inner sep=0.3pt] (m42) at (-1,-5.3) {{ or}};

	\node[shape=circle,minimum size = 24pt,inner sep=0.3pt] (m46) at (-3,-17.5) {{\large Switch $(1,4)$ with $(2,4)$ :}};
	\node[shape=circle,draw=black,minimum size = 12pt,inner sep=0.3pt] (n16) at (0,-17.5) {$1$};
	\node[shape=circle,draw=black,minimum size = 12pt,inner sep=0.3pt] (n26) at (1,-17.5) {$2$};
	\node[shape=circle,draw=black,minimum size = 12pt,inner sep=0.3pt] (n36) at (2,-17.5) {$3$};
  \node[shape=circle,draw=black,minimum size = 12pt,inner sep=0.3pt] (n46) at (3,-17.5) {$4$};
	\node[shape=circle,draw=black,minimum size = 12pt,inner sep=0.3pt] (n56) at (4,-17.5) {$5$};
  \node[shape=circle,draw=black,minimum size = 12pt,inner sep=0.3pt] (n66) at (2,-16.5) {$6$};
	\node[shape=circle,draw=black,minimum size = 12pt,inner sep=0.3pt] (n76) at (3,-16.5) {$7$};
  \node[shape=circle,draw=black,minimum size = 12pt,inner sep=0.3pt] (n86) at (3,-18.5) {$8$};

\path[line width=0.5mm,blue] (n16) edge[bend left=0] node [above] {$\Delta_4$} (n26);
\path[line width=0.5mm,blue] (n26) edge[bend left=0] node [above] {} (n36);
\path[line width=0.5mm,blue] (n36) edge[bend left=0] node [above] {} (n46);
\path[line width=0.5mm,blue] (n46) edge[bend left=0] node [above] {} (n56);
\path[line width=0.5mm,blue] (n36) edge[bend left=0] node [above] {} (n66);
\path[line width=0.5mm,blue] (n46) edge[bend left=0] node [below] {} (n76);
\path[line width=0.5mm,blue] (n46) edge[bend left=0] node [above] {} (n86);
\path[line width=0.5mm,red,dotted] (n16) edge[bend right=90] node [below] {$\Delta_3$} (n46);
\path[line width=0.5mm,green,dash dot] (n26) edge[bend right=90] node [above] {$\Delta_2$} (n46);
\path[line width=0.5mm,red,dotted] (n16) edge[bend right=90] node [below] {$\Delta_3$} (n86);

\end{tikzpicture}
\caption{Reduction of $T_{20}$ when $(1,4)\in E(\Gamma_2)$, $(1,8), (2,4)\in E(\Gamma_3)$\label{T20G3G2}}
\end{figure}
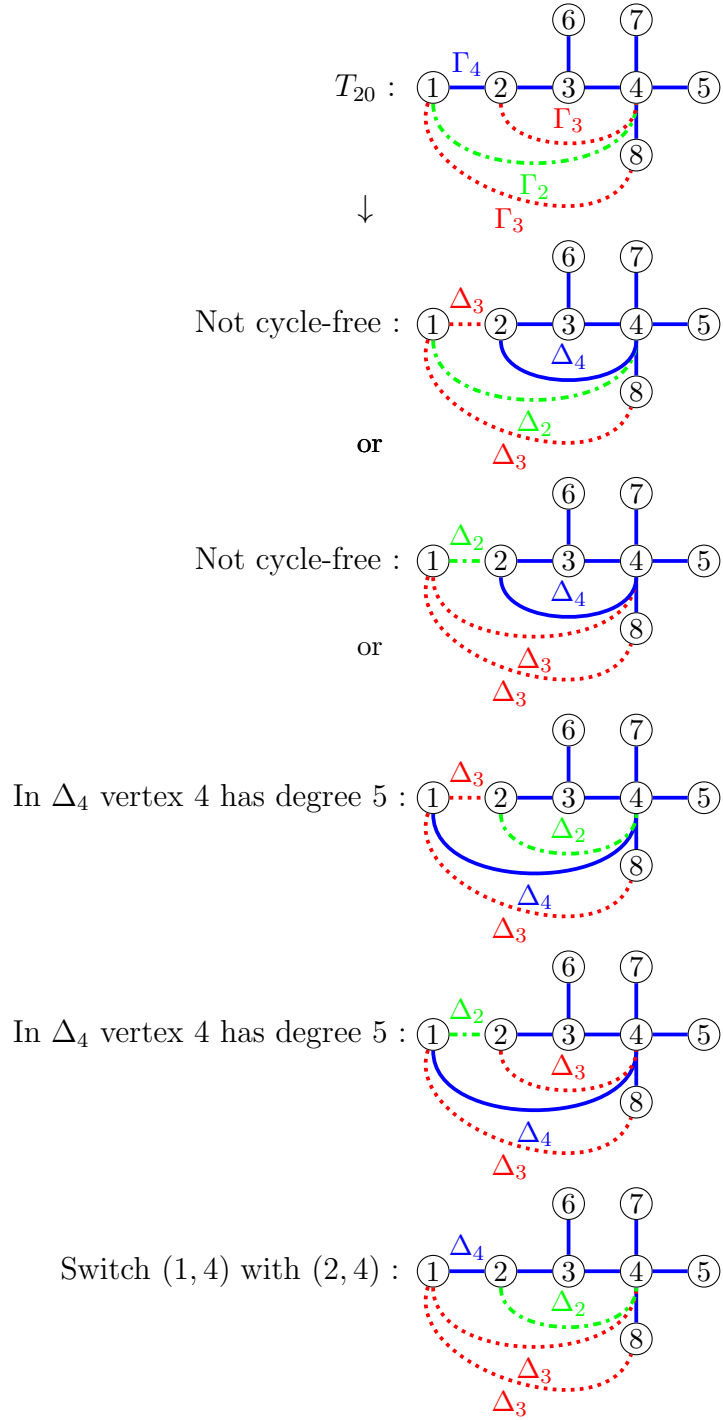

Case VI: Assume that  $(1,4), (2,8)\in E(\Gamma_3)$ and $(1,8)\in E(\Gamma_2)$. A similar argument shows that after using the involution $(1,2,8)$ we can reduce our partition to one with $\Delta_4=T_{17}$ (which by Remark  \ref{remark1723}  can be reduced further to one with $\Delta_4=T_{19}$), or we can switch the edges $(1,8)$ and $(2,8)$. This gives $(1,4), (1,8)\in E(\Delta_3)$ and $(2,8)\in E(\Delta_2)$,  a situation that was covered in Case II, which concludes our proof.

\end{proof}

\subsection{Reduction from $T_{19}$ to the twin-star $TS_4$}
\label{subsection42}

In this subsection we  deal with the case $\Gamma_4=T_{19}$. In the first version of this paper \cite{flss-v1}, the case $\Gamma_4=T_{19}$ was solved using MATLAB computations. For more details one can check the MATLAB code and the README file posted at \cite{flss}. That approach was not very satisfying, so we are glad we can give a direct, combinatorial proof.

For the purposes of this proof, we need new notation. For \(a\in\{2,4,6,8\}\), let \(H_a\) be the graph with
\begin{equation}\label{eq1}
E(H_a)=\{(1,2),(1,4),(1,6),(1,8),(2,3),(2,5),(a,7)\}.
\end{equation}
For example, we can see that \(H_4=T_{19}\), as shown in Figure \ref{fig-P}.
\begin{figure}[ht]
\centering
\begin{tikzpicture}[scale=0.9,auto=left]

\node[shape=circle,draw=black,minimum size=12pt,inner sep=0.3pt] (n1) at (2,0) {$1$};
\node[shape=circle,draw=black,minimum size=12pt,inner sep=0.3pt] (n2) at (1,0) {$2$};
\node[shape=circle,draw=black,minimum size=12pt,inner sep=0.3pt] (n3) at (0.2,1) {$3$};
\node[shape=circle,draw=black,minimum size=12pt,inner sep=0.3pt] (n5) at (0.2,-1) {$5$};
\node[shape=circle,draw=black,minimum size=12pt,inner sep=0.3pt] (n4) at (3,0) {$4$};
\node[shape=circle,draw=black,minimum size=12pt,inner sep=0.3pt] (n7) at (4,0) {$7$};
\node[shape=circle,draw=black,minimum size=12pt,inner sep=0.3pt] (n6) at (2,1) {$6$};
\node[shape=circle,draw=black,minimum size=12pt,inner sep=0.3pt] (n8) at (2,-1) {$8$};

\path[line width=0.5mm,blue] (n1) edge[bend left=0] node [above] {} (n2);
\path[line width=0.5mm,blue] (n1) edge[bend left=0] node [above] {} (n4);
\path[line width=0.5mm,blue] (n1) edge[bend left=0] node [above] {} (n6);
\path[line width=0.5mm,blue] (n1) edge[bend left=0] node [above] {} (n8);
\path[line width=0.5mm,blue] (n2) edge[bend left=0] node [above] {} (n3);
\path[line width=0.5mm,blue] (n2) edge[bend left=0] node [above] {} (n5);
\path[line width=0.5mm,blue] (n4) edge[bend left=0] node [above] {} (n7);

\end{tikzpicture}
\caption{The graph $H_4=T_{19}$}
\label{fig-P}
\end{figure}
Similarly,  \(H_6\) and \(H_8\) are isomorphic to \(T_{19}\), while
\(H_2\) is the twin-star graph \(TS_4=T_{22}\) (see Figure \ref{fig-P2}).

\begin{figure}[ht]
\centering
\begin{tikzpicture}[scale=0.9,auto=left]

\node[shape=circle,draw=black,minimum size=12pt,inner sep=0.3pt] (n1) at (2,0) {$1$};
\node[shape=circle,draw=black,minimum size=12pt,inner sep=0.3pt] (n2) at (1,0) {$2$};
\node[shape=circle,draw=black,minimum size=12pt,inner sep=0.3pt] (n3) at (0.2,1) {$3$};
\node[shape=circle,draw=black,minimum size=12pt,inner sep=0.3pt] (n5) at (0.2,-1) {$5$};
\node[shape=circle,draw=black,minimum size=12pt,inner sep=0.3pt] (n4) at (3,0) {$4$};
\node[shape=circle,draw=black,minimum size=12pt,inner sep=0.3pt] (n7) at (4,0) {$7$};
\node[shape=circle,draw=black,minimum size=12pt,inner sep=0.3pt] (n6) at (2,1) {$6$};
\node[shape=circle,draw=black,minimum size=12pt,inner sep=0.3pt] (n8) at (2,-1) {$8$};

\path[line width=0.5mm,blue] (n1) edge[bend left=0] node [above] {} (n2);
\path[line width=0.5mm,blue] (n1) edge[bend left=0] node [above] {} (n4);
\path[line width=0.5mm,blue] (n1) edge[bend left=0] node [above] {} (n6);
\path[line width=0.5mm,blue] (n1) edge[bend left=0] node [above] {} (n8);
\path[line width=0.5mm,blue] (n2) edge[bend left=0] node [above] {} (n3);
\path[line width=0.5mm,blue] (n2) edge[bend left=0] node [above] {} (n5);
%\path[line width=0.5mm,blue] (n4) edge[bend left=0] node [above] {} (n7);
\path[line width=0.5mm,blue] (n2) edge[bend right=90, looseness=1.5] node [above] {} (n7);

\end{tikzpicture}
\caption{The graph $H_2=T_{22}=TS_4$}
\label{fig-P2}
\end{figure}

We begin with an elementary observation.
\begin{lemma}\label{lem1}
Let \((\Gamma_1,\Gamma_2,\Gamma_3,\Gamma_4)\) be a cycle-free \(4\)-partition of \(K_8\) such that \(\Gamma_4=H_a\) for some \(a\in\{4,6,8\}\). Suppose that there exists \(b\in\{2,4,6,8\}\setminus\{a\}\), such that
\[
(a,b),(b,7)\in E(\Gamma_i),
\]
for some \(1\leq i\leq3\). If we denote
\[
(\Delta_1,\Delta_2,\Delta_3,\Delta_4)
=
(\Gamma_1,\Gamma_2,\Gamma_3,\Gamma_4)^{(a,b,7)},
\]
then \(\Delta_4=H_b\). % (i.e. the edge $(a,7)$ and $(b,7)$ are switched between $\Gamma_i$ and $\Gamma_4$)
\end{lemma}

\begin{proof}
By assumption,
\[
(a,b),(b,7)\in E(\Gamma_i),\quad \text{and} \quad (a,7)\in E(\Gamma_4).
\]
The edge \((a,7)\) cannot be replaced by \((a,b)\), because \((1,a)\), and \((1,b)\) belong to \(E(\Gamma_4)\), and placing \((a,b)\) in the fourth component would create the cycle \((1,a,b)\). The edge \((a,7)\) also cannot remain as \((a,7)\), since \((a,b)\) and \((b,7)\) already belong to the same component \(\Gamma_i\), and this would leave the partition unchanged on the three edges involved in the involution. Hence the nontrivial cycle-free redistribution determined by \((a,b,7)\) must replace \((a,7)\) by \((b,7)\) in the fourth component. Therefore \(\Delta_4=H_b\).
\end{proof}

The following lemma is essential to our argument.
\begin{lemma}\label{lem2}
Let \((\Gamma_1,\Gamma_2,\Gamma_3,\Gamma_4)\) be a cycle-free \(4\)-partition of \(K_8\) such that \(\Gamma_4=H_a\) for some \(a\in\{4,6,8\}\). Then there exists \(b\in\{2,4,6,8\}\setminus\{a\}\) such that, if
\[
(\Delta_1,\Delta_2,\Delta_3,\Delta_4)
=
(\Gamma_1,\Gamma_2,\Gamma_3,\Gamma_4)^{(a,b,7)},
\]
then \(\Delta_4=H_b\). In other words,  for at least one $b\in \{2,4,6,8\}\setminus \{a\}$, the involutions $(a,b,7)$ replaces $(a,7)$ by $(b,7)$ on the fourth component.
\end{lemma}

\begin{proof}
By means of contradiction, suppose that for every \(b\in\{2,4,6,8\}\setminus\{a\}\), the involution \((a,b,7)\) leaves the edge \((a,7)\) in the fourth component.

Fix such a vertex \(b\). Since \((a,b),(b,7)\notin E(\Gamma_4)\), there exist \(r,s\in\{1,2,3\}\) such that
\[
(a,b)\in E(\Gamma_r),\quad \text{and} \quad (b,7)\in E(\Gamma_s).
\]
By Lemma~\ref{lem1}, we must have that \(r\neq s\).

Let \[ (\Delta_1,\Delta_2,\Delta_3,\Delta_4)
=
(\Gamma_1,\Gamma_2,\Gamma_3,\Gamma_4)^{(a,b,7)}.\] By assumption \(\Delta_4=H_a\), so the involution $(a,b,7)$ exchanges \((a,b)\) and \((b,7)\) between the \(r\)-th and \(s\)-th components. Hence
\begin{equation}\label{eq2}
E(\Delta_r)=\bigl(E(\Gamma_r)\setminus\{(a,b)\}\bigr)\cup\{(b,7)\},
\quad \text{and} \quad
E(\Delta_s)=\bigl(E(\Gamma_s)\setminus\{(b,7)\}\bigr)\cup\{(a,b)\}.
\end{equation}

Since $\Gamma_r$ is a spanning tree any two vertices are connected by a unique path in $\Gamma_r$. In particular, since $(a,b)\in \Gamma_r$ we must have one of the two options depicted in  Figure \ref{fig-P3}. Note that the first option (i.e. the top graph in Figure \ref{fig-P3}) creates a cycle in $\Delta_r$. Since both \(\Gamma_r\) and \(\Delta_r\) are trees, we can conclude that the unique path from \(a\) to \(7\) in \(\Gamma_r\) does not contain the edge \((a,b)\).

\begin{figure}[ht]
\centering
\begin{tikzpicture}[scale=0.9,auto=left]

\node[shape=circle,draw=black,minimum size=12pt,inner sep=0.3pt] (n1) at (0,2) {$a$};
\node[shape=circle,draw=black,minimum size=12pt,inner sep=0.3pt] (n2) at (1,2) {$b$};
\node[shape=circle,draw=black,minimum size=12pt,inner sep=0.3pt] (n3) at (4,2) {$7$};

\path[line width=0.5mm,blue] (n1) edge[bend left=0] node [above] {} (n2);
\path[line width=0.5mm,blue,dashed] (n2) edge[bend left=0] node [above] {} (n3);

\node[shape=circle,minimum size = 24pt,inner sep=0.3pt] (m4) at (-1,1) {{\large $\Gamma_r$}};

\node[shape=circle,draw=black,minimum size=12pt,inner sep=0.3pt] (n11) at (0,0) {$b$};
\node[shape=circle,draw=black,minimum size=12pt,inner sep=0.3pt] (n21) at (1,0) {$a$};
\node[shape=circle,draw=black,minimum size=12pt,inner sep=0.3pt] (n31) at (4,0) {$7$};

\path[line width=0.5mm,blue] (n11) edge[bend left=0] node [above] {} (n21);
\path[line width=0.5mm,blue,dashed] (n21) edge[bend left=0] node [above] {} (n31);

%%%%%%%

\node[shape=circle,draw=black,minimum size=12pt,inner sep=0.3pt] (n12) at (7,2) {$a$};
\node[shape=circle,draw=black,minimum size=12pt,inner sep=0.3pt] (n22) at (8,2) {$b$};
\node[shape=circle,draw=black,minimum size=12pt,inner sep=0.3pt] (n32) at (11,2) {$7$};

\path[line width=0.5mm,blue] (n22) edge[bend left=90] node [above] {} (n32);
\path[line width=0.5mm,blue,dashed] (n22) edge[bend left=0] node [above] {} (n32);

\node[shape=circle,minimum size = 24pt,inner sep=0.3pt] (m4) at (6,1) {{\large $\Delta_r$}};

\node[shape=circle,draw=black,minimum size=12pt,inner sep=0.3pt] (n13) at (7,0) {$b$};
\node[shape=circle,draw=black,minimum size=12pt,inner sep=0.3pt] (n23) at (8,0) {$a$};
\node[shape=circle,draw=black,minimum size=12pt,inner sep=0.3pt] (n33) at (11,0) {$7$};

\path[line width=0.5mm,blue] (n13) edge[bend left=90] node [above] {} (n33);
\path[line width=0.5mm,blue,dashed] (n23) edge[bend left=0] node [above] {} (n33);

%\path[line width=0.5mm,blue] (n4) edge[bend left=0] node [above] {} (n7);
%\path[line width=0.5mm,blue] (n2) edge[bend right=90, looseness=1.5] node [above] {} (n7);

\end{tikzpicture}
\caption{The path from $a$ to $7$ in $\Gamma_r$ and $\Delta_r$}
\label{fig-P3}
\end{figure}

Since $(a,b)\in E(\Gamma_r)$,  the edge $(a,b)$ does not belong to any other $E(\Gamma_t)$ with $t\neq r$. This means that for each $1\leq i\leq 3$, the unique path in $\Gamma_i$ from $a$ to $7$ avoids the edge $(a,b)$.

Because of our assumption, we get that for every \(b\in\{2,4,6,8\}\setminus\{a\}\), the paths from \(a\) to \(7\) in \(\Gamma_1,\Gamma_2,\Gamma_3\) avoid the edge \((a,b)\).

Since \(\Gamma_4=H_a\), the edges \((1,a)\) and \((a,7)\) belong to \(E(\Gamma_4)\). Hence the three paths from \(a\) to \(7\) in \(\Gamma_1,\Gamma_2,\Gamma_3\) can start only with \((a,3)\) or \((a,5)\). But \(\Gamma_1,\Gamma_2,\Gamma_3\) are edge-disjoint spanning trees, so these three paths are pairwise edge-disjoint and must start with three distinct edges. This is impossible. Therefore our assumption was false, and so one of the three involutions changes the fourth component from \(H_a\) to \(H_b\).
\end{proof}
We also need this technical result.

\begin{lemma}\label{lem3}
Let \((\Gamma_1,\Gamma_2,\Gamma_3,\Gamma_4)\) be a cycle-free \(4\)-partition of \(K_8\) such that \(\Gamma_4=H_a\) for some \(a\in\{4,6,8\}\). Let \(c\in\{4,6,8\}\setminus\{a\}\), and suppose that, for some \(1\leq i\leq3\),
\begin{equation*}\label{eq3}
(2,c),(2,7)\in E(\Gamma_i).
\end{equation*}
Then \((\Gamma_1,\Gamma_2,\Gamma_3,\Gamma_4)\) is weakly equivalent to a cycle-free \(4\)-partition \((\Theta_1,\Theta_2,\Theta_3,\Theta_4)\) such that \(\Theta_4=H_2=TS_4\).
\end{lemma}

\begin{proof}
Up to weak equivalence we can assume that \(\Gamma_4=H_4\), $i=3$, and $c=6$.

\textit{Step 1.} Let
\[
(\Delta_1,\Delta_2,\Delta_3,\Delta_4)
=
(\Gamma_1,\Gamma_2,\Gamma_3,\Gamma_4)^{(4,6,7)}.
\]
First, notice that since \((1,4),(1,6)\in E(\Gamma_4)\), replacing \((4,7)\) by \((4,6)\) would create the cycle \((1,4,6)\) in \(\Delta_4\). Thus either \(\Delta_4=H_6\), or \(\Delta_4=H_4\).

\medskip

Case I: If \(\Delta_4=H_6\), then \((2,6),(2,7)\in E(\Delta_3)\), since neither of these edges is affected by the involution. Lemma~\ref{lem1}, applied to the cycle-free \(4\)-partition \((\Delta_1,\Delta_2,\Delta_3,\Delta_4)\) and the involution \((2,6,7)\), gives a new partition whose fourth component is \(H_2\).

Case II: If \(\Delta_4=H_4\). In this case the involution \((4,6,7)\) exchanges the edges \((4,6)\) and \((6,7)\). Neither edge belongs to \(\Gamma_3\). Indeed, if \((6,7)\in E(\Gamma_3)\), then \(\Gamma_3\) contains the cycle \((2,6,7)\). If \((4,6)\in E(\Gamma_3)\), then \((6,7)\in E(\Delta_3)\), so \(\Delta_3\) contains the cycle \((2,6,7)\).

So, up to weak equivalence  we may assume that $\Gamma_4=H_4$,
\begin{equation}\label{eq4}
(4,6)\in E(\Gamma_1),\quad \text{and} \quad (6,7)\in E(\Gamma_2),
\end{equation}
while
\begin{equation}\label{eq5}
(2,6),(2,7)\in E(\Gamma_3).
\end{equation}
In other words we have the partition from  Figure \ref{fig-P4}. Moreover, we may assume that the involution $(4,6,7)$ does not change $\Gamma_4=H_4$.

\begin{figure}[ht]
\centering
\begin{tikzpicture}[scale=0.9,auto=left]

\node[shape=circle,draw=black,minimum size=12pt,inner sep=0.3pt] (n1) at (1.5,0) {$1$};
\node[shape=circle,draw=black,minimum size=12pt,inner sep=0.3pt] (n2) at (-1,0) {$2$};
\node[shape=circle,draw=black,minimum size=12pt,inner sep=0.3pt] (n3) at (-1.8,1.5) {$3$};
\node[shape=circle,draw=black,minimum size=12pt,inner sep=0.3pt] (n5) at (-1.8,-1.5) {$5$};
\node[shape=circle,draw=black,minimum size=12pt,inner sep=0.3pt] (n4) at (3,0) {$4$};
\node[shape=circle,draw=black,minimum size=12pt,inner sep=0.3pt] (n7) at (4.5,0) {$7$};
\node[shape=circle,draw=black,minimum size=12pt,inner sep=0.3pt] (n6) at (2.3,1.5) {$6$};
\node[shape=circle,draw=black,minimum size=12pt,inner sep=0.3pt] (n8) at (2.3,-1.5) {$8$};

\path[line width=0.5mm,blue] (n1) edge[bend left=0] node [above] {} (n2);
\path[line width=0.5mm,blue] (n1) edge[bend left=0] node [above] {} (n4);
\path[line width=0.5mm,blue] (n1) edge[bend left=0] node [above] {} (n6);
\path[line width=0.5mm,blue] (n1) edge[bend left=0] node [above] {} (n8);
\path[line width=0.5mm,blue] (n2) edge[bend left=0] node [above] {} (n3);
\path[line width=0.5mm,blue] (n2) edge[bend left=0] node [above] {} (n5);
\path[line width=0.5mm,blue] (n4) edge[bend left=0] node [above] {} (n7);

\path[line width=0.5mm,red,dashed] (n2) edge[bend left=0] node [above] {$\Gamma_3$} (n6);
\path[line width=0.5mm,red,dashed] (n2) edge[bend left=90, looseness=1.2] node [above] {$\Gamma_3$} (n7);
\path[line width=0.5mm,orange,dotted] (n4) edge[bend left=0] node [left] {$\Gamma_1$} (n6);
\path[line width=0.5mm,green,dashed] (n6) edge[bend left=0] node [right] {$\Gamma_2$} (n7);

\end{tikzpicture}
\caption{The partition  $(\Gamma_1,\Gamma_2,\Gamma_3,\Gamma_4)$ after Step 1}
\label{fig-P4}
\end{figure}

\medskip
\noindent
\textit{Step 2. }

Consider the involution \((2,4,7)\) on \((\Gamma_1,\Gamma_2,\Gamma_3,\Gamma_4)\). If it changes  $\Gamma_4$ to \(H_2\), we are done. Otherwise, by  Lemma \ref{lem2}, the involution \((4,7,8)\) must change the fourth component to \(H_8\).
Let
\[
(\Delta_1,\Delta_2,\Delta_3,\Delta_4)
=
(\Gamma_1,\Gamma_2,\Gamma_3,\Gamma_4)^{(4,7,8)},
\]
so that \(\Delta_4=H_8\).
Since this involution changes none of the edges in \eqref{eq4} and \eqref{eq5}, we have
\begin{equation}\label{eq6}
(4,6)\in E(\Delta_1),\qquad (6,7)\in E(\Delta_2),\quad \text{and} \quad
(2,6),(2,7)\in E(\Delta_3).
\end{equation}
This gives the partition from Figure \ref{fig-P5}.

\begin{figure}[ht]
\centering
\begin{tikzpicture}[scale=0.9,auto=left]

\node[shape=circle,draw=black,minimum size=12pt,inner sep=0.3pt] (n1) at (1.5,0) {$1$};
\node[shape=circle,draw=black,minimum size=12pt,inner sep=0.3pt] (n2) at (-1,0) {$2$};
\node[shape=circle,draw=black,minimum size=12pt,inner sep=0.3pt] (n3) at (-1.8,1.5) {$3$};
\node[shape=circle,draw=black,minimum size=12pt,inner sep=0.3pt] (n5) at (-1.8,-1.5) {$5$};
\node[shape=circle,draw=black,minimum size=12pt,inner sep=0.3pt] (n4) at (3,0) {$4$};
\node[shape=circle,draw=black,minimum size=12pt,inner sep=0.3pt] (n7) at (4.5,0) {$7$};
\node[shape=circle,draw=black,minimum size=12pt,inner sep=0.3pt] (n6) at (2.3,1.5) {$6$};
\node[shape=circle,draw=black,minimum size=12pt,inner sep=0.3pt] (n8) at (2.3,-1.5) {$8$};

\path[line width=0.5mm,blue] (n1) edge[bend left=0] node [above] {} (n2);
\path[line width=0.5mm,blue] (n1) edge[bend left=0] node [above] {} (n4);
\path[line width=0.5mm,blue] (n1) edge[bend left=0] node [above] {} (n6);
\path[line width=0.5mm,blue] (n1) edge[bend left=0] node [above] {} (n8);
\path[line width=0.5mm,blue] (n2) edge[bend left=0] node [above] {} (n3);
\path[line width=0.5mm,blue] (n2) edge[bend left=0] node [above] {} (n5);
\path[line width=0.5mm,blue] (n8) edge[bend left=0] node [above] {} (n7);

\path[line width=0.5mm,red,dashed] (n2) edge[bend left=0] node [above] {$\Delta_3$} (n6);
\path[line width=0.5mm,red,dashed] (n2) edge[bend left=90, looseness=1.2] node [above] {$\Delta_3$} (n7);
\path[line width=0.5mm,orange,dotted] (n4) edge[bend left=0] node [left] {$\Delta_1$} (n6);
\path[line width=0.5mm,green,dashed] (n6) edge[bend left=0] node [right] {$\Delta_2$} (n7);

\end{tikzpicture}
\caption{The partition  $(\Delta_1,\Delta_2,\Delta_3,\Delta_4)$ after Step 2}
\label{fig-P5}
\end{figure}

\medskip
\noindent
\textit{Step 3.} Let
\[
(\Omega_1,\Omega_2,\Omega_3,\Omega_4)
=
(\Delta_1,\Delta_2,\Delta_3,\Delta_4)^{(6,7,8)}.
\]
The edge \((7,8)\in \Delta_4\) cannot be replaced by \((6,8)\), since \((1,8),(1,6)\in E(\Delta_4)\), and otherwise \(\Omega_4\) would contain the cycle \((1,6,8)\). Thus either \(\Omega_4=H_6\), or \(\Omega_4=H_8\).

Case I: If \(\Omega_4=H_6\), then \((2,6),(2,7)\in E(\Omega_3)\), and Lemma \ref{lem1}, applied to \((2,6,7)\), gives a partition whose fourth component is \(H_2\).

\medskip

Case II:  Suppose therefore that \(\Omega_4=H_8\).  We cannot have \((6,8)\in E(\Delta_3)\), because then the involution \((6,7,8)\) would place \((6,7)\) in \(E(\Omega_3)\), and together with \((2,6),(2,7)\in E(\Omega_3)\) this would give the cycle \((2,6,7)\). We cannot have \((6,8)\in E(\Delta_2)\), because \((6,7)\in E(\Delta_2)\) and Lemma~\ref{lem1} would imply \(\Omega_4=H_6\).  Hence
\[
(6,8)\in E(\Delta_1).
\]
Since \(\Omega_4=H_8\), the involution \((6,7,8)\) exchanges \((6,8)\) and \((6,7)\), so
\begin{equation}\label{eq7}
(6,8)\in E(\Omega_2),\qquad (4,6),(6,7)\in E(\Omega_1),
\end{equation}
and \((2,6),(2,7)\in E(\Omega_3)\). Which gives the partition from Figure \ref{fig-P6}.

\begin{figure}[ht]
\centering
\begin{tikzpicture}[scale=0.9,auto=left]

\node[shape=circle,draw=black,minimum size=12pt,inner sep=0.3pt] (n1) at (1.5,0) {$1$};
\node[shape=circle,draw=black,minimum size=12pt,inner sep=0.3pt] (n2) at (-1,0) {$2$};
\node[shape=circle,draw=black,minimum size=12pt,inner sep=0.3pt] (n3) at (-1.8,1.5) {$3$};
\node[shape=circle,draw=black,minimum size=12pt,inner sep=0.3pt] (n5) at (-1.8,-1.5) {$5$};
\node[shape=circle,draw=black,minimum size=12pt,inner sep=0.3pt] (n4) at (3,0) {$4$};
\node[shape=circle,draw=black,minimum size=12pt,inner sep=0.3pt] (n7) at (4.5,0) {$7$};
\node[shape=circle,draw=black,minimum size=12pt,inner sep=0.3pt] (n6) at (2.3,1.5) {$6$};
\node[shape=circle,draw=black,minimum size=12pt,inner sep=0.3pt] (n8) at (2.3,-1.5) {$8$};

\path[line width=0.5mm,blue] (n1) edge[bend left=0] node [above] {} (n2);
\path[line width=0.5mm,blue] (n1) edge[bend left=0] node [above] {} (n4);
\path[line width=0.5mm,blue] (n1) edge[bend left=0] node [above] {} (n6);
\path[line width=0.5mm,blue] (n1) edge[bend left=0] node [above] {} (n8);
\path[line width=0.5mm,blue] (n2) edge[bend left=0] node [above] {} (n3);
\path[line width=0.5mm,blue] (n2) edge[bend left=0] node [above] {} (n5);
\path[line width=0.5mm,blue] (n8) edge[bend left=0] node [above] {} (n7);

\path[line width=0.5mm,red,dashed] (n2) edge[bend left=0] node [above] {$\Omega_3$} (n6);
\path[line width=0.5mm,red,dashed] (n2) edge[bend left=90, looseness=1.2] node [above] {$\Omega_3$} (n7);
\path[line width=0.5mm,orange,dotted] (n4) edge[bend left=0] node [right] {$\Omega_1$} (n6);
\path[line width=0.5mm,green,dashed] (n6) edge[bend left=0] node [below, xshift=3mm, yshift=-2mm] {$ \Omega_2$} (n8);
\path[line width=0.5mm,orange,dotted] (n6) edge[bend left=0] node [right] {$\Omega_1$} (n7);
\end{tikzpicture}
\caption{The partition  $(\Omega_1,\Omega_2,\Omega_3,\Omega_4)$ after Step 3}
\label{fig-P6}
\end{figure}

\medskip
\noindent
\textit{Step 4.}  Consider the involution \((2,7,8)\), $(4,7,8)$ and $(6,7,8)$. By Lemma \ref{lem2} we know that one of these three
 involutions $(b,7,8)$  must change the fourth component $\Omega_4$ to $H_b$.

Case I: If \((2,7,8)\) changes the fourth component $\Omega_4$ to \(H_2\), then we are done.

Case II: If \((6,7,8)\) changes the fourth component $\Omega_4$ to \(H_6\), then since $(2,6)$, $(2,7)\in E(\Omega_3)$,   Lemma \ref{lem1} applied  to \((2,6,7)\), gives a cycle-free \(4\)-partition \((\Theta_1,\Theta_2,\Theta_3,\Theta_4)\) with \(\Theta_4=H_2\).

Case III: If \((4,7,8)\) changes the fourth component $\Omega_4$  to \(H_4\), then notice that $(4,6)$, $(6,7)\in E(\Omega_1)$ and so  the involution $(4,6,7)$ will change the fourth component  to \(H_6\), which reduces our problem to Case II.

\end{proof}

We are now ready to complete the reduction from \(T_{19}\) to \(TS_4\).

\begin{lemma}\label{lem4}
Let \((\Gamma_1,\Gamma_2,\Gamma_3,\Gamma_4)\) be a cycle-free \(4\)-partition of \(K_8\) such that \(\Gamma_4=T_{19}\). Then \((\Gamma_1,\Gamma_2,\Gamma_3,\Gamma_4)\) is weakly equivalent to a cycle-free \(4\)-partition \((\Delta_1,\Delta_2,\Delta_3,\Delta_4)\) such that \(\Delta_4=TS_4\).
\end{lemma}

\begin{proof} Up to weak equivalence, we can assume that $\Gamma_4=H_4$ (as in Figure~\ref{fig-P}). Consider the four edges
\[
(2,4),\qquad(2,6),\qquad(2,7),\qquad(2,8).
\]
These are the only edges incident with \(2\) that do not belong to \(\Gamma_4\). Since \(\Gamma_1,\Gamma_2,\Gamma_3\) are spanning trees, each contains at least one of these four edges. Hence one of these three components contains two of them. Up to weak equivalence we may assume that $\Gamma_3$ contains two of these four edges.

\medskip
\noindent
Case I:  \((2,7)\in \Gamma_3\) and one of the edges \((2,4),(2,6),(2,8)\) is also in $\Gamma_3$.

If \((2,4)\in \Gamma_3\), Lemma \ref{lem1}, applied to \((4,2,7)\), gives a cycle-free \(4\)-partition whose fourth component is \(H_2=TS_4\). If \((2,6)\), or $(2,8)\in \Gamma_3$ then Lemma \ref{lem3} gives the same conclusion.

\medskip
\noindent
Case II: If \((2,4),(2,6)\) are both in $\Gamma_3$. Let
\[
(\Psi_1,\Psi_2,\Psi_3,\Psi_4)
=
(\Gamma_1,\Gamma_2,\Gamma_3,\Gamma_4)^{(2,4,7)}.
\]

If \(\Psi_4=H_2\), we are done. Otherwise \(\Psi_4=H_4\), and the involution $(2,4,7)$ exchanges the edges \((2,4)\) and \((2,7)\), and  so \((2,7)\in E(\Psi_3)\). Since \((2,6)\) is unchanged,
\[
(2,6),(2,7)\in E(\Psi_3).
\]
Lemma~\ref{lem3}, applied to \((\Psi_1,\Psi_2,\Psi_3,\Psi_4)\), gives a weakly equivalent cycle-free \(4\)-partition whose fourth component is \(H_2\).

\medskip
\noindent
Case III: If $(2,4),(2,8) \in E(\Gamma_3)$. This is similar with Case II.

\medskip
\noindent
Case IV: If $(2,6),(2,8) \in E(\Gamma_3)$.  Consider first the involution \((4,2,7)\). If it changes the fourth component to \(H_2\), we are done. If it does not, then by Lemma~\ref{lem2}, at least one of \((4,6,7)\) and \((4,7,8)\) changes the fourth component. Interchanging \(6\) and \(8\) if necessary, suppose this occurs for \((4,6,7)\), and let
\[
(\Omega_1,\Omega_2,\Omega_3,\Omega_4)
=
(\Gamma_1,\Gamma_2,\Gamma_3,\Gamma_4)^{(4,6,7)}.
\]
Then \(\Omega_4=H_6\). The involution $(4,6,7)$ does not change \((2,6),(2,8)\), or \((2,7)\), so
\[
(2,6),(2,8)\in E(\Omega_3),\quad \text{and} \quad (2,7)\notin E(\Omega_3).
\]
Since  \(\Omega_4=H_6\) one can see that up to weak equivalence, this is the same as Case II. Hence this partition is weakly equivalent to a cycle-free \(4\)-partition whose fourth component is \(H_2=TS_4\).

This concludes the proof.
\end{proof}

%\begin{lemma}
%Let $(\Gamma_1,\Gamma_2,\Gamma_3,\Gamma_4)$ be a cycle-free $4$-partition of the complete graph $K_8$  such that $\Gamma_4$ is the graph  $T_{19}$. Then $(\Gamma_1,\Gamma_2,\Gamma_3,\Gamma_4)$ is weakly  equivalent to a cycle-free $4$-partition $(\Delta_1,\Delta_2,\Delta_3,\Delta_4)$ such that $\Delta_4$ is the twin-star graph $TS_4=T_{22}$.
%\end{lemma}
%\begin{proof} This result was proved using MATLAB.
%\end{proof}

%\begin{remark} It is somehow unsatisfying that we had to use MATLAB in order to settle the case $d=4$. However, it is worth noticing that even in the particular case $\Gamma_4=T_{19}$, our computational resources were stretched to the maximum (the program that we have runs for almost $3$ days on a regular desktop). So, the combinatorial reduction to the case $\Gamma_4=T_{19}$ was a necessary step for our approch to work. It would still be interesting to find a direct proof without the help of computers.
%\end{remark}

\begin{proposition} The twin-star hypothesis $\mathcal{TS}(4)$ holds true. \label{propsum}
\end{proposition}
\begin{proof}
This follows directly from the discussions in this section.
\end{proof}

\begin{corollary} Conjecture \ref{mainconj} is true for $d=4$. \label{corlsum1}
\end{corollary}
\begin{proof} This statement follows from Theorem \ref{mainth} and Proposition \ref{propsum}.
\end{proof}
Using the results and definitions  from \cite{edge,sta2} we have.
\begin{corollary} If $dim_k(V_4)=4$ then  $dim_k(\Lambda_{V_4}^{S^2}[8])=1$.
\end{corollary}
\begin{proof}  This is a consequence of Corollary \ref{corlsum1} and the discussions from \cite{edge}.
\end{proof}

\appendix

\maketitle

\section{Proof of Lemma \ref{lemma3}}

\label{appendix1}

In this section we give a proof for Lemma \ref{lemma3}. The plan is to show that for each $1\leq i\leq 2d-1$ we have that $(i,i+1)\circ E_d$ is involution equivalent to $E_d$, and for each $1\leq i\leq d-1$ we have that $(i,i+1)\odot E_d$ is involution equivalent to $E_d$. Using the fact that $S_{n}$ is generated by the set $\{(i,i+1)|1\leq i\leq n-1\}$, this proves or statement. First we need to elaborate on Remark \ref{remlocal}.

\begin{definition} Let $\{x_1,x_2,\dots,x_{2t}\}\subset \{1,2,\dots, 2d\}$. We denote by $K_{2t}(x_1,x_2,\dots,x_{2t})$ the complete subgraph of $K_{2d}$ determined by the vertices $x_1,x_2,\dots, x_{2t}$.
\end{definition}

\begin{remark} As a graph $K_{2t}(x_1,x_2,\dots,x_{2t})$ is isomorphic to $K_{2t}$. In general a $d$-partition $(\Gamma_1,\dots, \Gamma_d)$ of $K_{2d}$ will induce a $d$-partition of $K_{2t}$. However, if the edges of $K_{2t}(x_1,x_2,\dots,x_{2t})$ belong only to the graphs $\Gamma_{i_1}, \dots,\Gamma_{i_t}$ for some $1\leq i_1<i_2<\dots <i_t\leq 2d$, then the restriction of $(\Gamma_1,\dots, \Gamma_d)$ will induce a $t$-partition  of $K_{2t}(x_1,x_2,\dots,x_{2t})$. Moreover, if the $d$-partition $(\Gamma_1,\dots, \Gamma_d)$ is cycle-free then the induced $t$-partition of $K_{2t}$ will also be cycle-free.

For example, if we restrict the $d$-partition $E_d$ of the complete graph $K_{2d}$ to the subgraph  $K_{2d-2}(1,2,\dots,2d-3,2d-2)$ then the resulting $(d-1)$-partition is $E_{d-1}$.
\end{remark}

\begin{remark} In the proof of the next two lemmas we will use the fact that if $(\Gamma_1,\Gamma_2)$ is a cycle-free $2$-partition for the graph $K_4$ then $(\Gamma_1,\Gamma_2)$ is involution equivalent to $E_2$. Moreover, if $(\Gamma_1,\Gamma_2,\Gamma_3)$ is a cycle-free $3$-partition for the graph $K_6$ then $(\Gamma_1,\Gamma_2,\Gamma_3)$ is involution equivalent to $E_3$. Obviously these two statements are still true for the complete graph $K_4(x_1,x_2,x_3,x_4)$ and $K_6(x_1,x_2,x_3,x_4,x_5,x_6)$ respectively.
\end{remark}

\begin{remark}\label{reKsubG}

Let $\{x_1,x_2,\dots,x_{2t}\}\subset \{1,2,\dots, 2d\}$. Consider $\mathcal{P}=(\Gamma_1,\dots,\Gamma_d)$ and $\mathcal{Q}=(\Theta_1,\dots,\Theta_d)$ two cycle-free $d$-partitions of $K_{2d}$ such that the $\mathcal{P}$ and $\mathcal{Q}$ coincide outside of the complete graph $K_{2t}(x_1,x_2,\dots,x_{2t})$. Moreover, suppose that the restrictions of $\mathcal{P}$ and $\mathcal{Q}$  to the complete graph $K_{2t}(x_1,x_2,\dots,x_{2t})$ are two cycle-free $t$-partitions $\overline{\mathcal{P}}$ and $\overline{\mathcal{Q}}$ of the complete graph $K_{2t}$. If $\overline{\mathcal{P}}$ and $\overline{\mathcal{Q}}$ are involution equivalent then $\mathcal{P}$ and $\mathcal{Q}$ are involution equivalent.
\end{remark}

\begin{lemma} Let $1\leq a\leq d$ then $(2a-1,2a)\circ E_d$ is involution equivalent to $E_d$. \label{lemmaS2d1}
\end{lemma}
\begin{proof}  First one should notice that $E_d$ and $(2a-1,2a)\circ E_d$ coincide everywhere except on the complete subgraphs  $K_4(2a-1,2a,2x-1,2x)$ where $1\leq x\leq d$, $x\neq a$.

Take $x$ such that $1\leq a<x\leq d$. If we restrict the $d$-partition $E_d$ to the complete graph $K_4(2a-1,2a,2x-1,2x)$ then we get the $2$-partition of $K_4$ from Figure \ref{EdK4ab}.

 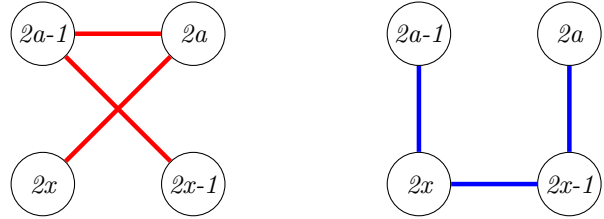
\begin{figure}[H]
	\centering
	\begin{tikzpicture}
		[scale=1,auto=left,every node/.style={shape = circle, draw, fill = white,minimum size = 4pt, inner sep=1pt}]%baseline=(a.center)]%{circle,fill=black}]
		%\tikzset{VertexStyle/.style = {shape = circle,fill = black,minimum size = 9mm,inner sep=2pt}}
		%\node[shape=circle,draw=black,minimum size = 24pt,inner sep=0.5pt] (nt1) at (9,2) {$ v_{t-2}$};
		\node[shape=circle,draw=black,minimum size = 24pt,inner sep=0.5pt] (n1) at (-2,3) {{\small \it 2a-1}};
		\node[shape=circle,draw=black,minimum size = 24pt,inner sep=0.5pt] (n2) at (0,3) {{\small \it 2a}};
			\node[shape=circle,draw=black,minimum size = 24pt,inner sep=0.5pt] (n3) at (0,1) {{\small \it 2x-1}};
				\node[shape=circle,draw=black,minimum size = 24pt,inner sep=0.5pt] (n4) at (-2,1) {{\small \it 2x}};
		%\node (n1) at (-2,3) {2a-1};
		%\node (n2) at (0,3)  {2a};
		%\node (n3) at (0,1)  {2b-1};
		%\node (n4) at (-2,1)  {2b};
		%\node (front) at (-2,0.5)   {$\Gamma_1$}
		\foreach \from/\to in {n1/n2,n1/n3,n2/n4}
		\draw[line width=0.6mm,red]  (\from) -- (\to);	
		\node[shape=circle,draw=black,minimum size = 24pt,inner sep=0.5pt] (n12) at (3,3) {{\small \it 2a-1}};
		\node[shape=circle,draw=black,minimum size = 24pt,inner sep=0.5pt] (n22) at (5,3) {{\small \it 2a}};
		\node[shape=circle,draw=black,minimum size = 24pt,inner sep=0.5pt] (n32) at (5,1) {{\small \it 2x-1}};
				\node[shape=circle,draw=black,minimum size = 24pt,inner sep=0.5pt] (n42) at (3,1) {{\small \it 2x}};
		%\node (n12) at (3,3) {{\small \it 2a-1}};
		%\node (n22) at (5,3)  {{\small \it 2a}};
		%\node (n32) at (5,1)  {2b-1};
		%\node (n42) at (3,1)  {2b};
		%\node (front) at (-2,0.5)   {$\Gamma_1$}
		\foreach \from/\to in {n12/n42,n22/n32,n32/n42}
		\draw[line width=0.6mm,blue]  (\from) -- (\to);	
	\end{tikzpicture}
	\caption{ $E_d$ restricted to  $K_4(2a-1,2a,2x-1,2x)$ } \label{EdK4ab}
\end{figure}

Next, if we restrict the $d$-partition $(2a-1,2a)\circ E_d$ to the complete graph $K_4(2a-1,2a,2x-1,2x)$ we obtain the $2$-partition of $K_4$ from Figure \ref{EdK4ab2}.

    \begin{figure}[h!]
	\centering
	\begin{tikzpicture}
		[scale=1,auto=left,every node/.style={shape = circle, draw, fill = white,minimum size = 4pt, inner sep=1pt}]%baseline=(a.center)]%{circle,fill=black}]
		%\tikzset{VertexStyle/.style = {shape = circle,fill = black,minimum size = 9mm,inner sep=2pt}}
				\node[shape=circle,draw=black,minimum size = 24pt,inner sep=0.5pt] (n1) at (-2,3) {{\small \it 2a-1}};
		\node[shape=circle,draw=black,minimum size = 24pt,inner sep=0.5pt] (n2) at (0,3) {{\small \it 2a}};
			\node[shape=circle,draw=black,minimum size = 24pt,inner sep=0.5pt] (n3) at (0,1) {{\small \it 2x-1}};
				\node[shape=circle,draw=black,minimum size = 24pt,inner sep=0.5pt] (n4) at (-2,1) {{\small \it 2x}};
		%\node (front) at (-2,0.5)   {$\Gamma_1$}
		\foreach \from/\to in {n1/n2,n1/n4,n2/n3}
			\draw[line width=0.6mm,red]  (\from) -- (\to);	
\node[shape=circle,draw=black,minimum size = 24pt,inner sep=0.5pt] (n12) at (3,3) {{\small \it 2a-1}};
		\node[shape=circle,draw=black,minimum size = 24pt,inner sep=0.5pt] (n22) at (5,3) {{\small \it 2a}};
		\node[shape=circle,draw=black,minimum size = 24pt,inner sep=0.5pt] (n32) at (5,1) {{\small \it 2x-1}};
				\node[shape=circle,draw=black,minimum size = 24pt,inner sep=0.5pt] (n42) at (3,1) {{\small \it 2x}};
		%\node (front) at (-2,0.5)   {$\Gamma_1$}
		\foreach \from/\to in {n12/n32,n22/n42,n32/n42}
		\draw[line width=0.6mm,blue]  (\from) -- (\to);	
		%\node[state,above of=B1] (C1) {$C_1$};
	\end{tikzpicture}
	\caption{ $(2a-1,2a)\circ E_d$  restricted to  $K_4(2a-1,2a,2x-1,2x)$  } \label{EdK4ab2}
\end{figure}
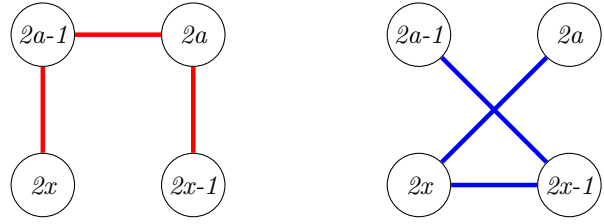
Since every cycle-free $2$-partitions of the graph $K_4$ are involution equivalent, we have that the $2$-partitions from Figures \ref{EdK4ab} and \ref{EdK4ab2} are involution equivalent.

A similar argument shows that if $1\leq y<a\leq d$ then the restriction of $E_d$ and of $(2a-1,2a)\circ E_d$ to the complete graph $K_4(2y-1,2y,2a-1,2a)$ are cycle-free $2$-partitions that are involution equivalent.

To summarize, we have that the cycle-free $d$-partitions $E_d$ and $(2a-1,2a)\circ E_d$ coincide everywhere except on the sub graphs $K_4(2a-1,2a,2x-1,2x)$ for all $1\leq a<x\leq d$, and on the sub graphs $K_4(2y-1,2y,2a-1,2a)$ for all $1\leq y<a\leq d$.  Moreover, on these subgraphs the induced $2$-partitions are involution equivalent. Using repeatedly Remark \ref{reKsubG} we get that $E_d$ and $(2a-1,2a)\circ E_d$ are involution equivalent.
\end{proof}

\begin{lemma} Let $1\leq a\leq d$ then $(2a,2a+1)\circ E_d$ is involution equivalent to $E_d$. \label{lemmaS2d2}
\end{lemma}
\begin{proof} First one should notice that $E_d$ and $(2a,2a+1)\circ E_d$ coincide everywhere except on the complete subgraphs $K_6(2x-1,2x,2a-1,2a,2a+1,2a+2)$ where $1\leq x\leq d$ and $a\neq x\neq a+1$.

Take $x$ such that $1\leq x<a\leq d$. In Figure \ref{EdK6ax} we have the restriction of $E_d$ to the graph $K_6(2x-1,2x,2a-1,2a,2a+1,2a+2)$.
    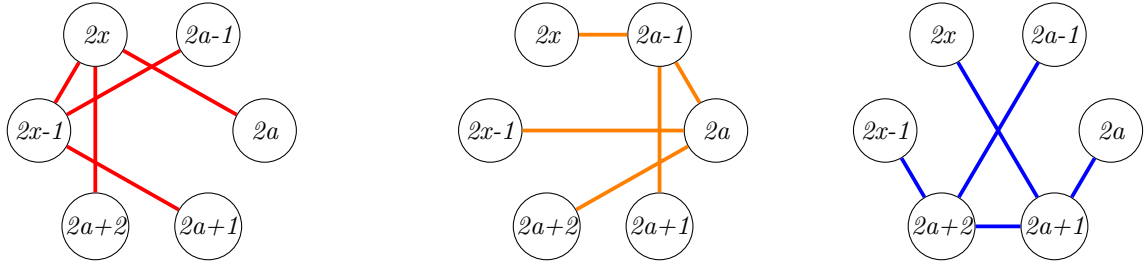
\begin{figure}[H]
	\centering
	\begin{tikzpicture}
		[scale=1.5,auto=left,every node/.style={shape = circle, draw, fill = white,minimum size = 4pt, inner sep=1pt}]%baseline=(a.center)]%{circle,fill=black}]
		%\tikzset{VertexStyle/.style = {shape = circle,fill = black,minimum size = 9mm,inner sep=2pt}}

\node[shape=circle,draw=black,minimum size = 24pt,inner sep=0.5pt] (n1) at (-4,-1) {{\small \it 2x-1}};
\node[shape=circle,draw=black,minimum size = 24pt,inner sep=0.5pt] (n2) at (-3.5,-0.15) {{\small \it 2x}};
\node[shape=circle,draw=black,minimum size = 24pt,inner sep=0.5pt] (n3) at (-2.5,-0.15) {{\small \it 2a-1}};
\node[shape=circle,draw=black,minimum size = 24pt,inner sep=0.5pt] (n4) at (-2,-1) {{\small \it 2a}};
\node[shape=circle,draw=black,minimum size = 24pt,inner sep=0.5pt] (n5) at (-2.5,-1.85) {{\small \it 2a+1}};
\node[shape=circle,draw=black,minimum size = 24pt,inner sep=0.5pt] (n6) at (-3.5,-1.85) {{\small \it 2a+2}};
    %\node (n1) at (-4,-1) {2a-1};
		%\node (n2) at (-3.5,-0.15)  {2a};
		%\node (n3) at (-2.5,-0.15)  {2b-1};
		%\node (n4) at (-2,-1)  {2b};
		%\node (n5) at (-2.5,-1.85)  {2c-1};
		%\node (n6) at (-3.5,-1.85)  {2c};
		%\node (front) at (-2,0.5)   {$\Gamma_1$}
		\foreach \from/\to in {n1/n2,n1/n3,n1/n5,n2/n4,n2/n6}
		\draw[line width=0.5mm,red]  (\from) -- (\to);	
		%\node[state,above of=B1] (C1) {$C_1$};
\node[shape=circle,draw=black,minimum size = 24pt,inner sep=0.5pt] (n11) at (0,-1) {{\small \it 2x-1}};
\node[shape=circle,draw=black,minimum size = 24pt,inner sep=0.5pt] (n21) at (0.5,-0.15) {{\small \it 2x}};
\node[shape=circle,draw=black,minimum size = 24pt,inner sep=0.5pt] (n31) at (1.5,-0.15) {{\small \it 2a-1}};
\node[shape=circle,draw=black,minimum size = 24pt,inner sep=0.5pt] (n41) at (2,-1) {{\small \it 2a}};
\node[shape=circle,draw=black,minimum size = 24pt,inner sep=0.5pt] (n51) at (1.5,-1.85) {{\small \it 2a+1}};
\node[shape=circle,draw=black,minimum size = 24pt,inner sep=0.5pt] (n61) at (0.5,-1.85) {{\small \it 2a+2}};

		%\node (front) at (-2,0.5)   {$\Gamma_1$}   3     5     7    11    12
		\foreach \from/\to in {n31/n41,n31/n21,n31/n51,n41/n11,n41/n61}
		\draw[line width=0.5mm,orange]  (\from) -- (\to);	
		
		\node[shape=circle,draw=black,minimum size = 24pt,inner sep=0.5pt] (n12) at (3.5,-1) {{\small \it 2x-1}};
\node[shape=circle,draw=black,minimum size = 24pt,inner sep=0.5pt] (n22) at (4,-0.15) {{\small \it 2x}};
\node[shape=circle,draw=black,minimum size = 24pt,inner sep=0.5pt] (n32) at (5,-0.15) {{\small \it 2a-1}};
\node[shape=circle,draw=black,minimum size = 24pt,inner sep=0.5pt] (n42) at (5.5,-1) {{\small \it 2a}};
\node[shape=circle,draw=black,minimum size = 24pt,inner sep=0.5pt] (n52) at (5,-1.85) {{\small \it 2a+1}};
\node[shape=circle,draw=black,minimum size = 24pt,inner sep=0.5pt] (n62) at (4,-1.85) {{\small \it 2a+2}};
		
		%\node (front) at (-2,0.5)   {$\Gamma_1$} 2     4     8     9    14
		\foreach \from/\to in {n52/n62,n52/n22,n52/n42,n62/n12,n62/n32}
		\draw[line width=0.5mm,blue]  (\from) -- (\to);
		
		%\node[state,above of=B1] (C1) {$C_1$};
	\end{tikzpicture}
	\caption{$E_d$  restricted to  $K_6(2x-1,2x,2a-1,2a,2a+1,2a+2)$ } \label{EdK6ax}
\end{figure}
In Figure \ref{EdK6ax2} we have the restriction of $(2a,2a+1)\circ E_d$ to the graph $K_6(2x-1,2x,2a-1,2a,2a+1,2a+2)$.

   \begin{figure}[H]
	\centering
	\begin{tikzpicture}
		[scale=1.5,auto=left,every node/.style={shape = circle, draw, fill = white,minimum size = 4pt, inner sep=1pt}]%baseline=(a.center)]%{circle,fill=black}]
		%\tikzset{VertexStyle/.style = {shape = circle,fill = black,minimum size = 9mm,inner sep=2pt}}

\node[shape=circle,draw=black,minimum size = 24pt,inner sep=0.5pt] (n1) at (-4,-1) {{\small \it 2x-1}};
\node[shape=circle,draw=black,minimum size = 24pt,inner sep=0.5pt] (n2) at (-3.5,-0.15) {{\small \it 2x}};
\node[shape=circle,draw=black,minimum size = 24pt,inner sep=0.5pt] (n3) at (-2.5,-0.15) {{\small \it 2a-1}};
\node[shape=circle,draw=black,minimum size = 24pt,inner sep=0.5pt] (n5) at (-2,-1) {{\small \it 2a}};
\node[shape=circle,draw=black,minimum size = 24pt,inner sep=0.5pt] (n4) at (-2.5,-1.85) {{\small \it 2a+1}};
\node[shape=circle,draw=black,minimum size = 24pt,inner sep=0.5pt] (n6) at (-3.5,-1.85) {{\small \it 2a+2}};
    %\node (n1) at (-4,-1) {2a-1};
		%\node (n2) at (-3.5,-0.15)  {2a};
		%\node (n3) at (-2.5,-0.15)  {2b-1};
		%\node (n4) at (-2,-1)  {2b};
		%\node (n5) at (-2.5,-1.85)  {2c-1};
		%\node (n6) at (-3.5,-1.85)  {2c};
		%\node (front) at (-2,0.5)   {$\Gamma_1$}
		\foreach \from/\to in {n1/n2,n1/n3,n1/n5,n2/n4,n2/n6}
		\draw[line width=0.5mm,red]  (\from) -- (\to);	
		%\node[state,above of=B1] (C1) {$C_1$};
\node[shape=circle,draw=black,minimum size = 24pt,inner sep=0.5pt] (n11) at (0,-1) {{\small \it 2x-1}};
\node[shape=circle,draw=black,minimum size = 24pt,inner sep=0.5pt] (n21) at (0.5,-0.15) {{\small \it 2x}};
\node[shape=circle,draw=black,minimum size = 24pt,inner sep=0.5pt] (n31) at (1.5,-0.15) {{\small \it 2a-1}};
\node[shape=circle,draw=black,minimum size = 24pt,inner sep=0.5pt] (n51) at (2,-1) {{\small \it 2a}};
\node[shape=circle,draw=black,minimum size = 24pt,inner sep=0.5pt] (n41) at (1.5,-1.85) {{\small \it 2a+1}};
\node[shape=circle,draw=black,minimum size = 24pt,inner sep=0.5pt] (n61) at (0.5,-1.85) {{\small \it 2a+2}};

		%\node (front) at (-2,0.5)   {$\Gamma_1$}   3     5     7    11    12
		\foreach \from/\to in {n31/n41,n31/n21,n31/n51,n41/n11,n41/n61}
		\draw[line width=0.5mm,orange]  (\from) -- (\to);	
		
		\node[shape=circle,draw=black,minimum size = 24pt,inner sep=0.5pt] (n12) at (3.5,-1) {{\small \it 2x-1}};
\node[shape=circle,draw=black,minimum size = 24pt,inner sep=0.5pt] (n22) at (4,-0.15) {{\small \it 2x}};
\node[shape=circle,draw=black,minimum size = 24pt,inner sep=0.5pt] (n32) at (5,-0.15) {{\small \it 2a-1}};
\node[shape=circle,draw=black,minimum size = 24pt,inner sep=0.5pt] (n52) at (5.5,-1) {{\small \it 2a}};
\node[shape=circle,draw=black,minimum size = 24pt,inner sep=0.5pt] (n42) at (5,-1.85) {{\small \it 2a+1}};
\node[shape=circle,draw=black,minimum size = 24pt,inner sep=0.5pt] (n62) at (4,-1.85) {{\small \it 2a+2}};
		
		%\node (front) at (-2,0.5)   {$\Gamma_1$} 2     4     8     9    14
		\foreach \from/\to in {n52/n62,n52/n22,n52/n42,n62/n12,n62/n32}
		\draw[line width=0.5mm,blue]  (\from) -- (\to);
		
		%\node[state,above of=B1] (C1) {$C_1$};
	\end{tikzpicture}
	\caption{$(2a,2a+1)\circ E_d$  restricted to  $K_6(2x-1,2x,2a-1,2a,2a+1,2a+2)$ } \label{EdK6ax2}
\end{figure}
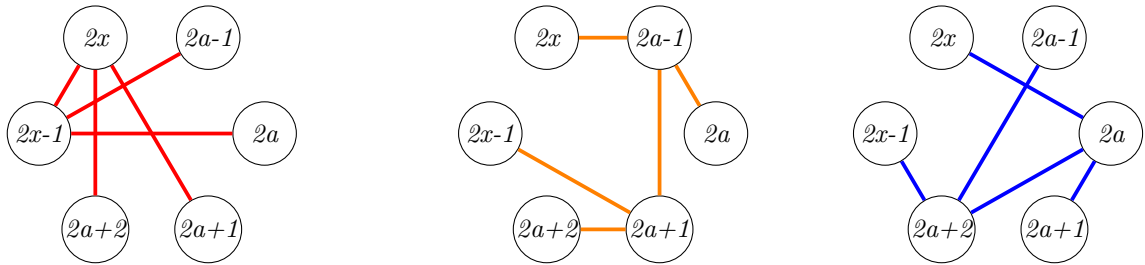

Finally, in Figure \ref{EdK6ax3} we have a $3$-partition of the complete graph $K_6(2x-1,2x,2a-1,2a,2a+1,2a+2)$ which will be used to show that  $E_d$ and $(2a,2a+1)\circ E_d$ are involution equivalent. Indeed, first notice that since the $3$-partitions from Figure \ref{EdK6ax}, \ref{EdK6ax2} and \ref{EdK6ax3} are cycle-free $3$-partitions of the complete graph $K_6(2x-1,2x,2a-1,2a,2a+1,2a+2)$, they  are involution equivalent.

    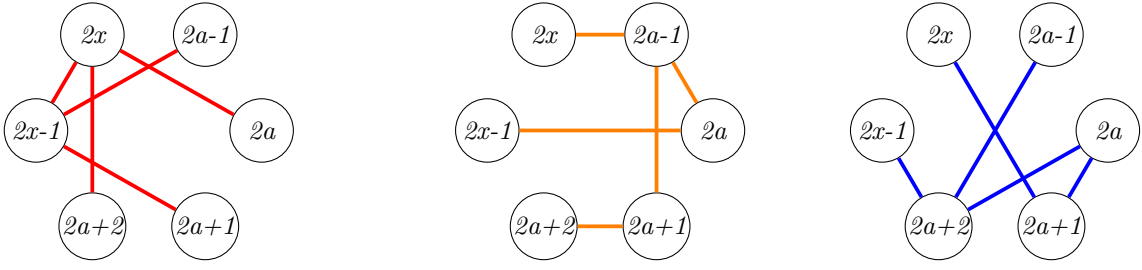
\begin{figure}[h!]
	\centering
	\begin{tikzpicture}
		[scale=1.5,auto=left,every node/.style={shape = circle, draw, fill = white,minimum size = 4pt, inner sep=1pt}]%baseline=(a.center)]%{circle,fill=black}]
		%\tikzset{VertexStyle/.style = {shape = circle,fill = black,minimum size = 9mm,inner sep=2pt}}
		
		\node[shape=circle,draw=black,minimum size = 24pt,inner sep=0.5pt] (n1) at (-4,-1) {{\small \it 2x-1}};
\node[shape=circle,draw=black,minimum size = 24pt,inner sep=0.5pt] (n2) at (-3.5,-0.15) {{\small \it 2x}};
\node[shape=circle,draw=black,minimum size = 24pt,inner sep=0.5pt] (n3) at (-2.5,-0.15) {{\small \it 2a-1}};
\node[shape=circle,draw=black,minimum size = 24pt,inner sep=0.5pt] (n4) at (-2,-1) {{\small \it 2a}};
\node[shape=circle,draw=black,minimum size = 24pt,inner sep=0.5pt] (n5) at (-2.5,-1.85) {{\small \it 2a+1}};
\node[shape=circle,draw=black,minimum size = 24pt,inner sep=0.5pt] (n6) at (-3.5,-1.85) {{\small \it 2a+2}};
		
		%\node (front) at (-2,0.5)   {$\Gamma_1$}
		\foreach \from/\to in {n1/n2,n1/n3,n1/n5,n2/n4,n2/n6}
		\draw[line width=0.5mm,red]  (\from) -- (\to);	
		%\node[state,above of=B1] (C1) {$C_1$};
		
		\node[shape=circle,draw=black,minimum size = 24pt,inner sep=0.5pt] (n11) at (0,-1) {{\small \it 2x-1}};
\node[shape=circle,draw=black,minimum size = 24pt,inner sep=0.5pt] (n21) at (0.5,-0.15) {{\small \it 2x}};
\node[shape=circle,draw=black,minimum size = 24pt,inner sep=0.5pt] (n31) at (1.5,-0.15) {{\small \it 2a-1}};
\node[shape=circle,draw=black,minimum size = 24pt,inner sep=0.5pt] (n41) at (2,-1) {{\small \it 2a}};
\node[shape=circle,draw=black,minimum size = 24pt,inner sep=0.5pt] (n51) at (1.5,-1.85) {{\small \it 2a+1}};
\node[shape=circle,draw=black,minimum size = 24pt,inner sep=0.5pt] (n61) at (0.5,-1.85) {{\small \it 2a+2}};

		%\node (front) at (-2,0.5)   {$\Gamma_1$}   3     5     7    11    12
		\foreach \from/\to in {n21/n31,n31/n41,n31/n51,n41/n11,n51/n61}
		\draw[line width=0.5mm,orange]  (\from) -- (\to);	
\node[shape=circle,draw=black,minimum size = 24pt,inner sep=0.5pt] (n12) at (3.5,-1) {{\small \it 2x-1}};
\node[shape=circle,draw=black,minimum size = 24pt,inner sep=0.5pt] (n22) at (4,-0.15) {{\small \it 2x}};
\node[shape=circle,draw=black,minimum size = 24pt,inner sep=0.5pt] (n32) at (5,-0.15) {{\small \it 2a-1}};
\node[shape=circle,draw=black,minimum size = 24pt,inner sep=0.5pt] (n42) at (5.5,-1) {{\small \it 2a}};
\node[shape=circle,draw=black,minimum size = 24pt,inner sep=0.5pt] (n52) at (5,-1.85) {{\small \it 2a+1}};
\node[shape=circle,draw=black,minimum size = 24pt,inner sep=0.5pt] (n62) at (4,-1.85) {{\small \it 2a+2}};

		%\node (front) at (-2,0.5)   {$\Gamma_1$} 2     4     8     9    14
		\foreach \from/\to in {n12/n62,n62/n32,n22/n52,n62/n42,n42/n52}
		\draw[line width=0.5mm,blue]  (\from) -- (\to);
		
		%\node[state,above of=B1] (C1) {$C_1$};
	\end{tikzpicture}
	\caption{Another $3$-partition of   $K_6(2x-1,2x,2a-1,2a,2a+1,2a+2)$ } \label{EdK6ax3}
\end{figure}

Moreover, notice the restrictions of the partitions $(2a,2a+1)\circ E_d$ (i.e. Figure \ref{EdK6ax2}) and of the partition from Figure \ref{EdK6ax3} coincide on the sub-graph $K_4(2a-1,2a,2a+1,2a+2)$. So, we can apply repeatedly the involution equivalence of these $3$-partitions and Remark \ref{reKsubG} to rearrange  the edges $(2x-1,2x)$, $(2x-1,2a-1)$, $(2x-1,2a+1)$ $(2x,2a)$ and $(2x,2a+2)$ of the partition $(2a,2a+1)\circ E_d$ so that they coincide with those of the partition $E_d$ (see Figures \ref{EdK6ax2} and \ref{EdK6ax3}). A similar argument can be used for the case  $1\leq a+1<y\leq d$ and the corresponding edges $(2a-1,2y)$, $(2a,2y-1)$, $(2a+1,2y)$, $(2a+2,2y-1)$, and $(2y-1,2y)$ of the partition $(2a,2a+1)\circ E_d$.

To conclude, notice that $E_d$ and $(2a,2a+1)\circ E_d$ induce on the complete graph $K_4(2a-1,2a,2a+1,2a+2)$ the cycle free $2$-partitions from Figures \ref{EdK4aa} and \ref{EdK4aa2} respectively, which are involution equivalent.

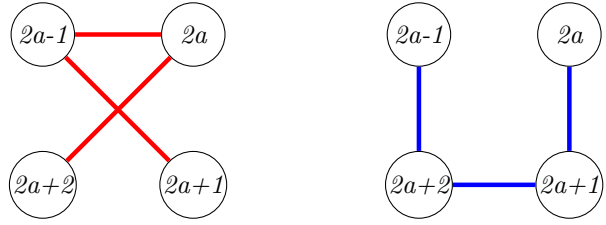
\begin{figure}[H]
	\centering
	\begin{tikzpicture}
		[scale=1,auto=left,every node/.style={shape = circle, draw, fill = white,minimum size = 4pt, inner sep=1pt}]%baseline=(a.center)]%{circle,fill=black}]
		%\tikzset{VertexStyle/.style = {shape = circle,fill = black,minimum size = 9mm,inner sep=2pt}}
				\node[shape=circle,draw=black,minimum size = 24pt,inner sep=0.5pt] (n1) at (-2,3) {{\small \it 2a-1}};
		\node[shape=circle,draw=black,minimum size = 24pt,inner sep=0.5pt] (n2) at (0,3) {{\small \it 2a}};
			\node[shape=circle,draw=black,minimum size = 24pt,inner sep=0.5pt] (n3) at (0,1) {{\small \it 2a+1}};
				\node[shape=circle,draw=black,minimum size = 24pt,inner sep=0.5pt] (n4) at (-2,1) {{\small \it 2a+2}};
		%\node (front) at (-2,0.5)   {$\Gamma_1$}
		\foreach \from/\to in {n1/n2,n1/n3,n4/n2}
			\draw[line width=0.6mm,red]  (\from) -- (\to);	
\node[shape=circle,draw=black,minimum size = 24pt,inner sep=0.5pt] (n12) at (3,3) {{\small \it 2a-1}};
		\node[shape=circle,draw=black,minimum size = 24pt,inner sep=0.5pt] (n22) at (5,3) {{\small \it 2a}};
		\node[shape=circle,draw=black,minimum size = 24pt,inner sep=0.5pt] (n32) at (5,1) {{\small \it 2a+1}};
				\node[shape=circle,draw=black,minimum size = 24pt,inner sep=0.5pt] (n42) at (3,1) {{\small \it 2a+2}};
		%\node (front) at (-2,0.5)   {$\Gamma_1$}
		\foreach \from/\to in {n12/n42,n32/n42,n32/n22}
		\draw[line width=0.6mm,blue]  (\from) -- (\to);	
		%\node[state,above of=B1] (C1) {$C_1$};
	\end{tikzpicture}
	\caption{ $E_d$  restricted to  $K_4(2a-1,2a,2a+1,2a+2)$  } \label{EdK4aa}
\end{figure}

    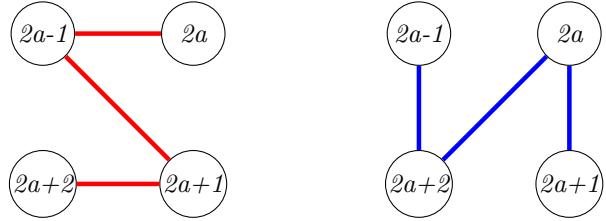
\begin{figure}[H]
	\centering
	\begin{tikzpicture}
		[scale=1,auto=left,every node/.style={shape = circle, draw, fill = white,minimum size = 4pt, inner sep=1pt}]%baseline=(a.center)]%{circle,fill=black}]
		%\tikzset{VertexStyle/.style = {shape = circle,fill = black,minimum size = 9mm,inner sep=2pt}}
				\node[shape=circle,draw=black,minimum size = 24pt,inner sep=0.5pt] (n1) at (-2,3) {{\small \it 2a-1}};
		\node[shape=circle,draw=black,minimum size = 24pt,inner sep=0.5pt] (n2) at (0,3) {{\small \it 2a}};
			\node[shape=circle,draw=black,minimum size = 24pt,inner sep=0.5pt] (n3) at (0,1) {{\small \it 2a+1}};
				\node[shape=circle,draw=black,minimum size = 24pt,inner sep=0.5pt] (n4) at (-2,1) {{\small \it 2a+2}};
		%\node (front) at (-2,0.5)   {$\Gamma_1$}
		\foreach \from/\to in {n1/n2,n1/n3,n4/n3}
			\draw[line width=0.6mm,red]  (\from) -- (\to);	
\node[shape=circle,draw=black,minimum size = 24pt,inner sep=0.5pt] (n12) at (3,3) {{\small \it 2a-1}};
		\node[shape=circle,draw=black,minimum size = 24pt,inner sep=0.5pt] (n22) at (5,3) {{\small \it 2a}};
		\node[shape=circle,draw=black,minimum size = 24pt,inner sep=0.5pt] (n32) at (5,1) {{\small \it 2a+1}};
				\node[shape=circle,draw=black,minimum size = 24pt,inner sep=0.5pt] (n42) at (3,1) {{\small \it 2a+2}};
		%\node (front) at (-2,0.5)   {$\Gamma_1$}
		\foreach \from/\to in {n12/n42,n22/n42,n32/n22}
		\draw[line width=0.6mm,blue]  (\from) -- (\to);	
		%\node[state,above of=B1] (C1) {$C_1$};
	\end{tikzpicture}
	\caption{ $(2a,2a+1)\circ E_d$  restricted to  $K_4(2a-1,2a,2a+1,2a+2)$  } \label{EdK4aa2}
\end{figure}

And so from  Remark \ref{reKsubG}  we have that $E_d$ and $(2a,2a+1)\circ E_d$ are involution equivalent.

\end{proof}

\begin{lemma} Let $1\leq c\leq d$. Take $\tau=(c,c+1)\in S_d$, and $\sigma=(2c-1,2c+1)(2c,2c+2)\in S_{2d}$  then $\tau\odot E_d= \sigma\circ E_d$. \label{lemmaSd}
\end{lemma}
\begin{proof}
This follows directly from the definition of $E_d$ and of the actions of $S_d$ and $S_{2d}$ on $\mathcal{P}_d^{cf}(K_{2d})$.
\end{proof}

\begin{remark} Since $S_{2d}$ is generated by the transpositions $(a,a+1)$ for all $1\leq a<2d$ it follows from Lemmas \ref{lemmaS2d1} and \ref{lemmaS2d2} that for all $\sigma\in S_{2d}$ the partition $\sigma\circ E_d$ is involution equivalent to $E_d$. Moreover, since $S_d$ is generated by $(c,c+1)$ for all $1\leq c<d$ using Lemma \ref{lemmaSd} we get that for all $\tau\in S_d$ the partition $\tau\odot E_d$ is involution equivalent to $E_d$. Combining these results we get the proof for Lemma \ref{lemma3}.

\end{remark}

\section{Trees with $8$ vertices}

\label{appendix2}
It is known that there are  $23$ isomorphism types of trees with $8$ vertices. For convenience, in this section we present their list taken from \cite{h} page 33.

%\begin{figure}[H]
%\begin{tabular}{ ||c|c|| }
 \begin{longtable}[c]{| c | c |}
\caption{Table of trees with $8$ vertices \label{tablepart1}}\\
 \hline
& Trees with $8$ vertices\\
 \hline
 \;&\;\\
 $T_1$ & 
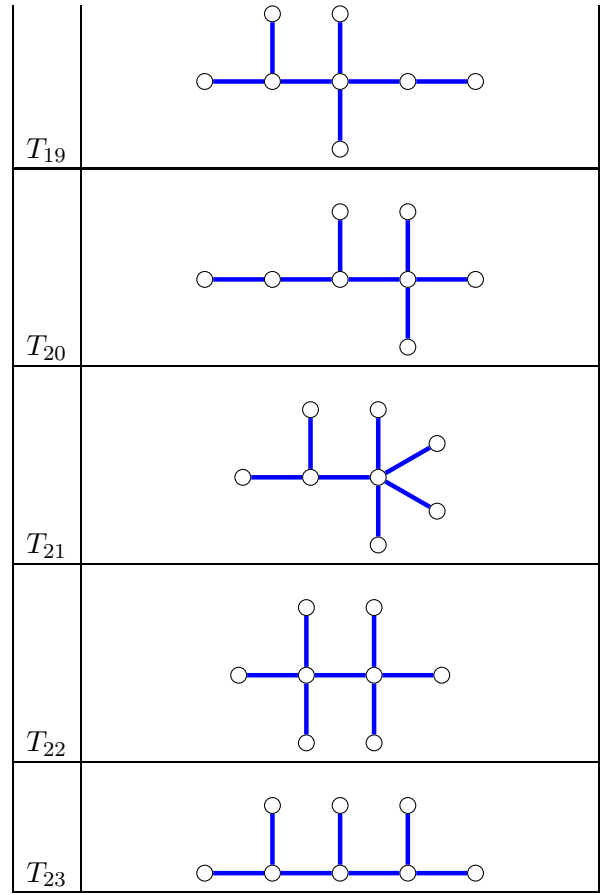
\begin{tikzpicture}
  [scale=0.9,auto=left]%,every node/.style={shape = circle, draw, fill = white,minimum size = 14pt, inner sep=0.3pt}]%baseline=(a.center)]%{circle,fill=black}]
	%\tikzset{VertexStyle/.style = {shape = circle,fill = black,minimum size = 9mm,inner sep=2pt}}
	\node[shape=circle,draw=black,minimum size = 6pt,inner sep=0.3pt] (n1) at (0,0) {};
	\node[shape=circle,draw=black,minimum size = 6pt,inner sep=0.3pt] (n2) at (1,0) {};
	\node[shape=circle,draw=black,minimum size = 6pt,inner sep=0.3pt] (n3) at (2,0) {};
  \node[shape=circle,draw=black,minimum size = 6pt,inner sep=0.3pt] (n4) at (3,0) {};
	\node[shape=circle,draw=black,minimum size = 6pt,inner sep=0.3pt] (n5) at (4,0) {};
  \node[shape=circle,draw=black,minimum size = 6pt,inner sep=0.3pt] (n6) at (5,0) {};
	\node[shape=circle,draw=black,minimum size = 6pt,inner sep=0.3pt] (n7) at (6,0) {};
  \node[shape=circle,draw=black,minimum size = 6pt,inner sep=0.3pt] (n8) at (7,0) {};

	  \draw[line width=0.6mm,blue]  (n1) -- (n2)  ;
		\draw[line width=0.6mm,blue]  (n2) -- (n3)  ;
		\draw[line width=0.6mm,blue]  (n3) -- (n4)  ;
		\draw[line width=0.6mm,blue]  (n4) -- (n5)  ;
		\draw[line width=0.6mm,blue]  (n5) -- (n6)  ;
		\draw[line width=0.6mm,blue]  (n6) -- (n7)  ;
		\draw[line width=0.6mm,blue]  (n7) -- (n8)  ;
		%\draw[line width=0.5mm,dashed]  (n2) -- (n3);	
\end{tikzpicture} \\

 \hline
 \;&\;\\
 $T_2$ & \begin{tikzpicture}
  [scale=0.9,auto=left]%,every node/.style={shape = circle, draw, fill = white,minimum size = 14pt, inner sep=0.3pt}]%baseline=(a.center)]%{circle,fill=black}]
	%\tikzset{VertexStyle/.style = {shape = circle,fill = black,minimum size = 9mm,inner sep=2pt}}
	\node[shape=circle,draw=black,minimum size = 6pt,inner sep=0.3pt] (n1) at (0,0) {};
	\node[shape=circle,draw=black,minimum size = 6pt,inner sep=0.3pt] (n2) at (1,0) {};
	\node[shape=circle,draw=black,minimum size = 6pt,inner sep=0.3pt] (n3) at (2,0) {};
  \node[shape=circle,draw=black,minimum size = 6pt,inner sep=0.3pt] (n4) at (3,0) {};
	\node[shape=circle,draw=black,minimum size = 6pt,inner sep=0.3pt] (n5) at (4,0) {};
  \node[shape=circle,draw=black,minimum size = 6pt,inner sep=0.3pt] (n6) at (5,0) {};
	\node[shape=circle,draw=black,minimum size = 6pt,inner sep=0.3pt] (n7) at (6,0) {};
  \node[shape=circle,draw=black,minimum size = 6pt,inner sep=0.3pt] (n8) at (5,1) {};

	  \draw[line width=0.6mm,blue]  (n1) -- (n2)  ;
		\draw[line width=0.6mm,blue]  (n2) -- (n3)  ;
		\draw[line width=0.6mm,blue]  (n3) -- (n4)  ;
		\draw[line width=0.6mm,blue]  (n4) -- (n5)  ;
		\draw[line width=0.6mm,blue]  (n5) -- (n6)  ;
		\draw[line width=0.6mm,blue]  (n6) -- (n7)  ;
		\draw[line width=0.6mm,blue]  (n6) -- (n8)  ;
		%\draw[line width=0.5mm,dashed]  (n2) -- (n3);	
\end{tikzpicture}  \\
 \hline
 \;&\;\\
$T_3$ & \begin{tikzpicture}
  [scale=0.9,auto=left]%,every node/.style={shape = circle, draw, fill = white,minimum size = 14pt, inner sep=0.3pt}]%baseline=(a.center)]%{circle,fill=black}]
	%\tikzset{VertexStyle/.style = {shape = circle,fill = black,minimum size = 9mm,inner sep=2pt}}
	\node[shape=circle,draw=black,minimum size = 6pt,inner sep=0.3pt] (n1) at (0,0) {};
	\node[shape=circle,draw=black,minimum size = 6pt,inner sep=0.3pt] (n2) at (1,0) {};
	\node[shape=circle,draw=black,minimum size = 6pt,inner sep=0.3pt] (n3) at (2,0) {};
  \node[shape=circle,draw=black,minimum size = 6pt,inner sep=0.3pt] (n4) at (3,0) {};
	\node[shape=circle,draw=black,minimum size = 6pt,inner sep=0.3pt] (n5) at (4,0) {};
  	\node[shape=circle,draw=black,minimum size = 6pt,inner sep=0.3pt] (n6) at (5,0) {};
    \node[shape=circle,draw=black,minimum size = 6pt,inner sep=0.3pt] (n7) at (6,0) {};
  \node[shape=circle,draw=black,minimum size = 6pt,inner sep=0.3pt] (n8) at (4,1) {};

	  \draw[line width=0.6mm,blue]  (n1) -- (n2)  ;
		\draw[line width=0.6mm,blue]  (n2) -- (n3)  ;
		\draw[line width=0.6mm,blue]  (n3) -- (n4)  ;
		\draw[line width=0.6mm,blue]  (n4) -- (n5)  ;
		\draw[line width=0.6mm,blue]  (n5) -- (n6)  ;
		\draw[line width=0.6mm,blue]  (n6) -- (n7)  ;
		\draw[line width=0.6mm,blue]  (n5) -- (n8)  ;
		%\draw[line width=0.5mm,dashed]  (n2) -- (n3);	
\end{tikzpicture} \\

 \hline
 \;&\;\\
$T_4$ & \begin{tikzpicture}
  [scale=0.9,auto=left]%,every node/.style={shape = circle, draw, fill = white,minimum size = 14pt, inner sep=0.3pt}]%baseline=(a.center)]%{circle,fill=black}]
	%\tikzset{VertexStyle/.style = {shape = circle,fill = black,minimum size = 9mm,inner sep=2pt}}
	%\node[shape=circle,draw=black,minimum size = 12pt,inner sep=0.3pt] (n1) at (0,0) {$1$};
	\node[shape=circle,draw=black,minimum size = 6pt,inner sep=0.3pt] (n1) at (1,0) {};
	\node[shape=circle,draw=black,minimum size = 6pt,inner sep=0.3pt] (n2) at (2,0) {};
  \node[shape=circle,draw=black,minimum size = 6pt,inner sep=0.3pt] (n3) at (3,0) {};
	\node[shape=circle,draw=black,minimum size = 6pt,inner sep=0.3pt] (n4) at (4,0) {};
  	\node[shape=circle,draw=black,minimum size = 6pt,inner sep=0.3pt] (n5) at (5,0) {};
    \node[shape=circle,draw=black,minimum size = 6pt,inner sep=0.3pt] (n6) at (6,0) {};
  \node[shape=circle,draw=black,minimum size = 6pt,inner sep=0.3pt] (n7) at (4,1) {};
\node[shape=circle,draw=black,minimum size = 6pt,inner sep=0.3pt] (n8) at (4,2) {};

	  \draw[line width=0.6mm,blue]  (n1) -- (n2)  ;
		\draw[line width=0.6mm,blue]  (n2) -- (n3)  ;
		\draw[line width=0.6mm,blue]  (n3) -- (n4)  ;
		\draw[line width=0.6mm,blue]  (n4) -- (n5)  ;
		\draw[line width=0.6mm,blue]  (n5) -- (n6)  ;
		\draw[line width=0.6mm,blue]  (n4) -- (n7)  ;
		\draw[line width=0.6mm,blue]  (n7) -- (n8)  ;
		%\draw[line width=0.5mm,dashed]  (n2) -- (n3);	
\end{tikzpicture} \\

 \hline
 \;&\;\\
 $T_5$ & \begin{tikzpicture}
  [scale=0.9,auto=left]%,every node/.style={shape = circle, draw, fill = white,minimum size = 14pt, inner sep=0.3pt}]%baseline=(a.center)]%{circle,fill=black}]
	%\tikzset{VertexStyle/.style = {shape = circle,fill = black,minimum size = 9mm,inner sep=2pt}}
	\node[shape=circle,draw=black,minimum size = 6pt,inner sep=0.3pt] (n1) at (0,0) {};
	\node[shape=circle,draw=black,minimum size = 6pt,inner sep=0.3pt] (n2) at (1,0) {};
	\node[shape=circle,draw=black,minimum size = 6pt,inner sep=0.3pt] (n3) at (2,0) {};
  \node[shape=circle,draw=black,minimum size = 6pt,inner sep=0.3pt] (n4) at (3,0) {};
	\node[shape=circle,draw=black,minimum size = 6pt,inner sep=0.3pt] (n5) at (4,0) {};
  \node[shape=circle,draw=black,minimum size = 6pt,inner sep=0.3pt] (n6) at (5,0) {};
	\node[shape=circle,draw=black,minimum size = 6pt,inner sep=0.3pt] (n7) at (6,0) {};
  \node[shape=circle,draw=black,minimum size = 6pt,inner sep=0.3pt] (n8) at (3,1) {};

	  \draw[line width=0.6mm,blue]  (n1) -- (n2)  ;
		\draw[line width=0.6mm,blue]  (n2) -- (n3)  ;
		\draw[line width=0.6mm,blue]  (n3) -- (n4)  ;
		\draw[line width=0.6mm,blue]  (n4) -- (n5)  ;
		\draw[line width=0.6mm,blue]  (n5) -- (n6)  ;
		\draw[line width=0.6mm,blue]  (n6) -- (n7)  ;
		\draw[line width=0.6mm,blue]  (n4) -- (n8)  ;
		%\draw[line width=0.5mm,dashed]  (n2) -- (n3);	
\end{tikzpicture} \\

 \hline
 \;&\;\\
 $T_6$ & \begin{tikzpicture}
  [scale=0.9,auto=left]%,every node/.style={shape = circle, draw, fill = white,minimum size = 14pt, inner sep=0.3pt}]%baseline=(a.center)]%{circle,fill=black}]
	%\tikzset{VertexStyle/.style = {shape = circle,fill = black,minimum size = 9mm,inner sep=2pt}}
	\node[shape=circle,draw=black,minimum size = 6pt,inner sep=0.3pt] (n1) at (0,0) {};
	\node[shape=circle,draw=black,minimum size = 6pt,inner sep=0.3pt] (n2) at (1,0) {};
	\node[shape=circle,draw=black,minimum size = 6pt,inner sep=0.3pt] (n3) at (2,0) {};
  \node[shape=circle,draw=black,minimum size = 6pt,inner sep=0.3pt] (n4) at (3,0) {};
	\node[shape=circle,draw=black,minimum size = 6pt,inner sep=0.3pt] (n5) at (4,0) {};
  \node[shape=circle,draw=black,minimum size = 6pt,inner sep=0.3pt] (n6) at (4,1) {};
	\node[shape=circle,draw=black,minimum size = 6pt,inner sep=0.3pt] (n8) at (4,-1) {};
  \node[shape=circle,draw=black,minimum size = 6pt,inner sep=0.3pt] (n7) at (5,0) {};

	  \draw[line width=0.6mm,blue]  (n1) -- (n2)  ;
		\draw[line width=0.6mm,blue]  (n2) -- (n3)  ;
		\draw[line width=0.6mm,blue]  (n3) -- (n4)  ;
		\draw[line width=0.6mm,blue]  (n4) -- (n5)  ;
		\draw[line width=0.6mm,blue]  (n5) -- (n6)  ;
		\draw[line width=0.6mm,blue]  (n5) -- (n7)  ;
		\draw[line width=0.6mm,blue]  (n5) -- (n8)  ;
		%\draw[line width=0.5mm,dashed]  (n2) -- (n3);	
\end{tikzpicture}  \\

 \hline
 \;&\;\\
 $T_7$ & \begin{tikzpicture}
  [scale=0.9,auto=left]%,every node/.style={shape = circle, draw, fill = white,minimum size = 14pt, inner sep=0.3pt}]%baseline=(a.center)]%{circle,fill=black}]
	%\tikzset{VertexStyle/.style = {shape = circle,fill = black,minimum size = 9mm,inner sep=2pt}}
	\node[shape=circle,draw=black,minimum size = 6pt,inner sep=0.3pt] (n1) at (0,0) {};
	\node[shape=circle,draw=black,minimum size = 6pt,inner sep=0.3pt] (n2) at (1,0) {};
	\node[shape=circle,draw=black,minimum size = 6pt,inner sep=0.3pt] (n3) at (2,0) {};
  \node[shape=circle,draw=black,minimum size = 6pt,inner sep=0.3pt] (n4) at (3,0) {};
	\node[shape=circle,draw=black,minimum size = 6pt,inner sep=0.3pt] (n5) at (4,0) {};
  \node[shape=circle,draw=black,minimum size = 6pt,inner sep=0.3pt] (n6) at (3,1) {};
	\node[shape=circle,draw=black,minimum size = 6pt,inner sep=0.3pt] (n7) at (3,-1) {};
  \node[shape=circle,draw=black,minimum size = 6pt,inner sep=0.3pt] (n8) at (5,0) {};

	  \draw[line width=0.6mm,blue]  (n1) -- (n2)  ;
		\draw[line width=0.6mm,blue]  (n2) -- (n3)  ;
		\draw[line width=0.6mm,blue]  (n3) -- (n4)  ;
		\draw[line width=0.6mm,blue]  (n4) -- (n5)  ;
		\draw[line width=0.6mm,blue]  (n4) -- (n6)  ;
		\draw[line width=0.6mm,blue]  (n4) -- (n7)  ;
		\draw[line width=0.6mm,blue]  (n5) -- (n8)  ;
		%\draw[line width=0.5mm,dashed]  (n2) -- (n3);	
\end{tikzpicture}  \\

 \hline
 \;&\;\\
 $T_8$ &\begin{tikzpicture}
  [scale=0.9,auto=left]%,every node/.style={shape = circle, draw, fill = white,minimum size = 14pt, inner sep=0.3pt}]%baseline=(a.center)]%{circle,fill=black}]
	%\tikzset{VertexStyle/.style = {shape = circle,fill = black,minimum size = 9mm,inner sep=2pt}}
	\node[shape=circle,draw=black,minimum size = 6pt,inner sep=0.3pt] (n1) at (0,0) {};
	\node[shape=circle,draw=black,minimum size = 6pt,inner sep=0.3pt] (n2) at (1,0) {};
	\node[shape=circle,draw=black,minimum size = 6pt,inner sep=0.3pt] (n3) at (2,0) {};
  \node[shape=circle,draw=black,minimum size = 6pt,inner sep=0.3pt] (n4) at (3,0) {};
	\node[shape=circle,draw=black,minimum size = 6pt,inner sep=0.3pt] (n5) at (4,0) {};
  \node[shape=circle,draw=black,minimum size = 6pt,inner sep=0.3pt] (n6) at (2,1) {};
	\node[shape=circle,draw=black,minimum size = 6pt,inner sep=0.3pt] (n7) at (2,2) {};
  \node[shape=circle,draw=black,minimum size = 6pt,inner sep=0.3pt] (n8) at (2,-1) {};

	  \draw[line width=0.6mm,blue]  (n1) -- (n2)  ;
		\draw[line width=0.6mm,blue]  (n2) -- (n3)  ;
		\draw[line width=0.6mm,blue]  (n3) -- (n4)  ;
		\draw[line width=0.6mm,blue]  (n4) -- (n5)  ;
		\draw[line width=0.6mm,blue]  (n3) -- (n6)  ;
		\draw[line width=0.6mm,blue]  (n6) -- (n7)  ;
		\draw[line width=0.6mm,blue]  (n3) -- (n8)  ;
		%\draw[line width=0.5mm,dashed]  (n2) -- (n3);	
\end{tikzpicture}  \\

 %\end{longtable}
%\end{tabular}
%\caption{Table of trees with $8$ vertices (Part 1)\label{tablepart1}}
%\end{figure}

%\begin{figure}[H]
%\begin{tabular}{||c|c||}
 \hline
 \;&\;\\
 $T_9$ & \begin{tikzpicture}
  [scale=0.9,auto=left]%,every node/.style={shape = circle, draw, fill = white,minimum size = 14pt, inner sep=0.3pt}]%baseline=(a.center)]%{circle,fill=black}]
	%\tikzset{VertexStyle/.style = {shape = circle,fill = black,minimum size = 9mm,inner sep=2pt}}
	%\node[shape=circle,draw=black,minimum size = 12pt,inner sep=0.3pt] (n1) at (0,0) {$1$};
	\node[shape=circle,draw=black,minimum size = 6pt,inner sep=0.3pt] (n1) at (1,0) {};
	\node[shape=circle,draw=black,minimum size = 6pt,inner sep=0.3pt] (n2) at (2,0) {};
  \node[shape=circle,draw=black,minimum size = 6pt,inner sep=0.3pt] (n3) at (3,0) {};
	\node[shape=circle,draw=black,minimum size = 6pt,inner sep=0.3pt] (n4) at (4,0) {};
  	\node[shape=circle,draw=black,minimum size = 6pt,inner sep=0.3pt] (n5) at (3.7,0.95) {};
    \node[shape=circle,draw=black,minimum size = 6pt,inner sep=0.3pt] (n6) at (4.9,0.6) {};
  \node[shape=circle,draw=black,minimum size = 6pt,inner sep=0.3pt] (n7) at (4.9,-0.6) {};
\node[shape=circle,draw=black,minimum size = 6pt,inner sep=0.3pt] (n8) at (3.7,-0.95) {};

	  \draw[line width=0.6mm,blue]  (n1) -- (n2)  ;
		\draw[line width=0.6mm,blue]  (n2) -- (n3)  ;
		\draw[line width=0.6mm,blue]  (n3) -- (n4)  ;
		\draw[line width=0.6mm,blue]  (n4) -- (n5)  ;
		\draw[line width=0.6mm,blue]  (n4) -- (n6)  ;
		\draw[line width=0.6mm,blue]  (n4) -- (n7)  ;
		\draw[line width=0.6mm,blue]  (n4) -- (n8)  ;
		%\draw[line width=0.5mm,dashed]  (n2) -- (n3);	
\end{tikzpicture} \\

 \hline
 \;&\;\\
 $T_{10}$ & \begin{tikzpicture}
  [scale=0.9,auto=left]%,every node/.style={shape = circle, draw, fill = white,minimum size = 14pt, inner sep=0.3pt}]%baseline=(a.center)]%{circle,fill=black}]
	%\tikzset{VertexStyle/.style = {shape = circle,fill = black,minimum size = 9mm,inner sep=2pt}}
	%\node[shape=circle,draw=black,minimum size = 12pt,inner sep=0.3pt] (n1) at (0,0) {$1$};
	\node[shape=circle,draw=black,minimum size = 6pt,inner sep=0.3pt] (n1) at (1,0) {};
	\node[shape=circle,draw=black,minimum size = 6pt,inner sep=0.3pt] (n2) at (2,0) {};
  \node[shape=circle,draw=black,minimum size = 6pt,inner sep=0.3pt] (n3) at (3,0) {};
	\node[shape=circle,draw=black,minimum size = 6pt,inner sep=0.3pt] (n4) at (4,0) {};
  	\node[shape=circle,draw=black,minimum size = 6pt,inner sep=0.3pt] (n5) at (5,0) {};
    \node[shape=circle,draw=black,minimum size = 6pt,inner sep=0.3pt] (n6) at (2.5,0.88) {};
  \node[shape=circle,draw=black,minimum size = 6pt,inner sep=0.3pt] (n7) at (3.5,0.88) {};
\node[shape=circle,draw=black,minimum size = 6pt,inner sep=0.3pt] (n8) at (3,-1) {};

	  \draw[line width=0.6mm,blue]  (n1) -- (n2)  ;
		\draw[line width=0.6mm,blue]  (n2) -- (n3)  ;
		\draw[line width=0.6mm,blue]  (n3) -- (n4)  ;
		\draw[line width=0.6mm,blue]  (n4) -- (n5)  ;
		\draw[line width=0.6mm,blue]  (n3) -- (n6)  ;
		\draw[line width=0.6mm,blue]  (n3) -- (n7)  ;
		\draw[line width=0.6mm,blue]  (n3) -- (n8)  ;
		%\draw[line width=0.5mm,dashed]  (n2) -- (n3);	
\end{tikzpicture}  \\
 \hline
 \;&\;\\
 $T_{11}$ & \begin{tikzpicture}
  [scale=0.9,auto=left]%,every node/.style={shape = circle, draw, fill = white,minimum size = 14pt, inner sep=0.3pt}]%baseline=(a.center)]%{circle,fill=black}]
	%\tikzset{VertexStyle/.style = {shape = circle,fill = black,minimum size = 9mm,inner sep=2pt}}
	%\node[shape=circle,draw=black,minimum size = 12pt,inner sep=0.3pt] (n1) at (0,0) {$1$};
	\node[shape=circle,draw=black,minimum size = 6pt,inner sep=0.3pt] (n1) at (1,0) {};
	\node[shape=circle,draw=black,minimum size = 6pt,inner sep=0.3pt] (n2) at (2,0) {};
  \node[shape=circle,draw=black,minimum size = 6pt,inner sep=0.3pt] (n3) at (3,0) {};
  	\node[shape=circle,draw=black,minimum size = 6pt,inner sep=0.3pt] (n4) at (2.5,0.88) {};
    \node[shape=circle,draw=black,minimum size = 6pt,inner sep=0.3pt] (n5) at (3.5,0.88) {};
			\node[shape=circle,draw=black,minimum size = 6pt,inner sep=0.3pt] (n6) at (4,0) {};
  \node[shape=circle,draw=black,minimum size = 6pt,inner sep=0.3pt] (n7) at (3.5,-0.88) {};
\node[shape=circle,draw=black,minimum size = 6pt,inner sep=0.3pt] (n8) at (2.5,-0.88) {};

	  \draw[line width=0.6mm,blue]  (n1) -- (n2)  ;
		\draw[line width=0.6mm,blue]  (n2) -- (n3)  ;
		\draw[line width=0.6mm,blue]  (n3) -- (n4)  ;
		\draw[line width=0.6mm,blue]  (n3) -- (n5)  ;
		\draw[line width=0.6mm,blue]  (n3) -- (n6)  ;
		\draw[line width=0.6mm,blue]  (n3) -- (n7)  ;
		\draw[line width=0.6mm,blue]  (n3) -- (n8)  ;
		%\draw[line width=0.5mm,dashed]  (n2) -- (n3);	
\end{tikzpicture}  \\
 \hline
 \;&\;\\
 $T_{12}$ & \begin{tikzpicture}
  [scale=0.9,auto=left]%,every node/.style={shape = circle, draw, fill = white,minimum size = 14pt, inner sep=0.3pt}]%baseline=(a.center)]%{circle,fill=black}]
	%\tikzset{VertexStyle/.style = {shape = circle,fill = black,minimum size = 9mm,inner sep=2pt}}
	%\node[shape=circle,draw=black,minimum size = 12pt,inner sep=0.3pt] (n1) at (0,0) {$1$};
	\node[shape=circle,draw=black,minimum size = 6pt,inner sep=0.3pt] (n1) at (1,0) {};
	\node[shape=circle,draw=black,minimum size = 6pt,inner sep=0.3pt] (n2) at (0,0) {};
  \node[shape=circle,draw=black,minimum size = 6pt,inner sep=0.3pt] (n3) at (0.36,0.78) {};
  	\node[shape=circle,draw=black,minimum size = 6pt,inner sep=0.3pt] (n4) at (1.22,0.95) {};
    \node[shape=circle,draw=black,minimum size = 6pt,inner sep=0.3pt] (n5) at (1.9,0.43) {};
			\node[shape=circle,draw=black,minimum size = 6pt,inner sep=0.3pt] (n6) at (1.9,-0.43) {};
  \node[shape=circle,draw=black,minimum size = 6pt,inner sep=0.3pt] (n7) at (1.22,-0.95) {};
\node[shape=circle,draw=black,minimum size = 6pt,inner sep=0.3pt] (n8) at (0.36,-0.78) {};

	  \draw[line width=0.6mm,blue]  (n1) -- (n2)  ;
		\draw[line width=0.6mm,blue]  (n1) -- (n3)  ;
		\draw[line width=0.6mm,blue]  (n1) -- (n4)  ;
		\draw[line width=0.6mm,blue]  (n1) -- (n5)  ;
		\draw[line width=0.6mm,blue]  (n1) -- (n6)  ;
		\draw[line width=0.6mm,blue]  (n1) -- (n7)  ;
		\draw[line width=0.6mm,blue]  (n1) -- (n8)  ;
		%\draw[line width=0.5mm,dashed]  (n2) -- (n3);	
\end{tikzpicture}  \\

 \hline
 \;&\;\\
 $T_{13}$ & \begin{tikzpicture}
  [scale=0.9,auto=left]%,every node/.style={shape = circle, draw, fill = white,minimum size = 14pt, inner sep=0.3pt}]%baseline=(a.center)]%{circle,fill=black}]
	%\tikzset{VertexStyle/.style = {shape = circle,fill = black,minimum size = 9mm,inner sep=2pt}}
	\node[shape=circle,draw=black,minimum size = 6pt,inner sep=0.3pt] (n1) at (0,0) {};
	\node[shape=circle,draw=black,minimum size = 6pt,inner sep=0.3pt] (n2) at (1,0) {};
	\node[shape=circle,draw=black,minimum size = 6pt,inner sep=0.3pt] (n3) at (2,0) {};
  \node[shape=circle,draw=black,minimum size = 6pt,inner sep=0.3pt] (n4) at (3,0) {};
	\node[shape=circle,draw=black,minimum size = 6pt,inner sep=0.3pt] (n5) at (4,0) {};
  \node[shape=circle,draw=black,minimum size = 6pt,inner sep=0.3pt] (n6) at (5,0) {};
	\node[shape=circle,draw=black,minimum size = 6pt,inner sep=0.3pt] (n7) at (1,1) {};
  \node[shape=circle,draw=black,minimum size = 6pt,inner sep=0.3pt] (n8) at (4,1) {};

	  \draw[line width=0.6mm,blue]  (n1) -- (n2)  ;
		\draw[line width=0.6mm,blue]  (n2) -- (n3)  ;
		\draw[line width=0.6mm,blue]  (n3) -- (n4)  ;
		\draw[line width=0.6mm,blue]  (n4) -- (n5)  ;
		\draw[line width=0.6mm,blue]  (n5) -- (n6)  ;
		\draw[line width=0.6mm,blue]  (n2) -- (n7)  ;
		\draw[line width=0.6mm,blue]  (n5) -- (n8)  ;
		%\draw[line width=0.5mm,dashed]  (n2) -- (n3);	
\end{tikzpicture}  \\
 \hline
 \;&\;\\
 $T_{14}$ & \begin{tikzpicture}
  [scale=0.9,auto=left]%,every node/.style={shape = circle, draw, fill = white,minimum size = 14pt, inner sep=0.3pt}]%baseline=(a.center)]%{circle,fill=black}]
	%\tikzset{VertexStyle/.style = {shape = circle,fill = black,minimum size = 9mm,inner sep=2pt}}
	\node[shape=circle,draw=black,minimum size = 6pt,inner sep=0.3pt] (n1) at (0,0) {};
	\node[shape=circle,draw=black,minimum size = 6pt,inner sep=0.3pt] (n2) at (1,1) {};
	\node[shape=circle,draw=black,minimum size = 6pt,inner sep=0.3pt] (n3) at (1,0) {};
  \node[shape=circle,draw=black,minimum size = 6pt,inner sep=0.3pt] (n4) at (2,0) {};
	\node[shape=circle,draw=black,minimum size = 6pt,inner sep=0.3pt] (n5) at (3,1) {};
  \node[shape=circle,draw=black,minimum size = 6pt,inner sep=0.3pt] (n6) at (3,0) {};
	\node[shape=circle,draw=black,minimum size = 6pt,inner sep=0.3pt] (n7) at (4,0) {};
  \node[shape=circle,draw=black,minimum size = 6pt,inner sep=0.3pt] (n8) at (5,0) {};

	  \draw[line width=0.6mm,blue]  (n1) -- (n3)  ;
		\draw[line width=0.6mm,blue]  (n2) -- (n3)  ;
		\draw[line width=0.6mm,blue]  (n3) -- (n4)  ;
		\draw[line width=0.6mm,blue]  (n4) -- (n6)  ;
		\draw[line width=0.6mm,blue]  (n5) -- (n6)  ;
		\draw[line width=0.6mm,blue]  (n6) -- (n7)  ;
		\draw[line width=0.6mm,blue]  (n7) -- (n8)  ;
		%\draw[line width=0.5mm,dashed]  (n2) -- (n3);	
\end{tikzpicture}  \\
 \hline
 \;&\;\\
 $T_{15}$ & \begin{tikzpicture}
  [scale=0.9,auto=left]%,every node/.style={shape = circle, draw, fill = white,minimum size = 14pt, inner sep=0.3pt}]%baseline=(a.center)]%{circle,fill=black}]
	%\tikzset{VertexStyle/.style = {shape = circle,fill = black,minimum size = 9mm,inner sep=2pt}}
	\node[shape=circle,draw=black,minimum size = 6pt,inner sep=0.3pt] (n1) at (0,0) {};
	\node[shape=circle,draw=black,minimum size = 6pt,inner sep=0.3pt] (n2) at (1,1) {};
	\node[shape=circle,draw=black,minimum size = 6pt,inner sep=0.3pt] (n3) at (1,0) {};
  \node[shape=circle,draw=black,minimum size = 6pt,inner sep=0.3pt] (n4) at (2,1) {};
	\node[shape=circle,draw=black,minimum size = 6pt,inner sep=0.3pt] (n5) at (2,0) {};
  \node[shape=circle,draw=black,minimum size = 6pt,inner sep=0.3pt] (n6) at (3,1) {};
	\node[shape=circle,draw=black,minimum size = 6pt,inner sep=0.3pt] (n7) at (3,0) {};
  \node[shape=circle,draw=black,minimum size = 6pt,inner sep=0.3pt] (n8) at (4,0) {};

	  \draw[line width=0.6mm,blue]  (n1) -- (n3)  ;
		\draw[line width=0.6mm,blue]  (n3) -- (n5)  ;
		\draw[line width=0.6mm,blue]  (n5) -- (n7)  ;
		\draw[line width=0.6mm,blue]  (n7) -- (n8)  ;
		\draw[line width=0.6mm,blue]  (n4) -- (n5)  ;
		\draw[line width=0.6mm,blue]  (n2) -- (n4)  ;
		\draw[line width=0.6mm,blue]  (n4) -- (n6)  ;
		%\draw[line width=0.5mm,dashed]  (n2) -- (n3);	
\end{tikzpicture}  \\
 \hline
 \;&\;\\
 $T_{16}$ & \begin{tikzpicture}
  [scale=0.9,auto=left]%,every node/.style={shape = circle, draw, fill = white,minimum size = 14pt, inner sep=0.3pt}]%baseline=(a.center)]%{circle,fill=black}]
	%\tikzset{VertexStyle/.style = {shape = circle,fill = black,minimum size = 9mm,inner sep=2pt}}
	\node[shape=circle,draw=black,minimum size = 6pt,inner sep=0.3pt] (n1) at (0,0) {};
	\node[shape=circle,draw=black,minimum size = 6pt,inner sep=0.3pt] (n2) at (1,0) {};
	\node[shape=circle,draw=black,minimum size = 6pt,inner sep=0.3pt] (n3) at (2,1) {};
  \node[shape=circle,draw=black,minimum size = 6pt,inner sep=0.3pt] (n4) at (2,0) {};
	\node[shape=circle,draw=black,minimum size = 6pt,inner sep=0.3pt] (n5) at (3,1) {};
  \node[shape=circle,draw=black,minimum size = 6pt,inner sep=0.3pt] (n6) at (3,0) {};
	\node[shape=circle,draw=black,minimum size = 6pt,inner sep=0.3pt] (n7) at (4,0) {};
  \node[shape=circle,draw=black,minimum size = 6pt,inner sep=0.3pt] (n8) at (5,0) {};

	  \draw[line width=0.6mm,blue]  (n1) -- (n2)  ;
		\draw[line width=0.6mm,blue]  (n2) -- (n4)  ;
		\draw[line width=0.6mm,blue]  (n3) -- (n4)  ;
		\draw[line width=0.6mm,blue]  (n4) -- (n6)  ;
		\draw[line width=0.6mm,blue]  (n5) -- (n6)  ;
		\draw[line width=0.6mm,blue]  (n6) -- (n7)  ;
		\draw[line width=0.6mm,blue]  (n7) -- (n8)  ;
		%\draw[line width=0.5mm,dashed]  (n2) -- (n3);	
\end{tikzpicture}  \\

 %\hline
%\end{longtable}

%\end{tabular}
%\caption{Table of trees with $8$ vertices (Part 2)\label{tablepart2}}
%\end{figure}

%\begin{figure}[H]
%\begin{tabular}{||c|c||}
 \hline
 \;&\;\\
 $T_{17}$ & \begin{tikzpicture}
  [scale=0.9,auto=left]%,every node/.style={shape = circle, draw, fill = white,minimum size = 14pt, inner sep=0.3pt}]%baseline=(a.center)]%{circle,fill=black}]
	%\tikzset{VertexStyle/.style = {shape = circle,fill = black,minimum size = 9mm,inner sep=2pt}}
	\node[shape=circle,draw=black,minimum size = 6pt,inner sep=0.3pt] (n1) at (0,0) {};
	\node[shape=circle,draw=black,minimum size = 6pt,inner sep=0.3pt] (n2) at (1,1) {};
	\node[shape=circle,draw=black,minimum size = 6pt,inner sep=0.3pt] (n3) at (1,0) {};
  \node[shape=circle,draw=black,minimum size = 6pt,inner sep=0.3pt] (n4) at (2,1) {};
	\node[shape=circle,draw=black,minimum size = 6pt,inner sep=0.3pt] (n5) at (2,0) {};
  \node[shape=circle,draw=black,minimum size = 6pt,inner sep=0.3pt] (n6) at (3,0) {};
	\node[shape=circle,draw=black,minimum size = 6pt,inner sep=0.3pt] (n7) at (4,0) {};
  \node[shape=circle,draw=black,minimum size = 6pt,inner sep=0.3pt] (n8) at (5,0) {};

	  \draw[line width=0.6mm,blue]  (n1) -- (n3)  ;
		\draw[line width=0.6mm,blue]  (n2) -- (n3)  ;
		\draw[line width=0.6mm,blue]  (n3) -- (n5)  ;
		\draw[line width=0.6mm,blue]  (n4) -- (n5)  ;
		\draw[line width=0.6mm,blue]  (n5) -- (n6)  ;
		\draw[line width=0.6mm,blue]  (n6) -- (n7)  ;
		\draw[line width=0.6mm,blue]  (n7) -- (n8)  ;
		%\draw[line width=0.5mm,dashed]  (n2) -- (n3);	
\end{tikzpicture}  \\
 \hline
 \;&\;\\
 $T_{18}$ & \begin{tikzpicture}
  [scale=0.9,auto=left]%,every node/.style={shape = circle, draw, fill = white,minimum size = 14pt, inner sep=0.3pt}]%baseline=(a.center)]%{circle,fill=black}]
	%\tikzset{VertexStyle/.style = {shape = circle,fill = black,minimum size = 9mm,inner sep=2pt}}
	\node[shape=circle,draw=black,minimum size = 6pt,inner sep=0.3pt] (n1) at (0,0) {};
	\node[shape=circle,draw=black,minimum size = 6pt,inner sep=0.3pt] (n2) at (1,1) {};
	\node[shape=circle,draw=black,minimum size = 6pt,inner sep=0.3pt] (n3) at (1,0) {};
  \node[shape=circle,draw=black,minimum size = 6pt,inner sep=0.3pt] (n4) at (2,0) {};
	\node[shape=circle,draw=black,minimum size = 6pt,inner sep=0.3pt] (n5) at (3,1) {};
  \node[shape=circle,draw=black,minimum size = 6pt,inner sep=0.3pt] (n6) at (3,-1) {};
	\node[shape=circle,draw=black,minimum size = 6pt,inner sep=0.3pt] (n7) at (3,0) {};
  \node[shape=circle,draw=black,minimum size = 6pt,inner sep=0.3pt] (n8) at (4,0) {};

	  \draw[line width=0.6mm,blue]  (n1) -- (n3)  ;
		\draw[line width=0.6mm,blue]  (n2) -- (n3)  ;
		\draw[line width=0.6mm,blue]  (n3) -- (n4)  ;
		\draw[line width=0.6mm,blue]  (n7) -- (n5)  ;
		\draw[line width=0.6mm,blue]  (n7) -- (n6)  ;
		\draw[line width=0.6mm,blue]  (n4) -- (n7)  ;
		\draw[line width=0.6mm,blue]  (n7) -- (n8)  ;
		%\draw[line width=0.5mm,dashed]  (n2) -- (n3);	
\end{tikzpicture}  \\
 \hline
 \;&\;\\
 $T_{19}$ & \begin{tikzpicture}
  [scale=0.9,auto=left]%,every node/.style={shape = circle, draw, fill = white,minimum size = 14pt, inner sep=0.3pt}]%baseline=(a.center)]%{circle,fill=black}]
	%\tikzset{VertexStyle/.style = {shape = circle,fill = black,minimum size = 9mm,inner sep=2pt}}
	\node[shape=circle,draw=black,minimum size = 6pt,inner sep=0.3pt] (n1) at (0,0) {};
	\node[shape=circle,draw=black,minimum size = 6pt,inner sep=0.3pt] (n2) at (1,1) {};
	\node[shape=circle,draw=black,minimum size = 6pt,inner sep=0.3pt] (n3) at (1,0) {};
  \node[shape=circle,draw=black,minimum size = 6pt,inner sep=0.3pt] (n4) at (2,0) {};
	\node[shape=circle,draw=black,minimum size = 6pt,inner sep=0.3pt] (n5) at (2,1) {};
  \node[shape=circle,draw=black,minimum size = 6pt,inner sep=0.3pt] (n6) at (2,-1) {};
	\node[shape=circle,draw=black,minimum size = 6pt,inner sep=0.3pt] (n7) at (3,0) {};
  \node[shape=circle,draw=black,minimum size = 6pt,inner sep=0.3pt] (n8) at (4,0) {};

	  \draw[line width=0.6mm,blue]  (n1) -- (n3)  ;
		\draw[line width=0.6mm,blue]  (n2) -- (n3)  ;
		\draw[line width=0.6mm,blue]  (n3) -- (n4)  ;
		\draw[line width=0.6mm,blue]  (n4) -- (n5)  ;
		\draw[line width=0.6mm,blue]  (n4) -- (n6)  ;
		\draw[line width=0.6mm,blue]  (n4) -- (n7)  ;
		\draw[line width=0.6mm,blue]  (n7) -- (n8)  ;
		%\draw[line width=0.5mm,dashed]  (n2) -- (n3);	
\end{tikzpicture}  \\
 \hline
 \;&\;\\
 $T_{20}$ & \begin{tikzpicture}
  [scale=0.9,auto=left]%,every node/.style={shape = circle, draw, fill = white,minimum size = 14pt, inner sep=0.3pt}]%baseline=(a.center)]%{circle,fill=black}]
	%\tikzset{VertexStyle/.style = {shape = circle,fill = black,minimum size = 9mm,inner sep=2pt}}
	\node[shape=circle,draw=black,minimum size = 6pt,inner sep=0.3pt] (n1) at (0,0) {};
	\node[shape=circle,draw=black,minimum size = 6pt,inner sep=0.3pt] (n2) at (1,0) {};
	\node[shape=circle,draw=black,minimum size = 6pt,inner sep=0.3pt] (n3) at (2,0) {};
  \node[shape=circle,draw=black,minimum size = 6pt,inner sep=0.3pt] (n4) at (2,1) {};
	\node[shape=circle,draw=black,minimum size = 6pt,inner sep=0.3pt] (n5) at (3,0) {};
  \node[shape=circle,draw=black,minimum size = 6pt,inner sep=0.3pt] (n6) at (3,1) {};
	\node[shape=circle,draw=black,minimum size = 6pt,inner sep=0.3pt] (n7) at (3,-1) {};
  \node[shape=circle,draw=black,minimum size = 6pt,inner sep=0.3pt] (n8) at (4,0) {};

	  \draw[line width=0.6mm,blue]  (n1) -- (n2)  ;
		\draw[line width=0.6mm,blue]  (n2) -- (n3)  ;
		\draw[line width=0.6mm,blue]  (n3) -- (n4)  ;
		\draw[line width=0.6mm,blue]  (n3) -- (n5)  ;
		\draw[line width=0.6mm,blue]  (n5) -- (n6)  ;
		\draw[line width=0.6mm,blue]  (n5) -- (n7)  ;
		\draw[line width=0.6mm,blue]  (n5) -- (n8)  ;
		%\draw[line width=0.5mm,dashed]  (n2) -- (n3);	
\end{tikzpicture}  \\
 \hline
 \;&\;\\
 $T_{21}$ & \begin{tikzpicture}
  [scale=0.9,auto=left]%,every node/.style={shape = circle, draw, fill = white,minimum size = 14pt, inner sep=0.3pt}]%baseline=(a.center)]%{circle,fill=black}]
	%\tikzset{VertexStyle/.style = {shape = circle,fill = black,minimum size = 9mm,inner sep=2pt}}

	\node[shape=circle,draw=black,minimum size = 6pt,inner sep=0.3pt] (n1) at (1,0) {};
	\node[shape=circle,draw=black,minimum size = 6pt,inner sep=0.3pt] (n3) at (2,0) {};
  \node[shape=circle,draw=black,minimum size = 6pt,inner sep=0.3pt] (n2) at (2,1) {};
	\node[shape=circle,draw=black,minimum size = 6pt,inner sep=0.3pt] (n4) at (3,1) {};
  \node[shape=circle,draw=black,minimum size = 6pt,inner sep=0.3pt] (n5) at (3,0) {};
	\node[shape=circle,draw=black,minimum size = 6pt,inner sep=0.3pt] (n6) at (3,-1) {};
  \node[shape=circle,draw=black,minimum size = 6pt,inner sep=0.3pt] (n7) at (3.87,0.5) {};
	\node[shape=circle,draw=black,minimum size = 6pt,inner sep=0.3pt] (n8) at (3.87,-0.5) {};
	  \draw[line width=0.6mm,blue]  (n1) -- (n3)  ;
		\draw[line width=0.6mm,blue]  (n2) -- (n3)  ;
		\draw[line width=0.6mm,blue]  (n3) -- (n5)  ;
		\draw[line width=0.6mm,blue]  (n4) -- (n5)  ;
		\draw[line width=0.6mm,blue]  (n5) -- (n6)  ;
		\draw[line width=0.6mm,blue]  (n5) -- (n7)  ;
		\draw[line width=0.6mm,blue]  (n5) -- (n8)  ;
		%\draw[line width=0.5mm,dashed]  (n2) -- (n3);	
\end{tikzpicture}  \\
 \hline
 \;&\;\\
 $T_{22}$ & \begin{tikzpicture}
  [scale=0.9,auto=left]%,every node/.style={shape = circle, draw, fill = white,minimum size = 14pt, inner sep=0.3pt}]%baseline=(a.center)]%{circle,fill=black}]
	%\tikzset{VertexStyle/.style = {shape = circle,fill = black,minimum size = 9mm,inner sep=2pt}}

	\node[shape=circle,draw=black,minimum size = 6pt,inner sep=0.3pt] (n1) at (1,0) {};
	\node[shape=circle,draw=black,minimum size = 6pt,inner sep=0.3pt] (n3) at (2,0) {};
  \node[shape=circle,draw=black,minimum size = 6pt,inner sep=0.3pt] (n2) at (2,1) {};
	\node[shape=circle,draw=black,minimum size = 6pt,inner sep=0.3pt] (n4) at (2,-1) {};
  \node[shape=circle,draw=black,minimum size = 6pt,inner sep=0.3pt] (n5) at (3,1) {};
	\node[shape=circle,draw=black,minimum size = 6pt,inner sep=0.3pt] (n6) at (3,0) {};
  \node[shape=circle,draw=black,minimum size = 6pt,inner sep=0.3pt] (n7) at (3,-1) {};
	\node[shape=circle,draw=black,minimum size = 6pt,inner sep=0.3pt] (n8) at (4,0) {};
	  \draw[line width=0.6mm,blue]  (n1) -- (n3)  ;
		\draw[line width=0.6mm,blue]  (n2) -- (n3)  ;
		\draw[line width=0.6mm,blue]  (n3) -- (n4)  ;
		\draw[line width=0.6mm,blue]  (n3) -- (n6)  ;
		\draw[line width=0.6mm,blue]  (n5) -- (n6)  ;
		\draw[line width=0.6mm,blue]  (n6) -- (n7)  ;
		\draw[line width=0.6mm,blue]  (n6) -- (n8)  ;
		%\draw[line width=0.5mm,dashed]  (n2) -- (n3);	
\end{tikzpicture}\\
 \hline
 \;&\;\\
 $T_{23}$ &\begin{tikzpicture}
  [scale=0.9,auto=left]%,every node/.style={shape = circle, draw, fill = white,minimum size = 14pt, inner sep=0.3pt}]%baseline=(a.center)]%{circle,fill=black}]
	%\tikzset{VertexStyle/.style = {shape = circle,fill = black,minimum size = 9mm,inner sep=2pt}}
	\node[shape=circle,draw=black,minimum size = 6pt,inner sep=0.3pt] (n1) at (0,0) {};
	\node[shape=circle,draw=black,minimum size = 6pt,inner sep=0.3pt] (n2) at (1,0) {};
	\node[shape=circle,draw=black,minimum size = 6pt,inner sep=0.3pt] (n3) at (1,1) {};
  \node[shape=circle,draw=black,minimum size = 6pt,inner sep=0.3pt] (n4) at (2,0) {};
	\node[shape=circle,draw=black,minimum size = 6pt,inner sep=0.3pt] (n5) at (2,1) {};
  \node[shape=circle,draw=black,minimum size = 6pt,inner sep=0.3pt] (n6) at (3,0) {};
	\node[shape=circle,draw=black,minimum size = 6pt,inner sep=0.3pt] (n7) at (3,1) {};
  \node[shape=circle,draw=black,minimum size = 6pt,inner sep=0.3pt] (n8) at (4,0) {};

	  \draw[line width=0.6mm,blue]  (n1) -- (n2)  ;
		\draw[line width=0.6mm,blue]  (n2) -- (n3)  ;
		\draw[line width=0.6mm,blue]  (n2) -- (n4)  ;
		\draw[line width=0.6mm,blue]  (n4) -- (n5)  ;
		\draw[line width=0.6mm,blue]  (n4) -- (n6)  ;
		\draw[line width=0.6mm,blue]  (n6) -- (n7)  ;
		\draw[line width=0.6mm,blue]  (n6) -- (n8)  ;
		%\draw[line width=0.5mm,dashed]  (n2) -- (n3);	
\end{tikzpicture} \\
 \hline
\end{longtable}
%\end{tabular}
%\caption{Table of trees with $8$ vertices (Part 3)\label{tablepart3}}
%\end{figure}

\section*{Statements and Declarations}
The authors have no relevant financial or non-financial interests to disclose.

The formulation of Lemma \ref{lem2} emerged from an interaction with ChatGPT 5.5. The argument in the proof is due to the authors. The authors assume full  responsibility  for all mathematical content and results presented in this paper.

\section*{Acknowledgment}
This paper started as an undergraduate research project. Matthew Fyfe would like to acknowledge partial support from a BGSU Center for Undergraduate Research and Scholarship grant.

The initial version of this paper, which had four authors, was posted on the mathematics arXiv in July 2024 (see \cite{flss-v1}). In that first version,  MATLAB was used for the reduction from $\Gamma_4=T_{19}$ to $\Delta_4=TS_{4}$. The current version includes a fifth co-author, and  gives a direct, combinatorial proof for the reduction from $\Gamma_4=T_{19}$ to $\Delta_4=TS_{4}$. %For more details see Subsection \ref{subsection42}.
%We thank ...

\bibliographystyle{amsalpha}

\end{document}